\documentclass[11pt,letterpaper]{amsart}
\usepackage{stmaryrd}
\usepackage{amsfonts}
\usepackage{esint}

\usepackage{amsmath}
\usepackage{amssymb}
\usepackage{latexsym}
\usepackage{amscd}
\usepackage{mathrsfs}

\usepackage{hyperref}
\usepackage{graphicx} 
\usepackage{amsthm}
\usepackage{xypic}
\usepackage{bm}
\usepackage[all]{xy}
\usepackage{color}

\newdimen\AAdi%
\newbox\AAbo%
\def\AAk#1#2{\s_etbox\AAbo=\hbox{#2}\AAdi=\wd\AAbo\kern#1\AAdi{}}%
\def\AAr#1#2#3{\s_etbox\AAbo=\hbox{#2}\AAdi=\ht\AAbo\raise#1\AAdi\hbox{#3}}%
\font\tenmsb=msbm10 at 12pt \font\sevenmsb=msbm7 at 8pt
\font\fivemsb=msbm5 at 6pt
\newfam\msbfam
\textfont\msbfam=\tenmsb \scriptfont\msbfam=\sevenmsb
\scriptscriptfont\msbfam=\fivemsb

\newtheorem{theorem}{Theorem}

\newtheorem{remark}[theorem]{Remark}

\newtheorem{lemma}[theorem]{Lemma}

\numberwithin{equation}{section} \numberwithin{theorem}{section}

\renewcommand{\topmargin}{0cm}
\renewcommand{\oddsidemargin}{5mm}
\renewcommand{\evensidemargin}{5mm}
\renewcommand{\textwidth}{150mm}
\renewcommand{\textheight}{230mm}

\def\R{\mathbb R}

\def\na{\nabla}
\def\bn{\overline\nabla}

\def\f#1#2{\frac{#1}{#2}}

\def\a{\alpha}
\def\be{\beta}

\def\r{\Re_{I\!V}}

\def\p#1{\partial #1}

\def\de{\delta}
\def\De{\Delta}
\def\e{\eta}
\def\ep{\epsilon}

\def\G{\Gamma}
\def\g{\gamma}
\def\k{\kappa}
\def\la{\lambda}
\def\La{\Lambda}
\def\lan{\langle}
\def\ran{\rangle}

\def\Om{\Omega}
\def\th{\theta}
\def\Th{\Theta}
\def\si{\sigma}
\def\Si{\Sigma}

\def\r{\rho}
\def\z{\zeta}

\begin{document}

\title
[Poincar\'e inequality on minimal graphs over manifolds and applications]
{Poincar\'e inequality on minimal graphs over manifolds and applications}

\author{Qi Ding}
\address{Shanghai Center for Mathematical Sciences, Fudan University, Shanghai 200438, China}
\email{dingqi@fudan.edu.cn}

\thanks{The author would like to thank J$\mathrm{\ddot{u}}$rgen Jost, Yuanlong Xin and Xi-Ping Zhu for their interests on this work. He would like to thank Gioacchino Antonelli and Daniele Semola for careful reading
and valuable comments. He is supported by NSFC 11871156 and NSFC 11922106}
\date{}
\begin{abstract}
Let $B_2(p)$ be an $n$-dimensional smooth geodesic ball with Ricci curvature $\ge-(n-1)\k^2$ for some $\k\ge0$.
We establish the Sobolev inequality and the uniform Neumann-Poincar\'e inequality on each minimal graph over $B_1(p)$ by combining Cheeger-Colding theory and the current theory from geometric measure theory,
where the constants in the inequalities only depend on $n$, $\k$, the lower bound of the volume of $B_1(p)$.
As applications, we derive gradient estimates and a Liouville theorem for a minimal graph $M$ over a smooth complete noncompact manifold $\Si$ of nonnegative Ricci curvature and Euclidean volume growth.
Furthermore, we can show that any tangent cone of $\Si$ at infinity splits off a line isometrically provided the graphic function of $M$ admits linear growth.
\end{abstract}


\maketitle
\tableofcontents

\section{Introduction}

Let $B_{2R}(p)$ be an $n$-dimensional geodesic ball centered at $p$ with radius $2R$ and with smooth metric, whose Ricci curvature $\ge-(n-1)\k^2R^{-2}$ for some $\k\ge0$ and $R\ge1$.
From Anderson \cite{An1} or Croke \cite{Cr} (see also \cite{Gri,LW} for instance), 
there is a constant $\a_{n,\k}>0$ depending only on $n,\k$ such that the Sobolev inequality holds
\begin{equation}\label{isoperi}
\f{\a_{n,\k}}R\left(\mathcal{H}^{n}(B_R(p))\right)^{\f1{n}}\left(\int_{B_R(p)}|\phi|^{\f n{n-1}}\right)^{\f{n-1}n}\le\int_{B_R(p)}|D\phi|
\end{equation}
for any Lipschitz function $\phi$ on $B_R(p)$ with compact support in $B_R(p)$,
where $D$ is the Levi-Civita connection of $B_{2R}(p)$.
From Buser \cite{Bu} or Cheeger-Colding \cite{CC}, the Neumann-Poincar\'e inequality holds on $B_R(p)$.
Namely, up to choosing the constant $\a_{n,\k}>0$, it holds
\begin{equation}\label{NPoincare}
\a_{n,\k}\int_{B_R(p)}|f-\bar{f}|\le R\int_{B_R(p)}|D f|
\end{equation}
for any Lipschitz function $f$ on $B_R(p)$ with $\bar{f}=\fint_{B_R(p)}f$. 
These two inequalities act an important role in partial differential equation and geometry.

In 1967, Miranda obtained a Sobolev inequality for minimal graphs in Euclidean space \cite{Mm}.
After that, Bombieri \cite{Bo} and Michael (see the appendix of \cite{Si0}) gave a simpler proof for this inequality.
More general, Michael-Simon \cite{MS} proved the Sobolev inequality on arbitrary submanifolds in Euclidean space.
In particular, there is a constant $c_n>0$ depending only on $n$ such that
if $M$ is a complete $n$-dimensional minimal hypersurface in $\R^{n+m}$ for some $m\ge1$, then
\begin{equation}\aligned\label{SobMin000}
c_n\left(\int_{M}|f|^{\f n{n-1}}\right)^{\f{n-1}n}\le \int_{M}|\na f|
\endaligned
\end{equation}
for each $f\in C^1_c(M)$, where $\na$ is the Levi-Civita connection with respect to the induced metric on $M$.
Clearly, the Sobolev inequality \eqref{SobMin000} is equivalent to the isoperimetric inequality on $M$.
In \cite{Al}, Almgren proved the sharp isoperimetric inequality for minimizing submanifolds in Euclidean space with arbitrary codimensions.
Recently, Brendle \cite{Br} proved the sharp isoperimetric inequality for minimal submanifolds in Euclidean space with the codimensions $\le2$ (see \cite{LWW} for the relative version).

The Sobolev inequality on minimal submanifolds can be generalized from Euclidean space to Riemannian manifolds.
In \cite{HS}, Hoffman-Spruck obtained the Sobolev inequality on a minimal submanifold $M$ in a manifold $N$, with some geometric restrictions
involving the volume of $M$, the sectional curvatures of $N$ and the injectivity radius of $N$ (see \cite{Ho}\cite{Sc} for more results). 
Recently, Brendle \cite{Br1} proved the Sobolev inequality on minimal submanifolds in manifolds of nonnegative sectional curvature and Euclidean volume growth.

For each area-minimizing hypersurface $M$ in $\R^{n+1}$, Bombieri-Giusti \cite{BG} proved the following uniform Neumann-Poincar\'e inequality on $M$:
there is a constant $c_n'\ge1$ depending only on $n$ such that for any $x\in M$, $r>0$, $f\in C^1(B_r(x))$
\begin{equation}\aligned\label{PoincareM0}
\min_{k\in\R}\left(\int_{B_r(x)\cap M}|f-k|^\f{n}{n-1}\right)^{\f{n-1}n}\le c_n'\int_{B_{c_n'r}(x)\cap M}|\na f|.
\endaligned
\end{equation}
Furthermore, they got the Harnack's inequality for uniform elliptic equations on $M$.
Note that in general, \eqref{PoincareM0} does not hold on every complete minimal hypersurface, for instance, the catenoid.

Let $\Si$ be an $n$-dimensional complete noncompact Riemannian manifold with metric $\si$.
Let $u$ be a smooth solution to the minimal hypersurface equation on an open set $\Om\subset\Si$, i.e., $u$ satisfies the following elliptic equation on $\Om$:
\begin{equation}\label{u0}
\mathrm{div}_\Si\left(\f{Du}{\sqrt{1+|Du|^2}}\right)=0,
\end{equation}
where $\mathrm{div}_\Si$ is the divergence on $\Si$.
This equation can be seen as a natural generalization of the minimal hypersurface equation on Euclidean space.
The graph $M=\{(x,u(x))\in\Si\times\R|\ x\in\Om\}$ is said to be a \emph{minimal graph} over $\Om$. In fact,  the mean curvature of $M$ is zero w.r.t. its induced metric from the product metric $\si+dt^2$ of $\Si\times\R$.
The function $u$ is the height function of $M$, and we say it a \emph{minimal graphic function} on $\Om$.

In this paper, we obtain Sobolev and Neumann-Poincar\'e inequalities on minimal graphs over manifolds with Ricci curvature bounded below as follows.
\begin{theorem}\label{INEQU*}
For two constants $\k\ge0$ and $v>0$,
let $\Si$ be an $n$-dimensional smooth complete noncompact manifold with Ricci curvature
\begin{equation}\aligned\label{ConditionSi}
\mathrm{Ric}\ge-(n-1)\k^2\ \mathrm{on}\ B_{2}(p),\  \mathrm{and}\ \mathcal{H}^n(B_1(p))\ge v. 
\endaligned
\end{equation}
There is a constant $\Th=\Th_{\k,v}\ge1$ depending only on $n,\k,v$ such that if $M$ is a minimal graph over $B_{2}(p)$ with $\bar{p}=(p,0)\in M$ and $\p M\subset\p B_{2}(p)\times\R$, then  there hold a Sobolev inequality
\begin{equation}\aligned\label{SobM}
\left(\int_{ M}|\phi|^{\f n{n-1}}\right)^{\f{n-1}n}\le\Th\int_{M}|\na \phi|
\endaligned
\end{equation}
for any function $\phi\in C_0^1(M\cap B_1(\bar{p}))$, and a Neumann-Poincar\'e inequality
\begin{equation}\aligned\label{PoincareM}
\int_{M\cap B_{1/\Th}(\bar{p})}|\varphi-\bar{\varphi}|\le\Th\int_{M\cap B_1(\bar{p})}|\na \varphi|
\endaligned
\end{equation}
for any function $\varphi\in C^1(M\cap B_1(\bar{p}))$ with $\bar{\varphi}=\fint_{M\cap B_{1/\Th}(\bar{p})}\varphi$.
Here, $\na$ is the Levi-Civita connection with respect to the induced metric on $M$.
\end{theorem}
For $M=B_2(p)\times\{0\}\subset\Si\times\R$, the inequalities \eqref{SobM}\eqref{PoincareM} reduce to \eqref{isoperi}\eqref{NPoincare} with $R=1$.
Hence, the constant $\Th_{\k,v}$ in Theorem \ref{INEQU*} indeed depends on the constants $\k,v$. 
Once we get \eqref{PoincareM}, then \eqref{SobM} can be easily derived using the covering technique and  \eqref{PoincareM}.
However, Bombieri-Giusti \cite{BG} deduced the Neumann-Poincar\'e inequality on area-minimizing hypersurfaces in Euclidean space with the help of the Sobolev inequality on them. So, the order of the proof of Theorem \ref{INEQU*} is following:
\begin{enumerate}
\item prove \eqref{SobM} for area-minimizing hypersurfaces in a special class of manifolds;
\item prove \eqref{PoincareM} for a special class of area-minimizing hypersurfaces in a special class of manifolds using i);
\item prove \eqref{PoincareM} for area-minimizing hypersurfaces in a special class of manifolds using i)ii);
\item prove \eqref{SobM} for minimal graphs over manifolds satisfying \eqref{ConditionSi} using i)iii);
\item prove \eqref{PoincareM} for minimal graphs over manifolds satisfying \eqref{ConditionSi} using iii)iv).
\end{enumerate}

More precisely, our proof of Theorem \ref{INEQU*} is divided into three sections. In $\S 3$, we prove the two inequalities \eqref{SobM}\eqref{PoincareM} for area-minimizing hypersurfaces in almost Euclidean spaces. 
Here, the spaces are complete manifolds of almost nonnegative Ricci curvature whose large balls are very close to Euclidean balls in the Gromov-Hausdorff sense. Then we argue by contradiction, and empoly modified $C^1$-mappings as Gromov-Hausdorff approximations essentially from Colding \cite{C}. With the mappings, we are able to use the current theory from geometric measure theory, and get a Sobolev inequality on area-minimizing hypersurfaces in almost Euclidean spaces. Furthermore, through outward minimizing sets in area-minimizing hypersurfaces, using the mappings mentioned above and the current theory, we can deduce the Neumann-Poincar\'e inequality on area-minimizing hypersurfaces in almost Euclidean spaces based on our Sobolev inequality.

In $\S 4$, we consider a sequence of minimal graphs $M_i$ over $B_1(p_i)$ in $B_1(p_i)\times\R$, where $B_1(p_i)$ are geodesic balls with Ricci curvature uniformly bounded below. Suppose that $B_1(p_i)$ has a non-collapsing  limit $B_1(p_\infty)$ with Gromov-Hausdorff approximations $\Psi_i:\ B_1(p_i)\to B_1(p_\infty)$, and $\Psi_i(M_i)$ converges to a limit $M_\infty$ in $B_1(p_\infty)\times\R$ in the Hausdorff sense. We prove that any tangent cone of $M_\infty$ is a metric cone in a tangent cone of $B_1(p_\infty)\times\R$, where we 	utilize the geometry of manifolds of Ricci curvature bounded below and the property of minimal graphs as well as some results in \cite{D}.

In $\S 5$, we deal with the general Sobolev inequality on minimal graphs through an argument by contradiction with Cheeger-Colding theory. 
From \cite{D1}, the cross section of every tangent cone of limits of minimal graphs is connected in the sense of splitting off a Euclidean factor isometrically.
Combining tangent cones of $M_\infty$ mentioned above and dimension estimate of singular sets of Ricci limit space, through dimension reduction argument we are able to reduce the Sobolev inequality to the Neumann-Poincar\'e inequality on area-minimizing hypersurfaces in almost Euclidean spaces already established. In a similar manner, we can get the Neumann-Poincar\'e inequality on minimal graphs in manifolds with Ricci curvature bounded below.

Moreover, our argument of the proof of Theorem \ref{INEQU*} works for area-minimizing hypersurfaces in almost Euclidean balls. In particular, we have a generalization of \eqref{PoincareM0} as follows (see Theorem \ref{PTh0varphi}). 
\begin{theorem}\label{almostEB}
For each integer $n\ge2$, there are constants $\ep,\Th_0>0$ depending only on $n$ such that
if $N$ is an $(n+1)$-dimensional smooth complete noncompact Riemannian manifold with
$\mathrm{Ric}\ge-n\ep^2$ on $B_2(q)\subset N$ and $\mathcal{H}^{n+1}(B_1(q))\ge(1-\ep)\omega_{n+1}$, and $M$ is an area-minimizing hypersurface in $B_2(q)$ with $q\in M$ and $\p M\subset\p B_2(q)$, then
\begin{equation}\aligned
\min_{k\in\R}\left(\int_{M\cap B_{1/\Th_0}(q)}|\varphi-k|^{\f n{n-1}}\right)^{\f{n-1}n}\le\Th_0\int_{M\cap B_1(q)}|\na \varphi|
\endaligned
\end{equation}
for any function $\varphi\in C^1(B_1(q))$. Here, $\na$ is the Levi-Civita connection with respect to the induced metric on $M$.
\end{theorem}

Now we consider applications of Theorem \ref{INEQU*}. First, let us review some classical results on gradient estimates for minimal graphs over Euclidean space. In 1969, Bombieri-De Giorgi-Miranda \cite{BDM}  (see also \cite{GT}) showed the interior gradient estimate for any solution $u$ to the minimal hypersurface equation on a ball $B_r(y)\subset\R^n$ (Finn proved the 2-dimensional case \cite{Fi}). Namely, there is a constant $c=c(n)>0$ depending only on $n$ so that
\begin{equation}\aligned\label{GESTEuc0}
|Du(y)|\le c e^{c r^{-1}\left(u(y)-\sup_{x\in B_{r}(y)}u(x)\right)}.
\endaligned
\end{equation}
From the examples constructed by Finn \cite{Fi1}, the estimate \eqref{GESTEuc0} is sharp in the sense of the linear exponential dependence on the solutions. 
Wang \cite{W} obtained a gradient estimate of $u$ via the maximum principle, and Spruck \cite{Sp} got it for the case of manifolds.
In \cite{RSS}, Rosenberg-Schulze-Spruck  proved an interior gradient estimate for minimal graphs over manifolds of nonnegative Ricci curvature, where the estimation depends on the lower bound of sectional curvature.

With Theorem \ref{INEQU*}, one can deduce the mean value inequality for positive subharmonic(superharmonic) functions on area-minimizing hypersurfaces by De Giorgi-Nash-Moser iteration. Combining the integral estimates for gradient of solutions in \cite{DJX2}, we obtain the following interior gradient estimates.
\begin{theorem}
Let $\Si$ be an $n$-dimensional complete noncompact Riemannian manifold with Ricci curvature $\ge-(n-1)\k^2r^{-2}$ on the geodesic ball $B_r(p)\subset\Si$ for some $\k\ge0$.
Suppose $\mathcal{H}^{n}(B_r(x))\ge vr^n$ for some constant $v>0$. Let $u$ be a smooth solution to \eqref{u0} on $B_{r}(p)$, then
\begin{equation}\aligned\label{0GEM}
|Du(p)|\le c_{\k,v} e^{c_{\k,v} r^{-1}\left(u(p)-\sup_{x\in B_{r}(p)}u(x)\right)},
\endaligned
\end{equation}
where $c_{\k,v}$ is a positive constant depending only on $n,\k,v$.
\end{theorem}
The estimation \eqref{0GEM} can be seen as a generalization of \eqref{GESTEuc0} on Euclidean space.
As far as our knowledge, 
\eqref{0GEM} is the first one independent of sectional curvature among interior gradient estimates of solutions to \eqref{u0} on manifolds. 
Moreover, if we further assume that $u$ is a smooth solution to \eqref{u0} on a manifold of nonnegative Ricci curvature and Euclidean volume growth, then $|Du|$ can be controlled by a polynomial of $|u|$ (see Theorem \ref{GEuglobal}) using the idea of Bombieri-Giusti \cite{BG}, where they have already obtained the Euclidean version.

With gradient estimate \eqref{GESTEuc0}, Bombieri-De Giorgi-Miranda \cite{BDM} proved that any smooth solution  $u$ to minimal hypersurface equation on $\R^n$ must be affine linear provided there is a constant $K>0$ so that $u(x)\ge-K(|x|+1)$ for any $x\in\R^n$.
Using \eqref{0GEM}, we get a Liouville type theorem for any entire solution to \eqref{u0} with sub-linear growth for its negative part  as follows.
\begin{theorem}\label{hlcMG}
Let $\Si$ be a complete noncompact manifold of nonnegative Ricci curvature and Euclidean volume growth. If a smooth solution $u$ to \eqref{u0} on $\Si$ admits sub-linear growth for its negative part, i.e.,
$$\limsup_{\Si\ni x\rightarrow\infty}\f{\max\{-u(x),0\}}{d(x,p)}=0$$
for some point $p\in\Si$, then $u$ is a constant.
\end{theorem}
Let us review Liouville type theorems for solutions to \eqref{u0} briefly.
In \cite{RSS}, Rosenberg-Schulze-Spruck proved that any smooth solution $u$ to \eqref{u0} on a complete manifold of nonnegative sectional curvature is a constant provided $u$ admits sub-linear growth for its negative part. 
Casteras-Heinonen-Holopainen \cite{CHH} showed that any smooth nonnegative solution $u$ to \eqref{u0} on a complete manifold of asympototically nonnegative sectional curvature is a constant with at most linear growth condition of $u$.
For nonnegative smooth solutions to \eqref{u0}, the condition of nonnegative sectional curvature can be weaken to satisfying nonnegative Ricci curvature and sectional curvature uniformly bounded below in \cite{RSS}.
In \cite{D0}, the condition can be further weaken to satisfying volume doubling property and Neumann-Poincar\'e inequality  (In particular, it holds for complete manifolds of nonnegative Ricci curvature, see \cite{CMMR,D0}).

In \cite{CCM}, Cheeger-Colding-Minicozzi proved the splitting of tangent cones at infinity for manifolds of nonnegative Ricci curvature which admit linear growth harmonic functions on them.
In general, such manifolds may not split off a line isometrically themself even provided they admit Euclidean volume growth.
As the same reason, if the non-constant solution $u$ to \eqref{u0} has linear growth at most for its negative part, i.e.,
$$\limsup_{\Si\ni x\rightarrow\infty}\f{\max\{-u(x),0\}}{d(x,p)}<\infty$$
for some point $p\in\Si$, then it may not be true about affine linearity of $u$, or the splitting of $\Si$. 
Instead, we can estimate the Hausdorff distance between different level sets of $u$ in compact sets in the scaling sense (see Lemma \ref{C0est}), and get the following splitting based on the almost rigidity theorem of Cheeger-Colding \cite{CC}.
\begin{theorem}\label{splitting}
Let $\Si$ be a complete noncompact Riemannian manifold of nonnegative Ricci curvature and Euclidean volume growth. If $u$ is a smooth non-constant solution to \eqref{u0} on $\Si$ with linear growth at most for its negative part, then any tangent cone of $\Si$ at infinity splits off a line isometrically.
\end{theorem}
The splitting of tangent cones of $\Si$ at infinity indeed requires linear growth of $u$ on $\Si$ (see examples in $\S 6$ of \cite{DJX2}).
Let $M$ denote the minimal graph over $\Si$ with the graphic function $u$. 
If $u$ has linear growth, then the tangent cone of $M$ at infinity can split off a line isometrically.
However, the splitting does not depend on the linear growth of $u$ as follows.
\begin{theorem}\label{splitting***}
Let $\Si$ be a complete noncompact Riemannian manifold of nonnegative Ricci curvature and Euclidean volume growth. If $M$ is a minimal graph over $\Si$ with the non-constant smooth graphic function, then there is a tangent cone of $M$ at infinity, which is a metric cone and splits off a line isometrically.
\end{theorem}

\section{Preliminary}

Let us recall the (integer multiplicity) current theory from geometric measure theory in smooth manifolds briefly (see \cite{LYa}\cite{S} for the classical case in Euclidean space, see \cite{AK} for a general case in metric spaces).
For an integer $n\ge1$, let $N$ be an $(n+1)$-dimensional complete smooth Riemannian manifold, and $U$ be an open subset of $N$.
For an integer $0\le k\le n+1$, let $\mathcal{D}^k(U)$ denote the set including all smooth $k$-forms with compact supports in $U$.
For any $\omega\in\mathcal{D}^k(U)$, we denote 
$$\omega=\sum_{1\le i_1<\cdots<i_k\le n+1}f_{i_1\cdots i_k}\be_1\wedge\cdots\wedge\be_{i_k},$$
where $\be_1,\cdots,\be_{n+1}$ are $(n+1)$-orthonormal smooth 1-forms with compact supports on $U$.
Then we define a norm $|\cdot|_U$ by
$$|\omega|_U=\sup_{x\in U}\lan\omega(x),\omega(x)\ran^{1/2}=\sup_{x\in U}\left(\sum_{1\le i_1<\cdots<i_k\le n+1}f_{i_1\cdots i_k}^2(x)\right)^{1/2}.$$
Let $\mathcal{D}_k(U)$ denote the set of $k$-currents in $U$, which are continuous linear functionals on $\mathcal{D}^k(U)$. 
For $k\ge1$, the boundary of a $k$-current $T$ in $U$ is the $(k-1)$-current $\p T$ defined by
$$\p T(\omega)=T(d\omega)\qquad \mathrm{for\ each}\ \omega\in \mathcal{D}^k(U).$$
Let $\mathbb{M}(T)$ denote the mass of $T$ defined by
$$\mathbb{M}(T)=\sup_{|\omega|_U\le1,\omega\in\mathcal{D}^k(U)}T(\omega).$$
For a sequence $T_j\in\mathcal{D}_n(U)$ and $T\in\mathcal{D}_n(U)$, we denote $T_j\rightharpoonup T$ 
if $\lim_{j\rightarrow\infty}T_j(\omega)=T(\omega)$ for all $\omega\in \mathcal{D}^k(U)$,
which implies
$$\mathbb{M}(T)\le\liminf_{j\rightarrow\infty}\mathbb{M}(T_j).$$
In other words, the mass is lower semi-continuous with respect to weak convergence.	

Let $\mathcal{H}^s$ denote the $s$-dimensional Hausdorff measure for each constant $s\ge0$.
For $T\in\mathcal{D}_k(U)$, $T$ is said to be an \emph{integer multiplicity current} if $T$ can be expressed as $(S,\th,\xi)$, i.e.,
$$T(\omega)=\int_S\th(x)\lan \omega(x),\xi(x)\ran d\mathcal{H}^k(x)\qquad \mathrm{for\ each}\ \omega\in \mathcal{D}^k(U),$$
where $S$ is a ($\mathcal{H}^k$-measurable) countably $k$-rectifiable subset of $U$, $\th$ is a locally $\mathcal{H}^k$-integrable positive integer-valued function, and $\xi(x)$ is a $k$-vector representing the approximate tangent space $T_xS$ for a.e. $x\in S$.
Here, $S$ is called the support of $T$, $\th$ is called the multiplicity function of $T$, and $\xi$ is called the orientation of $T$. For simplicity, we often omit the volume element $d\mathcal{H}^k$.

If both $T$ and $\p T$ are integer multiplicity rectifiable currents, then $T$ is called an \emph{integral current}.
For a Borel set $W\subset U$, the current $U\llcorner W$ is defined by
$$T\llcorner W(\omega)=\int_{S\cap W}\th\lan \omega,\xi\ran\qquad \mathrm{for\ each}\ \omega\in \mathcal{D}^k(U),$$
where $\chi_{_W}$ is the characteristic function on $W$. If we further assume $W$ open with $W\subset\subset U$, then
$$\mathbb{M}(T\llcorner W)=\sup_{|\omega|_U\le1,\omega\in\mathcal{D}^k(U),\mathrm{spt}\omega\subset W}T(\omega).$$
For the sequence $T_j\in\mathcal{D}_n(U)$ with $T_j\rightharpoonup T$, it follows that
\begin{equation}\label{TWTjW}
\mathbb{M}(T\llcorner W)\le\liminf_{j\rightarrow\infty}\mathbb{M}(T_j\llcorner W).
\end{equation}

Let $\tilde{N}$ be an $m$-dimensional smooth Riemannian manifold.
Let $f:\ U\rightarrow \tilde{N}$ be a $C^1$-mapping, and $T=(S,\th,\xi)$ be an integer multiplicity $k$-current in $U$ with $k\le m$. Let $f_*\xi$ denote the push-forward of $\xi$, which is an orientation of $f(S)$ in $\tilde{N}$. We define $f(T)\in \mathcal{D}_k(\tilde{N})$ by letting
\begin{equation*}
f(T)(\omega)=\int_S\th\lan \omega\circ f,f_*\xi\ran=\int_{f(S)}\left\lan\omega(y),\sum_{x\in f^{-1}(y)\cap S_*}\th(x)\f{f_*\xi(x)}{|f_*\xi(x)|}\right\ran d\mathcal{H}^k(y)
\end{equation*}
for each $\omega\in \mathcal{D}^k(\tilde{N})$,
where $S_*=\{x\in S|\, |f_*\xi(x)|>0\}$. It's clear that $f(T)$ is an integer multiplicity current in $\tilde{N}$, and $\p(f(T))=f(\p T)$.

For a countably $n$-rectifiable set $S$ (with orientation $\xi_S$) in $U$, let $[| S|]$ denote the integer multiplicity current associated with $S$, i.e., the current $[| S|]$ has the support $\overline{S}$ and the multiplicity one with the orientation $\xi_S$.
For an $(n+1)$-dimensional Borel set $V\subset U$ with $n$-rectifiable boundary, let $[| V|]$ denote the integer multiplicity current associated with $V$, i.e., the current $[| V|]$ has the support $\overline{V}$, the multiplicity one, and the orientation $e_1\wedge\cdots\wedge e_{n+1}$. Here, $e_1,\cdots, e_{n+1}$ is a standard orthonormal local tangent field of $U$.
It's clear that $\p [| V|]=[| \p V|]$.

A countably $n$-rectifiable set $S$ (with orientation $\xi$) is said to be an \emph{area-minimizing} hypersurface in $U$ if the associated current $[| S|]$ is a minimizing current in $U$. Namely,
$\mathbb{M}([| S|]\llcorner W)\le \mathbb{M}(T\llcorner W)$ whenever $W\subset\subset U$, $\p T=\p[| S|]$ in $U$, spt$(T-[| S|])$ is compact in $W$ (see \cite{S} for instance).
Let $\Si$ be an $n$-dimensional complete manifold, and $\Om$ be an open subset of $\Si$.
Let $M$ be a minimal graph over $\Om$ in $\Si\times\R$, then $M$ is an area-minimizing hypersurface in $\Om\times\R$ (see \cite{DJX1} for instance).

From Federer and Fleming \cite{FF}, there holds the following compactness theorem (see also Theorem 27.3 in \cite{S}).
\begin{theorem}\label{FF}
If $T_j$ is a sequence of integer multiplicity currents in $\mathcal{D}_n(U)$ with
$$\sup_{j\ge1}\left(\mathbb{M}(T_j\llcorner W)+\mathbb{M}(\p T_j\llcorner W)\right)<\infty\qquad \mathrm{for\ any\ open}\ W\subset\subset U,$$
then there is an integer multiplicity $T\in \mathcal{D}_n(U)$ and a subsequence $T_{j'}$ such that $T_{j'}\rightharpoonup T$ in $U$.
\end{theorem}

{\bf Notional convention.} If a point $q$ belongs to some metric space $Z$ defined in the text, we let $B_r(q)$ denote the geodesic ball in this $Z$ with radius $r$ and centered at $q$.
We always denote $B_r(0)$ be the ball in Euclidean space $\R^{n+1}$ with radius $r$ and centered at the origin.
Let $\omega_k$ denote the volume of $k$-dimensional unit Euclidean ball for each integer $k>0$.
If volume elements in integrations on countably $k$-rectifiable sets are $k$-dimensional Hausdorff measure, we often omit the volume elements for simplicity.
Let $N$ be a complete Riemannian manifold with metric $g$. For a point $p\in N$ and a subset $U\subset N$, let $\f1r(U,g,p)=(U,r^{-1}g,p)$, where $r^{-1}g$ is a scaling of metric $g$. For simplicity, we always denote $\f1r(U,p)=\f1r(U,g,p)$.

For two metric spaces $Z_1,Z_2$, let $d_{GH}(Z_1,Z_2)$ denote the Gromov-Hausdorff distance of $Z_1,Z_2$ (see \cite{BBI}\cite{GLP} for further results).
Suppose a sequence of metric spaces $(Z_i,p_i)$ converges to a metric space $(Z_\infty,p_\infty)$ with $\ep_i$-Gromov-Hausdorff approximations $\Psi_i:\, Z_i\to Z_\infty$. Let $K_i$ be a Borel set in $Z_i$, from Blaschke theorem there is a closed set $K_\infty\subset Z_\infty$ such that for each $0<r<\infty$, $\Psi_i(K_i\cap B_r(p_i))$ converges to $K_\infty\cap B_r(p_\infty)$ in the Hausdorff sense up to choosing the (diagonal) subsequence. For simplicity, we call that $(K_i, p_i)$ converges \emph{in the induced Hausdorff sense} to $(K_\infty, p_\infty)$. For the compact $K_\infty$, we may drop $p_i,p_\infty$ directly (see also Preliminary in \cite{D}).

\section{Area-minimizing hypersurfaces in almost Euclidean spaces}

For $n\ge2$, let $Q_i$ be a sequence of $(n+1)$-dimensional complete noncompact Riemannian manifolds such that
the geodesic balls $B_{r_i}(q_i)\subset Q_i$ with radius $r_i\rightarrow\infty$ and centered at $q_i\in Q_i$ satisfy
\begin{equation}\label{RicBRiqige0*}
\liminf_{i\rightarrow\infty}\left(r_i^2\inf_{B_{r_i}(q_i)}\mathrm{Ric}_{Q_i}\right)\ge0,
\end{equation}
and
\begin{equation}\label{BRiqiBRi0epi}
\lim_{i\rightarrow\infty}d_{GH}(B_{r_i}(q_i),B_{r_i}(0))=0.
\end{equation}
Then there is a sequence of $\de_i$-Gromov Hausdorff approximations $\Psi_{Q_i}:\, B_{r_i}(q_i)\to B_{r_i}(0)$ for some sequence $\de_i\to0$ as $i\to\infty$.
Here, $B_{r_i}(0)$ is the ball in $\R^{n+1}$ with the radius $r_i$ and centered at the origin.
In general, $\Psi_{Q_i}$ may not be continuous. 

Recalling the Gromov-Hausdorff approximation introduced by Colding \cite{C}.
Let $e_j$ be the standard orthonormal basis of $\R^{n+1}$, and we choose $e_{i,j}\in B_{r_i}(q_i)\subset Q_i$ for each integer $i\ge1$ so that
$d(r_ie_j,\Psi_{Q_i}(e_{i,j}))<\de_i$ for\ each $j=1,\cdots,n+1$.
Let $b_{i,j}$ be a Lipschitz function on $Q_i$ defined through distance functions by
$$b_{i,j}=d(e_{i,j},\cdot)-d(e_{i,j},q_i)\qquad\qquad \mathrm{for\ each}\ j=1,\cdots,n+1,\, \ i\ge1.$$
We set
$$\Psi_i=(b_{i,1},\cdots,b_{i,n+1}):\ B_1(q_i)\rightarrow \R^{n+1}.$$
Then there is a sequence of positive numbers $\de_i'\rightarrow0$ such that 
$\Psi_i:\ B_1(q_i)\rightarrow B_{1+\de_i'}(0)$ is a $\de_i'$-Gromov-Hausdorff approximation for each $i$. 
From the proof of Lemma 2.1 in \cite{C}, it follows that
\begin{equation}\label{PhiiB1qiomegan0}
\lim_{i\rightarrow\infty}\mathcal{H}^{n+1}(\Psi_i(B_1(q_i)))=\omega_{n+1}.
\end{equation}

For each $i,j$, the Lipschitz function $b_{i,j}$ has Lipschitz constant 1 with $|db_{i,j}|=1$ where defined, and $b_{i,j}$ is not $C^1$ on the cut locus $\mathcal{C}_{i,j}$ w.r.t. $e_{i,j}$.
Now let us modify $b_{i,j}$ by Lemma \ref{Mollifybij} in the appendix on a neighborhood of $\mathcal{C}_{i,j}$, and get a smooth function $\mathbf{b}_{i,j}$ on $\overline{B_1(q_i)}$ such that $\mathbf{b}_{i,j}=b_{i,j}$ in $\overline{B_1(q_i)}\setminus\mathcal{E}_i$ for some open set $\mathcal{E}_i\supset\mathcal{C}_{i,j}$ with $\mathcal{H}^{n+1}(\mathcal{E}_i)<\f1i$ and
\begin{equation}\aligned\label{bdij11i}
|d\mathbf{b}_{i,j}|\le1+\f1i\qquad \mathrm{on}\ \ \overline{B_1(q_i)}.
\endaligned
\end{equation}
For each $i\ge1$, let $\Phi_i$ be a smooth mapping defined by
\begin{equation}\label{Phii***}
\Phi_i=(\mathbf{b}_{i,1},\cdots,\mathbf{b}_{i,n+1}):\ \overline{B_1(q_i)}\rightarrow \R^{n+1}.
\end{equation}
Then $\lim_{i\to\infty}\mathcal{H}^{n+1}(\Phi_i(\mathcal{E}_i))=0$. 
Moreover, there is a sequence of positive numbers $\tau_i\rightarrow0$ such that 
$\Phi_i:\ B_1(q_i)\rightarrow B_{1+\tau_i}(0)$ is a $\tau_i$-Gromov-Hausdorff approximation.
Combining \eqref{PhiiB1qiomegan0}, it follows that
\begin{equation}\label{PhiiB1qiomegan}
\lim_{i\rightarrow\infty}\mathcal{H}^{n+1}(\Phi_i(B_1(q_i)))=\omega_{n+1}.
\end{equation}
For each $i$, $d\Phi_i(d\Phi_i)^T$ is nonnegative definite, and with \eqref{bdij11i} its trace satisfies
\begin{equation*}\aligned
\mathrm{tr}\left(d\Phi_i(d\Phi_i)^T\right)=\mathrm{tr}\left((\lan d\mathbf{b}_{i,j},d\mathbf{b}_{i,k}\ran)_{(n+1)\times (n+1)}\right)=\sum_{j=1}^{n+1}\lan d\mathbf{b}_{i,j},d\mathbf{b}_{i,j}\ran\le(n+1)\left(1+\f1i\right)^2.
\endaligned
\end{equation*}
From the inequality of arithmetic and geometric means, the determinant satisfies
\begin{equation}\aligned\label{detPhi}
\mathrm{det}(d\Phi_i)=\sqrt{\mathrm{det}\left(d\Phi_i(d\Phi_i)^T\right)}\le\left(\f{\mathrm{tr}\left(d\Phi_i(d\Phi_i)^T\right)}{n+1}\right)^{(n+1)/2}\le\left(1+\f1i\right)^{n+1}.
\endaligned
\end{equation}

Let $M_i$ be a sequence of area-minimizing hypersurfaces in $B_1(q_i)$ with $q_i\in M_i$ and $\p M_i\subset\p B_1(q_i)$. 
From Theorem A.1.8 in \cite{CCo1} by Cheeger-Colding, there are open sets $U_i\subset B_1(q_i)$ such that $U_i$ is
homeomorphic to $B_1(0)$, and $U_i\supset B_{1-\tau_i'}(q_i)$ for some sequence $\tau_i'\to0$ as $i\to\infty$. Then
there are two open subsets $W_i^+,W_i^-$ in $U_i$ with $\p W_i^\pm\cap U_i=M_i\cap U_i$ and $W_i^+\cup W_i^-=U_i\setminus M_i$. There is a sequence $\tau_i^*\in(\tau_i',2\tau_i')$ such that $\mathcal{H}^{n-1}(M_i\cap\p B_{1-\tau_i^*})<\infty$.
Let $E_i^\pm=W_i^\pm\cap B_{1-\tau_i^*}(q_i)$. Since $M_i$ is area-minimizing in $B_1(q_i)$, then
\begin{equation}\aligned
\mathcal{H}^n(\p W_i^\pm\cap B_{1-\tau_i^*}(q_i))\le\mathcal{H}^n(W_i^\mp\cap\p B_{1-\tau_i^*}(q_i)).
\endaligned
\end{equation}
It follows that
\begin{equation}\aligned\label{HnEipm}
&\mathcal{H}^n(\p E_i^\pm)=\mathcal{H}^n(W_i^\pm\cap\p B_{1-\tau_i^*}(q_i))+\mathcal{H}^n(\p W_i^\pm\cap B_{1-\tau_i^*}(q_i))\\
\le&\mathcal{H}^n(W_i^\pm\cap\p B_{1-\tau_i^*}(q_i))+\mathcal{H}^n(W_i^\mp\cap\p B_{1-\tau_i^*}(q_i))=\mathcal{H}^n(\p B_{1-\tau_i^*}(q_i)).
\endaligned
\end{equation}

Suppose that $\Phi_i(E_i^\pm),\Phi_i(M_i)$ converges to closed sets $E_\infty^\pm,M_\infty\subset\R^{n+1}$ in the Hausdorff sense, respectively. 
From Theorem 1.1 in \cite{D}, $M_\infty\cap B_1(0)=\p E_\infty^\pm\cap B_1(0)$ is area-minimizing in $B_1(0)$ with
\begin{equation}\aligned\label{HnMinftyBtpMipi000}
\mathcal{H}^n(M_\infty\cap B_t(0))=\lim_{i\rightarrow\infty}\mathcal{H}^{n}(M_i\cap B_t(q_i))
\endaligned
\end{equation}
for all $t\in(0,1)$.
For each $i$, let $[| E^\pm_i|]$ denote the integer multiplicity current associated with $E^\pm_i$, and $[| M_i|]$ denote the integer multiplicity current associated with $M_i$ such that $[| M_i|]\llcorner U_i=[| \p E_i^+|]\llcorner U_i$. 
Let $[| M_\infty^*|]$ denote the integer multiplicity current associated with $M_\infty\cap B_1(0)$ such that $[| M_\infty^*|]=[| \p E_\infty^+|]\llcorner B_1(0)$. 
\begin{lemma}\label{ConvcurrentPhiiMi}
There is a subsequence $i'\rightarrow\infty$ such that $\Phi_{i'}([| M_{i'}|])$ converges weakly to $[| M_\infty^*|]$ (in the sense of distribution) in $B_1(0)$.
 \end{lemma}
\begin{proof}
Let $\Phi_i([| E^\pm_i|])$ be the integer multiplicity current in $B_{1+\tau_i}(0)$ with the support in $\Phi_i(E^\pm_i)$, which is the mapping of $[| E^\pm_i|]$ by $\Phi_i$.
Then $\Phi_i(\p [| E^\pm_i|])=\p\Phi_i([| E^\pm_i|])$.
For any ball $B_r(z)\subset B_1(0)$,
\begin{equation*}\aligned
&\mathcal{H}^{n+1}\left((\Phi_i(E^+_i)\cup\Phi_i(E^-_i))\cap B_r(z)\right)=\mathcal{H}^{n+1}(\Phi_i(B_1(q_i))\cap B_r(z))\\
=&\mathcal{H}^{n+1}(\Phi_i(B_1(q_i)))-\mathcal{H}^{n+1}(\Phi_i(B_1(q_i))\setminus B_r(z))\\
\ge&\mathcal{H}^{n+1}(\Phi_i(B_1(q_i)))-\left(1+\f1i\right)^{n+1}\mathcal{H}^{n+1}(B_{1+\tau_i}(0)\setminus B_r(z)).
\endaligned
\end{equation*}
With the volume convergence by Colding \cite{C}, one has
\begin{equation}\aligned\label{PhiiVpmige}
\liminf_{i\rightarrow\infty}\mathcal{H}^{n+1}\left((\Phi_i(E^+_i)\cup\Phi_i(E^-_i))\cap B_r(z)\right)\ge\omega_{n+1}-\omega_{n+1}\left(1-r^{n+1}\right)=\omega_{n+1}r^{n+1}.
\endaligned
\end{equation}
From \eqref{detPhi} and co-area formula, 
\begin{equation}\aligned\label{Phii|Eipm|}
\mathcal{H}^{n+1}(\Phi_i(E^\pm_i))\le\int_{E^\pm_i}\det(d\Phi_i)\le \left(1+\f1i\right)^{n+1}\mathcal{H}^{n+1}(E^\pm_i).
\endaligned
\end{equation}
Then it follows that
\begin{equation*}\aligned
&\mathcal{H}^{n+1}(\Phi_i(E^+_i)\cap B_r(z))+\mathcal{H}^{n+1}(\Phi_i(E^-_i)\cap B_r(z))\\
=&\mathcal{H}^{n+1}(\Phi_i(E^+_i))+\mathcal{H}^{n+1}(\Phi_i(E^-_i))-\mathcal{H}^{n+1}(\Phi_i(E^+_i)\setminus B_r(z))-\mathcal{H}^{n+1}(\Phi_i(E^-_i)\setminus B_r(z))\\
\le& \left(1+\f1i\right)^{n+1}\mathcal{H}^{n+1}(E^+_i)+\left(1+\f1i\right)^{n+1}\mathcal{H}^{n+1}(E^-_i)-\mathcal{H}^{n+1}(\Phi_i(B_1(q_i))\setminus B_r(z))\\
\le& \left(1+\f1i\right)^{n+1}\mathcal{H}^{n+1}(B_1(q_i))-\mathcal{H}^{n+1}(\Phi_i(B_1(q_i)))+\mathcal{H}^{n+1}(B_r(z)).
\endaligned
\end{equation*}
 Combining \eqref{PhiiB1qiomegan}\eqref{PhiiVpmige}, we get
\begin{equation}\aligned\label{PhiiVpmile}
\lim_{i\rightarrow\infty}\left(\mathcal{H}^{n+1}(\Phi_i(E^+_i)\cap B_r(z))+\mathcal{H}^{n+1}(\Phi_i(E^-_i)\cap B_r(z))\right)=\omega_{n+1}r^{n+1}.
\endaligned
\end{equation}
Let $T_i=\Phi_i([| E^+_i|])+\Phi_i([| E^-_i|])$. Combining \eqref{detPhi}\eqref{PhiiVpmige}\eqref{PhiiVpmile}, we conclude
\begin{equation}\aligned\label{eqivVT}
\lim_{i\rightarrow\infty}\mathbb{M}(T_i\llcorner B_r(z))=\omega_{n+1}r^{n+1}.
\endaligned
\end{equation}
Hence we have $T_i\rightharpoonup[| B_1(0)|]$ as $i\rightarrow\infty$, i.e.,
\begin{equation}\aligned\label{eqivVT}
\lim_{i\rightarrow\infty}T_i(\omega)=[| B_1(0)|](\omega)\qquad \mathrm{for\ any}\ \omega\in \mathcal{D}^{n+1}(B_1(0)).
\endaligned
\end{equation}

For each integer $i\ge1$, let $\{\xi_{i,j}\}_{j=1}^n$ denote a local orthonormal tangent field in a neighborhood of the considered regular point of $M_i$, and $\na^{M_i}\Phi_i$ denote an $(n+1)\times n$ matrix on $B_1(q_i)$ defined by $(\lan d\mathbf{b}_{i,k},\xi_{i,j}\ran)$ with $k=1,\cdots,n+1$ and $j=1,\cdots,n$.
From \eqref{bdij11i}, we have
$$\mathrm{tr}\left(\na^{M_i}\Phi_i(\na^{M_i}\Phi_i)^T\right)=\sum_{j=1}^n\sum_{k=1}^{n+1}\lan d\mathbf{b}_{i,k},\xi_{i,j}\ran^2\le\sum_{k=1}^{n+1}|d\mathbf{b}_{i,k}|^2\le (n+1)\left(1+\f1i\right)^2.$$
Similar to \eqref{detPhi}, the inequality of arithmetic and geometric means implies
\begin{equation}\aligned\label{JPhi}
J\Phi_i\triangleq\sqrt{\mathrm{det}\big(\na^{M_i}\Phi_i(\na^{M_i}\Phi_i)^T\big)}\le\left(\f{n+1}n\left(1+\f1i\right)^2\right)^{n/2}<\sqrt{e}\left(1+\f1i\right)^n,
\endaligned
\end{equation}
where we have used $\left(1+1/n\right)^n<e$.
With \eqref{HnEipm}, the mass of $\Phi_i(\p [| E^\pm_i|])$ satisfies
\begin{equation}\aligned\label{MPhiipEipm}
&\mathbb{M}\left(\Phi_i(\p [| E^\pm_i|])\right)\le\int_{\p E^\pm_i}J\Phi_i\\
\le& \sqrt{e}\left(1+\f1i\right)^n\mathcal{H}^{n}(\p E^\pm_i)\le\sqrt{e}\left(1+\f1i\right)^n\mathcal{H}^n(\p B_{1-\tau_i^*}(q_i)).
\endaligned
\end{equation}
By Federer-Fleming compactness theorem (see Theorem \ref{FF}), from \eqref{Phii|Eipm|}\eqref{MPhiipEipm} there are two integer multiplicity currents $T^+_\infty,T^-_\infty$,
and a subsequence $i'\to\infty$ such that $\Phi_{i'}([| E^\pm_{i'}|])\rightharpoonup T^\pm_\infty$.
From $T_i=\Phi_i([| E^+_i|])+\Phi_i([| E^-_i|])\rightharpoonup[| B_1(0)|]$, it follows that
$$T^+_\infty+T^-_\infty=[| B_1(0)|].$$
Recalling that $\Phi_i(E_i^\pm)$ converges to the closed set $E_\infty^\pm\subset\overline{B_1(0)}$ in the Hausdorff sense such that $\p E^\pm_\infty\cap B_1(0)=M_\infty\cap B_1(0)$. Then with \eqref{PhiiVpmige}\eqref{PhiiVpmile}, we deduce
\begin{equation}\aligned
\lim_{i\rightarrow\infty}\Phi_{i}([| E^\pm_{i}|])=T^\pm_\infty=[| E^\pm_\infty|].
\endaligned
\end{equation}
Let $\omega$ be a smooth $n$-form with the support in $B_1(0)$. For large $i$, we have
\begin{equation}\aligned
\Phi_{i}([| E_{i}^+|])(d\omega)=\p\Phi_{i}([| E_{i}^+|])(\omega)=\Phi_{i}([| M_i|])(\omega).
\endaligned
\end{equation}
Since $\p [| E^+_\infty|]\llcorner B_1(0)=[| M_\infty^*|]$, then taking the limit in the above equality implies
$$[|M_\infty^*|](\omega)= [| E^+_\infty|](d\omega)=\lim_{i\to\infty} \Phi_{i}([| E_{i}^+|])(d\omega)=\lim_{i\rightarrow\infty} \Phi_{i}([| M_i|])(\omega).$$
This completes the proof.
\end{proof}

From Lemma 3.3 and Corollary 4.6 in \cite{D},
we have the following volume estimates for area-minimizing hypersurfaces in a class of manifolds.
\begin{lemma}\label{epdeHM}
For any $\ep>0$ there is a constant $\de>0$ such that if $N$ is an $(n+1)$-dimensional complete non-compact Riemannian manifold with Ricci curvature $\mathrm{Ric}\ge-n\de^2$ on a unit geodesic ball $B_1(p)\subset N$, and $\mathcal{H}^{n+1}(B_1(p))\ge(1-\de)\omega_{n+1}$,
and $M$ is an area-minimizing hypersurface in $B_1(p)$ with $p\in M$ and $\p M\subset\p B_1(p)$, then
\begin{equation}\aligned\label{1pmepHnM0}
(1-\ep)\omega_n\le\mathcal{H}^n(M)\le\f12(1+\ep)(n+1)\omega_{n+1}.
\endaligned
\end{equation}
\end{lemma}

Now let us prove a technical result, which will be needed in Lemma \ref{Rfiso}.
\begin{lemma}\label{iso-sub}
Let $N$ be an $(n+1)$-dimensional complete smooth manifold with Ricci curvature $\ge-nR^{-4}$ on $B_R(p)$ and $\mathcal{H}^{n+1}(B_{R}(p))\ge\omega_{n+1}\left(R-1\right)R^{n}$.
Let $M$ be an area-minimizing hypersurface in $B_{4}(p)$ with $p\in M$, $\p M\subset\p B_4(p)$, and $U$ be an open subset of $M\cap B_{1}(p)$ with
\begin{equation}\aligned\label{pSthomega}
\mathcal{H}^{n-1}(\p U)\le \th\left(\mathcal{H}^{n}(U)\right)^{\f{n-1}n}
\endaligned
\end{equation}
for some constant $\th>0$.
For the sufficiently large $R$, there is an $(n+1)$-dimensional complete manifold $N_*$ with Ricci curvature $\ge-9nR^{-4}$ on $B_{R/4}(\xi)\subset N_*$ and $\mathcal{H}^{n+1}(B_{r}(\xi))\ge\left(1-\f1{\sqrt{R}}\right)\omega_{n+1}r^{n+1}$ for each $0<r\le R/4$,
and an area-minimizing hypersurface $S$ in $B_{1}(\xi)$ with $\xi\in S$, $\p S\subset\p B_{1}(\xi)$, and an open set $V$ in $S\cap B_1(\xi)$ such that $\mathcal{H}^{n}(V\cap B_1(\xi))=\f12\mathcal{H}^n(S\cap B_1(\xi))$, and
$$\mathcal{H}^{n-1}(\p V\cap B_1(\xi))\le n^*\th$$
for some constant $n^*$ depending only on $n$.
\end{lemma}
\begin{proof}
For any $x\in U$, there is a constant $r_x>0$ such that
\begin{equation}\aligned\label{HnSrxhalf}
\mathcal{H}^n(U\cap B_{r_x}(x))=\f12\mathcal{H}^n(M\cap B_{r_x}(x))
\endaligned
\end{equation}
and
\begin{equation}\aligned\label{HnSrxhalf*}
\mathcal{H}^n(U\cap B_{r}(x))>\f12\mathcal{H}^n(M\cap B_{r}(x))\quad \mathrm{for\ all}\ 0<r<r_x.
\endaligned
\end{equation}
Denote $\la=\f1{(n+1)\omega_{n+1}}\mathcal{H}^{n}(U)$ for convenience.  Since $R$ is sufficiently large, from Lemma \ref{epdeHM} we have $\la\le1$ and
\begin{equation}\aligned\label{HnMBrx21}
\mathcal{H}^n(M\cap B_{r}(x))\ge\f{8}{9}\omega_n r^n\qquad \mathrm{for\ any\ ball}\ B_r(x)\subset B_1(p)\ \mathrm{with}\ x\in M.
\endaligned
\end{equation}
With \eqref{HnSrxhalf}-\eqref{HnMBrx21}, for any $0<r\le r_x$ we have
\begin{equation}\aligned
\f49\omega_n r^n\le\f12\mathcal{H}^n(M\cap B_{r}(x))\le\mathcal{H}^n(U\cap B_{r}(x))\le\mathcal{H}^n(U)=(n+1)\omega_{n+1}\la.
\endaligned
\end{equation}
Since an $n$-dimensional unit sphere in $\R^{n+1}$ has the volume $(n+1)\omega_{n+1}=n\omega_n\int_0^{\pi}\sin^{n-1}\th d\th\le 2n\omega_n$, the above inequality implies
\begin{equation}\aligned\label{rx2}
r_x\le \left(\f{9n}2\la\right)^{\f 1{n}}\le 3\la^{\f1n}.
\endaligned
\end{equation}
By the definition of $r_x$, $\overline{U}\subset\cup_{x\in U}B_{r_x}(x)$ clearly.
From Besicovitch covering lemma, we can choose a finite collection of balls $\{B_{r_{p_i}}(p_i)\}_{i=1}^l$ with $\overline{U}\subset \bigcup_{1\le i\le l}B_{r_{p_i}}(p_i)$ such that the number of balls containing any point in $U$ is uniformly bounded, i.e., 
\begin{equation}\aligned
\sharp\{i\in\{1,\cdots,l\}|\ q\in B_{r_{p_i}}(p_i)\}\le n_*\left((n+1)\omega_{n+1}\right)^{\f1n}\qquad \mathrm{for\ any}\ q\in U,
\endaligned
\end{equation}
where $n_*$ is a constant depending only on $n$.
Combining \eqref{pSthomega} and the definition of $\la$, we have
\begin{equation}\aligned
&\sum_{i=1}^l\mathcal{H}^{n-1}(\p U\cap B_{r_{p_i}}(p_i))\le n_*\left((n+1)\omega_{n+1}\right)^{\f1n}\mathcal{H}^{n-1}(\p U)\\
\le n_*\th&\left((n+1)\omega_{n+1}\right)^{\f1n}\left(\mathcal{H}^{n}(U)\right)^{\f{n-1}n}
=n_*\th\la^{-\f1n}\mathcal{H}^n(U)\le n_*\th\la^{-\f1n}\sum_{i=1}^l\mathcal{H}^{n}(U\cap B_{r_{p_i}}(p_i)).
\endaligned
\end{equation}
Then there exists a point $p_*\in\bigcup_i\{p_i\}\subset B_1(p)$ so that
\begin{equation}\aligned\label{HnpSthS}
\mathcal{H}^{n-1}(\p U\cap B_{r_{p_*}}(p_*))\le n_*\th\la^{-\f1n}\mathcal{H}^{n}(U\cap B_{r_{p_*}}(p_*)).
\endaligned
\end{equation}

Let $V^k_s(r)$ denote the volume of a $k$-dimensional geodesic ball with radius $r$ in a space form with constant sectional curvature $-s^2$ for $s\ge0$, i.e., 
\begin{equation}\aligned\label{Vksrsinh}
V^k_s(r)=k\omega_k s^{1-k}\int_0^r\sinh^{k-1}(st)dt.
\endaligned\end{equation}
For any point $x\in B_{1}(p)$ and $0<r<R-1$, Bishop-Gromov volume comparison implies
\begin{equation}\aligned\label{Hn+1BrxRVR}
&\f{\mathcal{H}^{n+1}(B_{r}(x))}{V^{n+1}_{R^{-2}}(r)}\ge\f{\mathcal{H}^{n+1}(B_{R-1}(x))}{V^{n+1}_{R^{-2}}(R-1)}
\ge\f{\mathcal{H}^{n+1}(B_{R-2}(p))}{V^{n+1}_{R^{-2}}(R-1)}\\
\ge&\f{V^{n+1}_{R^{-2}}(R-2)}{V^{n+1}_{R^{-2}}(R-1)}\f{\mathcal{H}^{n+1}(B_{R}(p))}{V^{n+1}_{R^{-2}}(R)}\ge\left(1-\f1{R}\right)\f{V^{n+1}_{R^{-2}}(R-2)}{V^{n+1}_{R^{-2}}(R-1)}\f{\omega_{n+1}R^{n+1}}{V^{n+1}_{R^{-2}}(R)}.
\endaligned\end{equation}
For the sufficiently large $R$ and every $t\in(0,R]$,
\begin{equation}\aligned\label{sinhR-2t}
\sinh^n(R^{-2}t)=R^{-2n}t^n\left(1+O(R^{-1})\right).
\endaligned\end{equation}
From \eqref{Vksrsinh}\eqref{sinhR-2t}, \eqref{Hn+1BrxRVR} infers
\begin{equation}\aligned\label{Hn+1BrxlowR}
\mathcal{H}^{n+1}(B_{r}(x))\ge\left(1-\f1{\sqrt{R}}\right)\omega_{n+1}r^{n+1}\qquad \mathrm{for\ any}\ 0<r<R-1.
\endaligned\end{equation}
Note $r_{p_*}\le3$ from \eqref{rx2} and $\la\le1$. Let $(N_*,\xi)=\f1{r_{p_*}}(N,p_*)$, then $N_*$ is an $(n+1)$-dimensional complete smooth manifold with Ricci curvature $\ge-9nR^{-4}$ on $B_{R/4}(\xi)$. 
From \eqref{Hn+1BrxlowR}, it follows that
\begin{equation}\aligned
\mathcal{H}^{n+1}(B_{r}(\xi))\ge\left(1-\f1{\sqrt{R}}\right)\omega_{n+1}r^{n+1}
\endaligned
\end{equation}
for each $0<r\le R/4$ provided $R$ is sufficiently large.

Noting $B_{r_{p_*}}(p_*)\subset B_{3}(p_*)\subset B_{4}(p)$ for the suitably large $R$.
Let $S=\f1{r_{p_*}}\left(M\cap B_{r_{p_*}}(p_*)\right)$, then $S$ is area-minimizing in $B_{1}(\xi)$ with $\xi\in S$ and $\p S\subset\p B_1(\xi)$.
Let $V=\f1{r_{p_*}}\left(U\cap B_{r_{p_*}}(p_*)\right)$, then from \eqref{HnSrxhalf} we have
$$\mathcal{H}^{n}(V\cap B_1(\xi))=\f1{r_{p_*}^{n}}\mathcal{H}^{n}(U\cap B_{r_{p_*}}(p_*))
=\f1{2r_{p_*}^{n}}\mathcal{H}^{n}(M\cap B_{r_{p_*}}(p_*))=\f12\mathcal{H}^n(S\cap B_1(\xi)).$$
Since $R$ is  sufficiently large, from Lemma \ref{epdeHM} we get
\begin{equation}\aligned\label{Mrp*p*111}
\mathcal{H}^n(M\cap B_{r_{p_*}}(p_*))\le(n+1)\omega_{n+1}r_{p_*}^{n}.
\endaligned
\end{equation}
Combining \eqref{rx2}\eqref{HnpSthS}\eqref{Mrp*p*111}, it follows that
\begin{equation}\aligned
&\mathcal{H}^{n-1}(\p V\cap B_1(\xi))=\f1{r_{p_*}^{n-1}}\mathcal{H}^{n-1}(\p U\cap B_{r_{p_*}}(p_*))
\le\f{n_*\th\la^{-\f1n}}{r_{p_*}^{n-1}}\mathcal{H}^{n}(U\cap B_{r_{p_*}}(p_*))\\
\le&\f{n_*\th\la^{-\f1n}}{r_{p_*}^{n-1}}\mathcal{H}^{n}(M\cap B_{r_{p_*}}(p_*))
\le n_*(n+1)\omega_{n+1}r_{p_*}\th\la^{-\f1n}\le3n_*(n+1)\omega_{n+1}\th.
\endaligned
\end{equation}
We complete the proof.
\end{proof}

When the ambient manifolds are close to Euclidean space in some sense, we have the isoperimetric inequality on area-minimizing hypersurfaces in unit geodesic balls as follows.
\begin{lemma}\label{Rfiso}
For each integer $n\ge2$, there is a constant $\Th\ge4$ depending only on $n$ such that
if $N$ is an $(n+1)$-dimensional smooth complete manifold with Ricci curvature $\ge-nR^{-4}$ on $B_{R}(p)$ and with $\mathcal{H}^{n+1}(B_{R}(p))\ge\omega_{n+1}(R-1)R^n$ for some $R\ge\Th$, and $M$ is an area-minimizing hypersurface in $B_{4}(p)$ with $p\in M$ and $\p M\subset\p B_4(p)$, then for any open set $U\subset M\cap B_{1}(p)$, there holds
\begin{equation}
\mathcal{H}^{n-1}(\p U)\ge \f1{\Th}\left(\mathcal{H}^{n}(U)\right)^{\f{n-1}n}.
\end{equation}
\end{lemma}
\begin{proof}
Let us prove the isoperimetric inequality by contradiction. We suppose that there are a sequence $R_i\rightarrow\infty$, a sequence of $(n+1)$-dimensional smooth complete manifolds $N_i$ with Ricci curvature $\ge-nR_i^{-4}$ on geodesic balls $B_{R_i}(p_i)$, $\mathcal{H}^{n+1}(B_{R_i}(p_i))\ge\omega_{n+1}(R_i-1)R_i^n$,
a sequence of area-minimizing hypersurfaces $M_i\subset B_{4}(p_i)$ with $p_i\in M_i$ and $\p M_i\subset\p B_4(p_i)$
and a sequence of open sets $U_i\subset M_i\cap B_{1}(p_i)$ satisfying
\begin{equation}
\mathcal{H}^{n-1}(\p U_i)\le \f1{R_i}\left(\mathcal{H}^{n}(U_i)\right)^{\f{n-1}n}.
\end{equation}
From Lemma \ref{iso-sub}, there are a sequence of $(n+1)$-dimensional manifolds $Q_i$ with Ricci curvature $\ge-9nR_i^{-4}$ on $B_{R_i/4}(q_i)\subset Q_i$,
$\mathcal{H}^{n+1}(B_{r}(q_i))\ge\left(1-\f1{\sqrt{R_i}}\right)\omega_{n+1}r^{n+1}$ for any $0<r\le R_i/4$,
a sequence of area-minimizing hypersurfaces $S_i$ in $B_{1}(q_i)$ with $q_i\in S_i$ and $\p S_i\subset\p B_{1}(q_i)$, and a sequence of open sets $V_i$ in $S_i\cap B_1(q_i)$ such that $\mathcal{H}^{n}(V_i\cap B_1(q_i))=\f12\mathcal{H}^n(S_i\cap B_1(q_i))$, and
\begin{equation}\aligned\label{TiB1xii}
\mathcal{H}^{n-1}(\p V_i\cap B_1(q_i))\le \f{n^*}{R_i}.
\endaligned
\end{equation}

From Theorem 4.85 in \cite{CC} by Cheeger-Colding,
there is a sequence $0<r_i<R_i$ with $\lim_{i\to\infty}r_i=\infty$ such that
\begin{equation}\label{BRiqiBRi0epi}
\lim_{i\rightarrow\infty}d_{GH}(B_{r_i}(q_i),B_{r_i}(0))=0.
\end{equation}
Let $\Phi_i:\ \overline{B_1(q_i)}\rightarrow B_{1+\tau_i}(0)$ be a $\tau_i$-Gromov-Hausdorff approximation for a sequence of positive numbers $\tau_i\rightarrow0$ defined in \eqref{Phii***}.
In particular, $\Phi_i$ has the Lipschitz constant $\le\sqrt{n+1}(1+1/i)$ for all $i$. 
Up to choosing the subsequence, $\Phi_i(S_i)$ converges in the Hausdorff sense to a closed set $S_\infty$ in $\overline{B_{1}(0)}$ with $0\in S_\infty$.
From Theorem 1.1 in \cite{D}, $S_\infty\cap B_1(0)$ is an area-minimizing hypersurface in $B_1(0)$ with
\begin{equation}\aligned\label{VwqVixii}
\mathcal{H}^n(S_\infty\cap B_t(0))=\lim_{i\rightarrow\infty}\mathcal{H}^n\left(S_i\cap B_t(q_i)\right)\qquad\mathrm{for\ each}\ t\in(0,1). 
\endaligned
\end{equation}
From co-area formula, for any $\ep\in(0,\f12]$ there is a sequence of numbers $t_i\in[1-\ep,1]$ such that $\limsup_{i\rightarrow\infty}\mathcal{H}^{n-1}\left(V_i\cap\p B_{t_i}(q_i)\right)<\infty$.
Combining \eqref{TiB1xii}, we have
\begin{equation}\aligned
&\limsup_{i\rightarrow\infty}\mathbb{M}\left(\p\Phi_i([| V_i\cap B_{t_i}(q_i)|])\right)=\limsup_{i\rightarrow\infty}\mathbb{M}\left(\Phi_i(\p[| V_i\cap B_{t_i}(q_i)|])\right)\\
\le&\limsup_{i\rightarrow\infty}\mathbb{M}\left(\Phi_i([|\p V_i\cap B_{t_i}(q_i)|])\right)+\limsup_{i\rightarrow\infty}\mathbb{M}\left(\Phi_i([| V_i\cap \p B_{t_i}(q_i)|])\right)<\infty,
\endaligned
\end{equation}
and $\limsup_{i\rightarrow\infty}\mathbb{M}\left(\p\Phi_i([| S_i\cap B_{t_i}(q_i)\setminus V_i|])\right)<\infty$ similarly.
From Theorem \ref{FF}, there are integer multiplicity currents $V_\infty, V^*_\infty$ in $B_{t_*}(0)$ with
$t_*=\lim_{i\rightarrow\infty}t_i\in[1-\ep,1]$, such that $\Phi_i([| V_i\cap B_{t_i}(q_i)|])\rightharpoonup V_\infty$ and $\Phi_i([| S_i\cap B_{t_i}(q_i)\setminus V_i|])\rightharpoonup V^*_\infty$ as $i\rightarrow\infty$ up to a choice of the subsequences.
Moreover, $\p V^*_\infty$ is an integer multiplicity $(n-1)$-current from boundary rectifiablity theorem (see \cite{S} for instance).
Since $\Phi_i([| V_i\cap B_{t_i}(q_i)|])+\Phi_i([| S_i\cap B_{t_i}(q_i)\setminus V_i|])=\Phi_i([| S_i\cap B_{t_i}(q_i)|])$, from Lemma \ref{ConvcurrentPhiiMi} we have
\begin{equation}\aligned\label{a111}
\mathcal{H}^n(S_\infty\cap B_{t_*}(0))\le\mathbb{M}\left(V_\infty\right)+\mathbb{M}\left(V^*_\infty\right).
\endaligned
\end{equation}
With \eqref{JPhi}\eqref{VwqVixii} and $\mathcal{H}^{n}(V_i\cap B_1(q_i))=\f12\mathcal{H}^n(S_i\cap B_1(q_i))$, we get
\begin{equation}\aligned\label{a112}
\mathbb{M}\left(V_\infty\right)\le&\liminf_{i\rightarrow\infty}\mathbb{M}\left(\Phi_i([| V_i\cap B_{t_i}(q_i)|])\right)\le\sqrt{e}\liminf_{i\rightarrow\infty}\mathbb{M}\left([| V_i\cap B_{t_i}(q_i)|]\right)\\
\le&\f{\sqrt{e}}2\liminf_{i\rightarrow\infty}\mathcal{H}^n\left(S_i\cap B_1(q_i)\right)=\f{\sqrt{e}}2\mathcal{H}^n(S_\infty\cap B_1(0)),
\endaligned
\end{equation}
and similarly
\begin{equation}\aligned\label{a113}
\mathbb{M}\left(V^*_\infty\right)\le\f{\sqrt{e}}2\mathcal{H}^n(S_\infty\cap B_1(0)).
\endaligned
\end{equation}
Combining \eqref{a111}\eqref{a112}\eqref{a113}, we have
\begin{equation}\aligned\label{a114}
\mathcal{H}^n(S_\infty\cap B_{t_*}(0))-\f{\sqrt{e}}2\mathcal{H}^n(S_\infty\cap B_1(0))\le\mathbb{M}\left(V_\infty\right)\le\f{\sqrt{e}}2\mathcal{H}^n(S_\infty\cap B_1(0)).
\endaligned
\end{equation}
Combining \eqref{TWTjW}\eqref{TiB1xii} and the definition of $\Phi_i$, we have
\begin{equation}\aligned\label{Mtitptiti}
\mathbb{M}\left(\p V_\infty\llcorner B_t(0)\right)\le&\liminf_{i\rightarrow\infty}\mathbb{M}\left(\p\Phi_i([| V_i\cap B_{t_i}(q_i)|])\llcorner B_t(0)\right)\\
=&\liminf_{i\rightarrow\infty}\mathbb{M}\left(\Phi_i([| \p V_i\cap B_{t_i}(q_i)|])\llcorner B_t(0)\right)=0
\endaligned
\end{equation}
for any $t\in(0,t_*)$, which implies $\mathbb{M}\left(\p V_\infty\llcorner B_{t_*}(0)\right)=0$ by letting $t\to t_*$ in \eqref{Mtitptiti}.
However, $\mathbb{M}\left(\p V_\infty\llcorner B_{t_*}(0)\right)=0$ contradicts to \eqref{a114} and the Neumann-Poincar\'e inequality on area-minimizing hypersurfaces in Euclidean space \cite{BG}. This completes the proof.
\end{proof}

Let $N,M$ be defined in Lemma \ref{Rfiso}.
For an open set $U\subset M$ and a closed set $\Om\subset M$ with $(n-1)$-rectifiable $\p\Om$, we call $U$ \emph{an outward minimizing set} in $\Om$, if
\begin{equation}\label{ODEEE*}
\mathcal{H}^{n-1}(\p U\cap \Om)\le\mathcal{H}^{n-1}(\p V\cap \Om)
\end{equation}
for any open set $V$ in $M$ with $U\subset V\subset M$ and $U=V$ in $M\setminus \Om$.
Assume $M\cap B_1(p)\subset \Om$.
Let $U_t=U\cap B_t(p)$, and $\G_{t}=M\cap\p B_t(p)\setminus U$ for any $t\in(0,1]$. Then
\begin{equation}
\mathcal{H}^{n-1}(\p U_t)\le\mathcal{H}^{n-1}(\p B_t(p)),
\end{equation}
which implies
\begin{equation}
\mathcal{H}^{n-1}(\p U_t\cap B_t(p))\le\mathcal{H}^{n-1}(\G_{t}).
\end{equation}
Analog to the proof of Lemma 3.5 in \cite{D}, from the isoperimetric inequality in Lemma \ref{Rfiso}, we have
\begin{equation}
\f1{\Th}\left(\int_0^t\mathcal{H}^{n-1}(\G_{\tau})d\tau\right)^{\f{n-1}{n}}\le \mathcal{H}^{n-1}(\G_{t})+\mathcal{H}^{n-1}(\p U_t\cap B_t(p))\le2\mathcal{H}^{n-1}(\G_{t})
\end{equation}
for any $t\in(0,1]$. 
If $0\in M\setminus U$, then we can solve the above ordinary differential inequality on $t$ and get
\begin{equation}\label{ODEEE}
\mathcal{H}^{n}(M\cap B_t(p)\setminus U)=\int_0^t\mathcal{H}^{n-1}(\G_{\tau})d\tau\ge\left(\f{t}{2n\Th}\right)^{n}\qquad \mathrm{for\ any}\ t\in(0,1].
\end{equation}

In the following text of this section, we will use the constant $\Th$ appeared in Lemma \ref{Rfiso} several times.
Inspired by Lemma 1 in \cite{BG}, we have the following result with the help of the Sobolev inequality obtained in Lemma \ref{Rfiso}.
\begin{lemma}\label{Sob-Coarea}
There are constants $R,\be$ depending only on $n$ with $R\ge \Th>1>\be>0$ such that
if $N$ is an $(n+1)$-dimensional smooth complete manifold with Ricci curvature $\ge-nR^{-4}$ on $B_R(p)$ and with $\mathcal{H}^{n+1}(B_R(p))\ge\omega_{n+1}(R-1)R^n$
and $M$ is an area-minimizing hypersurface in $B_4(p)$ with $p\in M$ and $\p M\subset\p B_4(0))$, 
$S\subset M\cap B_1(p)$ is a non-empty open set with $\mathcal{H}^n(S\cap B_1(p))\le\f12\mathcal{H}^n(M\cap B_1(p))$ and
\begin{equation}\aligned\label{Hn-1pSf12RSbe}
\mathcal{H}^{n-1}(\p S\cap B_1(p))\le\f1{2\Th}\left(\mathcal{H}^n(S\cap B_{\be}(p))\right)^{\f{n-1}n},
\endaligned
\end{equation}
then there is an outward minimizing set $V$ in $M\cap B_1(p)$ with $S\subset V$, $\mathcal{H}^{n-1}(\p V\cap B_1(p))\le\mathcal{H}^{n-1}(\p S\cap B_1(p))$ and
\begin{equation}\aligned
\left(4n\Th\right)^{-n}\le \mathcal{H}^{n}(V)\le\f23\mathcal{H}^{n}(M\cap B_1(p)).
\endaligned
\end{equation}
\end{lemma}
\begin{proof}
For almost all $t\in[\be,1]$, we have
\begin{equation}\aligned
\f{\p}{\p t}\mathcal{H}^{n}(S\cap B_t(p))\ge\mathcal{H}^{n-1}(\p(S\cap B_t(p)))-\mathcal{H}^{n-1}(\p S\cap B_t(p)).
\endaligned
\end{equation}
Combining Lemma \ref{Rfiso} and \eqref{Hn-1pSf12RSbe}, for the suitably large $R\ge \Th$ we get
\begin{equation}\aligned
\f{\p}{\p t}\mathcal{H}^{n}(S\cap B_t(p))\ge&\f1{\Th}\left(\mathcal{H}^n(S\cap B_{t}(p))\right)^{\f{n-1}n}-\f1{2\Th}\left(\mathcal{H}^n(S\cap B_{t}(p))\right)^{\f{n-1}n}\\
=&\f1{2\Th}\left(\mathcal{H}^n(S\cap B_{t}(p))\right)^{\f{n-1}n}.
\endaligned
\end{equation}
By solving the above ordinary differential inequality on $t$, we have
\begin{equation}\aligned
\mathcal{H}^{n}(S\cap B_t(p))\ge\left(\f{t-\be}{2n\Th}\right)^n.
\endaligned
\end{equation}
In particular, $\mathcal{H}^{n}(S)\ge\left(4n\Th\right)^{-n}$ for $\be<\f12$, and then
\begin{equation}\aligned\label{HnSMB1p***}
\left(4n\Th\right)^{-n}\le \mathcal{H}^{n}(S)\le\f12\mathcal{H}^{n}(M\cap B_1(p)).
\endaligned
\end{equation}

By Federer-Fleming compactness theorem, there is an open set $V$ in $M\cap B_1(p)$ with $S\subset V$ such that $V$ is outward minimizing in $M\cap \overline{B_1(p)}$ and
\begin{equation}
\mathcal{H}^{n-1}\left(\p V\cap \overline{B_1(p)}\right)\le\mathcal{H}^{n-1}\left(\p S\cap \overline{B_1(p)}\right).
\end{equation}
Then it follows that $\mathcal{H}^{n-1}(\p(V\setminus S))\le 2\mathcal{H}^{n-1}(\p S\cap B_1(p))$.
With Lemma \ref{Rfiso} and \eqref{Hn-1pSf12RSbe}, for the suitably large $R\ge \Th$ we have
\begin{equation}\aligned
&\left(\mathcal{H}^n(V\setminus S)\right)^{\f{n-1}n}\le \Th \mathcal{H}^{n-1}(\p(V\setminus S))\le 2\Th \mathcal{H}^{n-1}(\p S\cap B_1(p))\\
\le&\left(\mathcal{H}^n(S\cap B_{\be}(p))\right)^{\f{n-1}n}\le\left(\mathcal{H}^n(M\cap B_{\be}(p))\right)^{\f{n-1}n}.
\endaligned
\end{equation}
Combining \eqref{1pmepHnM0}\eqref{HnSMB1p***}, we have
\begin{equation}\aligned
\left(4n\Th\right)^{-n}\le \mathcal{H}^{n}(V)\le\f23\mathcal{H}^{n}(M\cap B_1(p))
\endaligned
\end{equation}
for the suitably large $R\ge \Th$ and the small $\be>0$, which completes the proof.
\end{proof}

\begin{lemma}\label{BRpSEf1R}
There are constants $R,\be$ depending only on $n$ with $R\ge \Th>1>\be>0$ such that
if $N$ is an $(n+1)$-dimensional smooth complete manifold with Ricci curvature $\ge-nR^{-4}$ on $B_R(p)$ and with $\mathcal{H}^{n+1}(B_R(p))\ge\omega_{n+1}(R-1)R^n$
and $M$ is an area-minimizing hypersurface in $B_4(p)$ with $p\in M$, $\p M\subset\p B_4(0)$ and $\mathcal{H}^n(M\cap B_1(p))<(1+1/R)\omega_n$,
then for any open set $U\subset M\cap B_1(p)$, we have
\begin{equation}\label{Hn-1pUB1pR*}
\mathcal{H}^{n-1}(\p U\cap B_1(p))\ge\f1{2\Th}\left(\min\{\mathcal{H}^n(U\cap B_{\be}(p)),\mathcal{H}^n(M\cap B_{\be}(p)\setminus U)\}\right)^{\f{n-1}n}.
\end{equation}
\end{lemma}
\begin{proof}
We prove the inequality \eqref{Hn-1pUB1pR*} by contradiction. Suppose Lemma \ref{BRpSEf1R} fails. From Lemma \ref{Sob-Coarea}, there are a constant $c_0=\left(4n\Th\right)^{-n}\omega_n^{-1}$, a sequence $R_i\rightarrow\infty$, a sequence of $(n+1)$-dimensional smooth complete manifolds $N_i$ with Ricci curvature $\ge-nR_i^{-4}$ on $B_{R_i}(q_i)\subset N_i$, $\mathcal{H}^{n+1}(B_{R_i}(q_i))\ge\omega_{n+1}(R_i-1)R_i^n$,
and a sequence of area-minimizing hypersurfaces $M_i\subset B_4(q_i)$ with $q_i\in M_i$, $\p M_i\subset\p B_4(q_i)$ and $\mathcal{H}^n(M_i\cap B_1(q_i))<(1+1/R_i)\omega_n$,
such that for some sequence of outward minimizing open sets $U_i\subset M_i\cap B_1(q_i)$ with 
\begin{equation}\label{minSiEiSi}
\min\{\mathcal{H}^{n}(U_i),\mathcal{H}^{n}(M_i\cap B_1(q_i)\setminus U_i)\}\ge c_0 \omega_n
\end{equation}
and
\begin{equation}\label{331}
\lim_{i\to\infty}\mathcal{H}^{n-1}(\p U_i\cap B_1(q_i))=0.
\end{equation}

From Theorem 4.85 in \cite{CC},
there is a sequence $0<r_i<R_i$ with $\lim_{i\to\infty}r_i=\infty$ such that
\begin{equation}
\lim_{i\rightarrow\infty}d_{GH}(B_{r_i}(q_i),B_{r_i}(0))=0.
\end{equation}
Let $\mathbf{b}_{i,j}$ be smooth functions  on $\overline{B_1(q_i)}$ such that $|d\mathbf{b}_{i,j}|\le1+\f1i$ as in \eqref{bdij11i}.
Let $\Phi_i=(\mathbf{b}_{i,1},\cdots,\mathbf{b}_{i,n+1}):\ \overline{B_1(q_i)}\rightarrow B_{1+\tau_i}(0)$ be a $\tau_i$-Gromov-Hausdorff approximation for a sequence of positive numbers $\tau_i\rightarrow0$ defined in \eqref{Phii***}.
Up to choosing the subsequence, $\Phi_i(M_i\cap B_1(q_i))$ converges in the Hausdorff sense to a closed set $M_\infty$ in $\overline{B_{1}(0)}$ with $0\in M_\infty$.
From Theorem 1.1 in \cite{D}, $M_\infty$ is an area-minimizing hypersurface in $B_1(0)$, and
\begin{equation}\aligned\label{MinftyBt**}
\mathcal{H}^n(M_\infty\cap B_t(0))=\lim_{i\rightarrow\infty}\mathcal{H}^n\left(M_i\cap B_t(q_i)\right)
\endaligned
\end{equation}
for each $t\in(0,1)$. From $\mathcal{H}^n(M_i\cap B_1(q_i))<(1+1/R_i)\omega_n$, we get $\mathcal{H}^n(M_\infty)\le\omega_n$, which implies the flatness of $M_\infty$ in $B_1(0)$. Without loss of generality, we may assume
\begin{equation}\aligned\label{Minftyx1xn01}
M_\infty\cap B_1(0)=B_1(0^n)\triangleq\{(x_1,\cdots,x_n,0)\in\R^{n+1}|\, x_1^2+\cdots+x_n^2<1\}.
\endaligned
\end{equation}

Let $\pi$ be the projection from $\R^{n+1}$ into $\R^n$ by $\pi(x_1,\cdots,x_n,x_{n+1})=(x_1,\cdots,x_n)$ for any $(x_1,\cdots,x_{n+1})\in\R^{n+1}$.
We define a mapping $\Phi^*_i:\ \overline{B_1(q_i)}\rightarrow \R^{n}$ by 
$$\Phi^*_i=\pi\circ\Phi_i=(\mathbf{b}_{i,1},\cdots,\mathbf{b}_{i,n}).$$
Let $\xi_i$ denote the orientation of $\Phi_i([| M_i\cap B_1(q_i)|])$, and $\th_i$ be the multiplicity function of $\Phi_i([| M_i\cap B_1(q_i)|])$.
Let $\omega$ be a smooth $n$-form defined by $\omega=\varphi dx_1\wedge\cdots\wedge dx_{n}$ for a smooth function $\varphi$ with compact support in $B_1(0)$. Denote $\omega\circ\pi=(\varphi\circ\pi) dx_1\wedge\cdots dx_{n}$ with $\varphi\circ\pi(x_1,\cdots,x_{n+1})=\varphi(x_1,\cdots,x_n,0)$.
Let $\omega_* $ be a smooth $n$-form on $\R^n$ defined by $\omega_*=\varphi(x_1,\cdots,x_n,0) dx_1\wedge\cdots\wedge dx_{n}$. Then $\omega_*\circ\pi=\omega\circ\pi$.
With $\Phi^*_i=\pi\circ\Phi_i$, we have
\begin{equation}\aligned
\Phi^*_i([| M_i\cap B_1(q_i)|])(\omega_*)=&\int_{\mathrm{spt}(\Phi_i([| M_i\cap B_1(q_i)|]))}\th_i\lan\omega_*\circ\pi,\pi_*\xi_i\ran\\
=&\int_{\mathrm{spt}(\Phi_i([| M_i\cap B_1(q_i)|]))}\th_i\lan\omega\circ\pi,\xi_i\ran,
\endaligned
\end{equation}
i.e.,
$\Phi^*_i([| M_i\cap B_1(q_i)|])(\omega_*)=\Phi_i([| M_i\cap B_1(q_i)|])(\omega\circ\pi)$.
With Lemma \ref{ConvcurrentPhiiMi} and \eqref{Minftyx1xn01}, we obtain
\begin{equation}\aligned\label{Phi*Miw***}
\lim_{i\rightarrow\infty}\Phi^*_i([| M_i\cap B_1(q_i)|])(\omega_*)
=&\lim_{i\rightarrow\infty}\Phi_i([| M_i\cap B_1(q_i)|])(\omega\circ\pi)\\
=&[| B_1(0^n)|](\omega\circ\pi)
=[| B_1(0^n)|](\omega_*).
\endaligned
\end{equation}
In other words, $\Phi^*_i([| M_i\cap B_1(q_i)|])\rightharpoonup [| B_1(0^n)|]$ in $B_1(0)$ as $i\rightarrow\infty$.

Since 
$$\mathrm{tr}\left(d\Phi^*_i(d\Phi^*_i)^T\right)=\mathrm{tr}\left((\lan d\mathbf{b}_{i,j},d\mathbf{b}_{i,k}\ran)_{n\times n}\right)=\sum_{j=1}^{n}\lan d\mathbf{b}_{i,j},d\mathbf{b}_{i,j}\ran\le n\left(1+\f1i\right)^2,$$
then the determinant
\begin{equation}\aligned\label{detPhi^*}
\mathrm{det}(d\Phi^*_i)=\sqrt{\mathrm{det}\left(d\Phi^*_i(d\Phi^*_i)^T\right)}\le\left(1+\f1i\right)^n.
\endaligned
\end{equation}
We fix a constant $\ep\in(0,1/2)$ so that $(1-\ep)^2-1+c_0>0$. 
Up to choosing the subsequences, from co-area formula there are a sequence $t_i\in(1-\ep,1)$ with $\lim_{i\rightarrow\infty}t_i=t_*\in[1-\ep,1]$, integer multiplicity currents $T_\infty, T^*_\infty$ in $B_{t_*}(0)$
such that $\Phi^*_i([| U_i\cap B_{t_i}(q_i)|])\rightharpoonup T_\infty$ and $\Phi^*_i([| M_i\cap B_{t_i}(q_i)\setminus U_i|])\rightharpoonup T^*_\infty$ as $i\rightarrow\infty$.
Combining \eqref{minSiEiSi}\eqref{detPhi^*}, we have
\begin{equation}\aligned\label{MPhiiSi}
&\left(\f{i}{i+1}\right)^n\mathbb{M}(\Phi^*_i([| U_i\cap B_{t_i}(q_i)|]))\le\mathcal{H}^n(U_i)\\
=&\mathcal{H}^n(M_i\cap B_1(q_i))-\mathcal{H}^n(M_i\cap B_1(q_i)\setminus U_i)
\le\mathcal{H}^n(M_i\cap B_1(q_i))-c_0\omega_n,\\
\endaligned
\end{equation}
and
\begin{equation}\aligned\label{MPhiiSi*}
&\left(\f{i}{i+1}\right)^n\mathbb{M}(\Phi^*_i([| M_i\cap B_{t_i}(q_i)\setminus U_i|]))\le \mathcal{H}^n(M_i\cap B_1(q_i)\setminus U_i)\\
=&\mathcal{H}^n(M_i\cap B_1(q_i))-\mathcal{H}^n(U_i)\le\mathcal{H}^n(M_i\cap B_1(q_i))-c_0\omega_n.
\endaligned
\end{equation}
Combining \eqref{MinftyBt**}\eqref{MPhiiSi}\eqref{MPhiiSi*}, we deduce $\mathbb{M}(T_\infty)\le(1-c_0)\omega_n$ and $\mathbb{M}(T^*_\infty)\le(1-c_0)\omega_n$.
Since $\Phi^*_i([| U_i\cap B_{t_i}(q_i)|])+\Phi^*_i([| M_i\cap B_{t_i}(q_i)\setminus U_i|])=\Phi^*_i([| M_i\cap B_{t_i}(q_i)|])$, with Lemma \ref{ConvcurrentPhiiMi} and \eqref{Phi*Miw***} we get
$T_\infty+T^*_\infty=[| M_\infty\cap B_{t_*}(0)|]=[| B_{t_*}(0^n)|]$, which implies
$$\omega_nt_*^n=\mathcal{H}^n(B_{t_*}(0^n))\le\mathbb{M}(T_\infty)+\mathbb{M}(T^*_\infty).$$
Combining $\mathbb{M}(T^*_\infty)\le(1-c_0)\omega_n$, we can deduce
\begin{equation}\aligned\label{UplowerMTinfty}
 0<(t_*^n-1+c_0)\omega_n\le\mathbb{M}\left(T_\infty\right)\le(1-c_0)\omega_n.
\endaligned
\end{equation}
From \eqref{331}, we obtain
\begin{equation}\aligned
\mathbb{M}(\p T_\infty\llcorner B_{t}(0))\le&\liminf_{i\rightarrow\infty}\mathbb{M}(\p\Phi^*_i([| U_i\cap B_{t_i}(q_i)|])\llcorner B_{t}(0))\\
=&\liminf_{i\rightarrow\infty}\mathbb{M}(\Phi^*_i([|\p U_i\cap B_{t_i}(q_i)|])\llcorner B_{t}(0))=0
\endaligned
\end{equation}
for any $t\in(0,t_*)$, which contradicts to \eqref{UplowerMTinfty}.
We complete the proof.
\end{proof}

Furthermore, we can remove the additional volume condition of $M$ in Lemma \ref{BRpSEf1R} as follows.
\begin{theorem}\label{BRpNf1RM}
There are constants $R,\be$ depending only on $n$ with $R\ge \Th>1>\be>0$ such that
if $N$ is an $(n+1)$-dimensional smooth complete manifold with Ricci curvature $\ge-nR^{-4}$ on $B_R(p)$ and with $\mathcal{H}^{n+1}(B_R(p))\ge\omega_{n+1}(R-1)R^n$
and $M$ is an area-minimizing hypersurface in $B_4(p)$ with $p\in M$ and $\p M\subset\p B_4(p)$,
then for any open set $U\subset M\cap B_1(p)$, we have
\begin{equation}
\mathcal{H}^{n-1}(\p U\cap B_1(p))\ge\f1{2\Th}\left(\min\{\mathcal{H}^n(U\cap B_{\be}(p)),\mathcal{H}^n(M\cap B_{\be}(p)\setminus U)\}\right)^{\f{n-1}n}.
\end{equation}
\end{theorem}
\begin{proof}
Let us prove it by contradiction. Suppose Theorem \ref{BRpNf1RM} fails. From Lemma \ref{Sob-Coarea}, there are a sequence $R_i\rightarrow\infty$, a sequence of $(n+1)$-dimensional smooth complete manifolds $N_i$ with Ricci curvature $\ge-nR_i^{-4}$ on $B_{R_i}(q_i)\subset N_i$, $\mathcal{H}^{n+1}(B_{R_i}(q_i))\ge\omega_{n+1}(R_i-1)R_i^n$,
and a sequence of area-minimizing hypersurfaces $M_i\subset B_4(q_i)$ with $q_i\in M_i$ and $\p M_i\subset\p B_4(q_i)$ such that for some sequence of outward minimizing open sets $U_i\subset M_i\cap B_1(q_i)$ with $p_i\in\p U_i$ and 
\begin{equation}\label{Uiqi00**}
\min\{\mathcal{H}^{n}(U_i),\mathcal{H}^{n}(M_i\cap B_1(q_i)\setminus U_i)\}\ge  \th_0
\end{equation}
with $\th_0=\left(4n\Th\right)^{-n}$, we have
\begin{equation}\label{pUiqi00**}
\lim_{i\to\infty}\mathcal{H}^{n-1}(\p U_i\cap B_1(q_i))=0.
\end{equation}
Up to a choice of the subsequences, from Theorem 4.4 in \cite{D}, there are an area-minimizing hypersurface $M_\infty\subset B_4(0^{n+1})$ and a closed set $U_\infty\subset M_\infty\cap\overline{B_{1}(0^{n+1})}$ so that $(N_i,q_i)$ converges to Euclidean space $(\R^{n+1},0^{n+1})$ in the pointed Gromov-Hausdorff sense,
$M_i$ converges to a closed set $M_\infty$ and $U_i$ converges to $U_\infty$ in $M_\infty\cap\overline{B_{1}(\xi_\infty)}$ both in the induced Hausdorff sense. 
Moreover, we can assume that $\p U_i$ converges to a closed set $\G_\infty$ in $M_\infty\cap\overline{B_{1}(0^{n+1})}$ in the induced Hausdorff sense.

For any $x\in U_\infty$, we say $x\in U_\infty^*$ if there is a sequence $x_i\in U_i$ with $x_i\to x$ so that 
\begin{equation}\aligned
\limsup_{i\to\infty}\mathcal{H}^{n}\left(U_i\cap B_{\tau}(x_i)\right)>0\qquad \mathrm{for\ any}\  \tau>0.
\endaligned
\end{equation}
On the contrary, if $y\in U_\infty\setminus U_\infty^*$, then for any sequence $y_i\in U_i$ with $y_i\to y$ there is a constant $\tau_y>0$ such that
\begin{equation}\aligned\label{HnUi'ry0}
\lim_{i\to\infty}\mathcal{H}^{n}\left(U_i\cap B_{\tau_y}(y_i)\right)=0.
\endaligned
\end{equation}
Then for any $y'\in B_{\tau_y/2}(y)\cap U_\infty\setminus U_\infty^*$, if $y_i'\in U_i$ with $y_i'\to y'$, then \eqref{HnUi'ry0} implies
\begin{equation}\aligned\label{UiBryyi'}
\lim_{i\to\infty}\mathcal{H}^{n}\left(U_i\cap B_{\tau_y/2}(y_i')\right)\le \lim_{i\to\infty}\mathcal{H}^{n}\left(U_i\cap B_{\tau_y}(y_i)\right)=0.
\endaligned
\end{equation}
Hence, $U_\infty\setminus U_\infty^*$ is open in $U_\infty^*$, i.e., $U_\infty^*$ is a closed set in $U_\infty$. For any $z\in\p U_\infty^*$, from \eqref{UiBryyi'} there is a sequence $z_i\in\p U_i$ with $z_i\to z$, which implies  $\p U_\infty^*\subset\G_\infty$.
For any $\ep>0$, there is a finite open covering $\{W_j\}_{j=1}^{N_\ep}$ of $U_\infty^*$ so that 
\begin{equation}\aligned\label{Sinf*Wj}
\mathcal{H}^{n}\left(U_\infty^*\right)\ge2^{-n}\omega_n\sum_{j=1}^{N_\ep}(\mathrm{diam}W_j)^n-\ep.
\endaligned
\end{equation}
Here, 'diam' denotes the diameter.
Since $U_\infty$ is closed, by finite covering lemma again, there is a sequence of open balls $\{B_{\tau_{x_j}}(x_j)\}_{j=1}^{m}$ with $x_j\in U_\infty\setminus U_\infty^*$ and $B_{\tau_{x_j}}(x_j)\cap U_\infty^*=\emptyset$ such that $$U_\infty\subset\bigcup_{j=1}^{m}B_{\tau_{x_j}}(x_j)\cup\bigcup_{j=1}^{N_\ep}W_j$$
and
\begin{equation}\aligned\label{HnUi'ryj0}
\lim_{i\to\infty}\mathcal{H}^{n}\left(U_i\cap B_{\tau_{x_j}}(x_{i,j})\right)=0\qquad \mathrm{for\ each}\ j=1,\cdots,m,
\endaligned
\end{equation}
where $x_{i,j}$ is a sequence in $U_i$ with $x_{i,j}\to x_j$ as $i\to\infty$. For each $i\ge1$ and $1\le j\le N_\ep$, there is  an open set $W_{i,j}\subset N_i$ with $\mathrm{diam}W_{i,j}=\mathrm{diam}W_j$ such that $W_{i,j}\to W_j$ as $i\to\infty$ and
$$U_i\subset\bigcup_{j=1}^{m}B_{\tau_{x_j}}(x_{i,j})\cup\bigcup_{j=1}^{N_\ep}W_{i,j}.$$
This implies
\begin{equation}\aligned
\mathcal{H}^{n}\left(U_i\right)\le\sum_{j=1}^{m}\mathcal{H}^{n}\left(U_i\cap B_{\tau_{x_j}}(x_{i,j})\right)+\sum_{j=1}^{N_\ep}\mathcal{H}^{n}\left(U_i\cap W_{i,j}\right).
\endaligned
\end{equation}
Clearly, for each $i,j$ there is a ball $B_{\mathrm{diam}W_j}(x_{i,j}')\subset N_i$ with $W_{i,j}\subset B_{\mathrm{diam}W_j}(x_{i,j}')$.
Combining Lemma \ref{epdeHM} and \eqref{Uiqi00**}\eqref{Sinf*Wj}\eqref{HnUi'ryj0}, we get
\begin{equation}\aligned
\th_0\le&
\liminf_{i\to\infty}\mathcal{H}^{n}\left(U_i\right)\le\liminf_{i\to\infty}\sum_{j=1}^{N_\ep}\mathcal{H}^{n}\left(U_i\cap W_{i,j}\right)\\
\le&\liminf_{i\to\infty}\sum_{j=1}^{N_\ep}\mathcal{H}^{n}\left(M_i\cap B_{\mathrm{diam}W_j}(x_{i,j})\right)
\le\liminf_{i\to\infty}\sum_{j=1}^{N_\ep}(n+1)\omega_{n+1}(\mathrm{diam}W_j)^n\\\le&\f{2^n(n+1)\omega_{n+1}}{\omega_n}\left(\mathcal{H}^{n}\left(U_\infty^*\right)+\ep\right).
\endaligned
\end{equation}
Letting $\ep\to0$ in the above inequality implies
\begin{equation}\aligned
\mathcal{H}^{n}\left(U_\infty^*\right)\ge\f{\th_0\omega_n}{2^n(n+1)\omega_{n+1}}.
\endaligned
\end{equation}
Since the singular set of the area-minimizing hypersurface $M_\infty$ has dimension $\le n-7$, there is a point $z\in\p U_\infty^*$ so that $z$ is a regular point of $M_\infty$. 
Noting $\p U_\infty^*\subset\G_\infty$. With \eqref{ODEEE}, there is a sequence $z_i\in\p U_i$ with $z_i\to z$ so that
\begin{equation}\aligned\label{UiBrzi>0}
\limsup_{i\to\infty}\mathcal{H}^{n}\left(U_i\cap B_{\tau}(z_i)\right)>0\qquad \mathrm{for\ any}\  \tau>0,
\endaligned
\end{equation}
and
\begin{equation}\aligned\label{MiUiBrzi>0}
\limsup_{i\to\infty}\mathcal{H}^{n}\left(B_{\tau}(z_i)\cap M_i\setminus U_i\right)>0\qquad \mathrm{for\ any}\  \tau>0.
\endaligned
\end{equation}
However, \eqref{UiBrzi>0}\eqref{MiUiBrzi>0} contradicts to \eqref{pUiqi00**} and Lemma \ref{BRpSEf1R}.
This completes the proof.
\end{proof}

\section{Tangent cones of limits of minimal graphs over manifolds}

From Theorem 30.1 in \cite{S},  the Sobolev inequality holds on integral currents in Euclidean case.
Similarly, the Sobolev inequality \eqref{isoperi} can hold on integral $n$-currents in Riemannian manifolds of Ricci curvature bounded below. For completeness, we give a proof using \eqref{isoperi} and functions of bounded variation as follows.
\begin{lemma}\label{IneqnSi}
Let $\Si$ be an $n$-dimensional manifold with Ricci curvature $\ge-(n-1)\k^2r^{-2}$ on $B_{2r}(p)$ for some $\k\ge0$. 
There is a constant $\a_{n,\k}>0$ depending only on $n,\k$ such that 
if $T$ is an integral $n$-current in $\Si$ with $\mathrm{spt}\p T\subset B_r(p)$, then
\begin{equation}\label{isoperi*}
\f{\a_{n,\k}}r\left(\mathcal{H}^{n}(B_r(p))\right)^{\f1{n}}\left(\mathbb{M}(T)\right)^{\f{n-1}n}\le\mathbb{M}(\p T).
\end{equation}
\end{lemma}
\begin{proof}
Let $T$ be an integral $n$-current with $\mathrm{spt}\p T\subset B_r(p)$. We denote $T=(S,\th,\xi)$ with the $n$-rectifiable set $S$ in $B_r(p)$, the locally $\mathcal{H}^n$-integrable positive integer-valued function $\th$ on $S$ and the orientation $\xi$. 
Let $\phi$ be a function on $B_r(p)$ defined by 
 \begin{equation}\label{PhithchiS}
\phi=\th\chi_{_S},
\end{equation}
where $\chi_{_S}$ is the characteristic function on $S$. Let $\omega_\xi$ be the dual form of $\xi$, which is a smooth $n$-form on $B_r(p)$ with $|\omega_\xi|\equiv1$.
Then
 \begin{equation}\label{MTMT}
\mathbb{M}(T)\le\int_{B_r(p)} \phi=\int_{B_r(p)} \th\chi_{_S}=T(\omega_\xi)\le\sup_{\omega\in\mathcal{D}^n(B_r(p)),|\omega|_{B_r(p)}\le1}T(\omega)=\mathbb{M}(T).
\end{equation}
Let $D$ denote the Levi-Civita connection of $\Si$ and $\mathrm{div}_\Si$ denote the divergence on $\Si$.
For a bounded open set $\Om$ in $\Si$ and a function $f\in L^1(\Om)$, we define
\begin{equation}\aligned
\int_\Om|Df|=\sup\left\{\int_\Om f\mathrm{div}_\Si \mathscr{X}\Big|\, \mathscr{X}\in \G^1_c(T\Si),\, |\mathscr{X}|\le1\ \mathrm{on}\ \Si\right\},
\endaligned
\end{equation}
where $\G^1_c(T\Si)$ denotes the space containing all $C^1$ tangent vector fields on $\Si$ with compact supports (A function $f\in L^1(\Om)$ is said to have bounded variation on $\Om$ if $\int_\Om|Df|<\infty$, see \cite{Gi} for the Euclidean case).

Let $X$ be a smooth vector field with support in $B_r(p)$, and $\omega_X$ be the dual 1-form of $X$. Set $\omega^*_X$ be an $(n-1)$-form defined by $*\omega_X$ with Hodge star $*$. More precisely,  let $\{e_i\}$ denote the local orthonormal basis of the tangent frame, and $\be_i$ denote the dual 1-form of $e_i$. If $X$ can be written as $\sum_{i=1}f_ie_i$ locally for some smooth functions $f_i$ with supports in $B_r(p)$, then
 \begin{equation}
\omega^*_X=\sum_{i=1}^n(-1)^{i-1}f_i\be_1\wedge\cdots\wedge\widehat{\be_i}\wedge\cdots\wedge\be_n.
\end{equation}
From 
\begin{equation}\aligned\label{fTdivSiX}
\int_{B_r(p)}\phi\,\mathrm{div}_\Si X=\int_{S}\th\,\mathrm{div}_\Si X=T(d\omega^*_X)=\p T(\omega^*_X),
\endaligned
\end{equation}
it follows that
\begin{equation}\aligned\label{|DfT|}
\int_{B_r(p)}|D\phi|\le\mathbb{M}(\p T).
\endaligned
\end{equation}
By approaching functions of bounded variation by smooth functions, \eqref{isoperi} holds for $\phi$ defined as \eqref{PhithchiS} (see Theorem 1.28 in \cite{Gi}). With \eqref{MTMT}\eqref{|DfT|} we have
\begin{equation}\aligned
\f{\a_{n,\k}}r\left(\mathcal{H}^{n}(B_r(p))\right)^{\f1{n}}\left(\mathbb{M}(T)\right)^{\f{n-1}n}=&\a_{n,\k}\left(\mathcal{H}^{n}(B_r(p))\right)^{\f1{n}}\left(\int_{B_r(p)}\phi\right)^{\f{n-1}n}\\
\le&\int_{B_r(p)}|D\phi|\le\mathbb{M}(\p T),
\endaligned\end{equation}
which completes the proof.
\end{proof}
\begin{remark}
By the above argument, the Neumann-Poincar\'e inequality \eqref{NPoincare} can also hold on integral $n$-currents in $B_r(p)\subset\Si$.
\end{remark}

Manifolds satisfying \eqref{ConditionSi} may have topological obstruction, but using Lemma \ref{IneqnSi} we can solve Plateau problem for $(n-1)$-rectifiable multiplicity one currents with small mass.
\begin{lemma}\label{ExistcurrentSi}
Let $\Si$ be an $n$-dimensional smooth complete noncompact Riemannian manifold with
$\mathrm{Ric}\ge-(n-1)\k^2$ on $B_3(p)$ for some $\k\ge0$. Suppose $\mathcal{H}^n(B_1(p))\ge v$ for some constant $v>0$.
There is a constant $\de_{n,\k,v}>0$ depending only on $n,\k,v$ such that if
$\G$ is an $(n-1)$-rectifiable current in $B_1(p)\times\R$ with multiplicity one, $\p\G=0$, $\mathbb{M}(\G)\le \de_{n,\k,v}$, then there is an integral current $T$ in $B_{3/2}(p)\times\R$ of multiplicity one with $\p T=\G$ satisfying
\begin{equation}\aligned
\mathbb{M}(T)\le\mathbb{M}(\G)h_\De(\G)+\left(\mathcal{H}^{n}(B_{3/2}(p))\right)^{-\f1{n-1}}\left(\f3{2\a_{n,\k}}\mathbb{M}(\G)\right)^{\f n{n-1}},
\endaligned
\end{equation}
where $h_\De(\G)=\inf\{t_2-t_1|\ \mathrm{spt}\G\subset\Si\times[t_1,t_2]\}$.
\end{lemma}
\begin{proof}
Let $\pi:\Si\times\R\to\Si$ be the projection defined by $\pi(x,t)=x$ for each $(x,t)\in\Si\times\R$.
Then $\pi(\G)$ is an integer multiplicity current with $\p(\pi(\G))=\pi(\p\G)=0$.
We claim that 
\begin{center}
\emph{there is an integer multiplicity current $T'$ in $\overline{B_2(p)}$ such that 
$\p T'=\pi(\G)$ provided $\mathbb{M}(\G)$ is suitably small. }
\end{center}

Let us prove the claim by contradiction. If the claim fails, then we can consider the homological minimizers for the current $\pi(\G)$ in $\overline{B_2(p)}$.
Namely, let $\mathcal{M}$ denote the space including all integral $(n-1)$-currents $T$ in $\overline{B_2(p)}$ such that there is an integral $n$-current $V\neq0$ with spt$V\subset \overline{B_2(p)}$ satisfying $T=\p V+\pi(\G)$. By Federer-Fleming compactness theorem, there is a current $T^*\neq 0$ in $\mathcal{M}$ so that
\begin{equation}\label{MT*G}
\mathbb{M}(T^*)=\inf_{T\in \mathcal{M}}\mathbb{M}(T)\le\mathbb{M}(\pi(\G))\le\mathbb{M}(\G).
\end{equation}
Moreover, $T^*$  has constant multiplicity with $\p T^*=0$.
For any $z\in\mathrm{spt}T^*\cap B_2(p)$, there is a constant $r_z>0$ with $B_{r_z}(z)\subset B_2(p)$ such that $T^*\llcorner B_{r_z}(z)$ is a minimizing current in $B_{r_z}(z)$. Hence, $\mathrm{spt}T^*\cap B_2(p)$ is smooth outside a set of Hausdorff dimension $\le n-8$.

Now  we consider the case of $\mathrm{spt}T^*\cap\p B_2(p)\neq\emptyset$. Let $x\in\mathrm{spt}T^*\cap\p B_2(p)$.
By Hessian comparison theorem, there are constants $\La>0,1>>\de>0$ depending on the sectional curvature of $\Si$ such that the Lipschitz function
$$\La \r_{x}^2-\r_p$$
is convex on $B_\de(x)\cap B_3(p)\setminus \mathcal{C}_p$. 
Here, $\r_{p}$ is the distance function from $p$, $\mathcal{C}_p$ is the cut locus of $\r_p$.
By the definition of (weak) convexity, the function $\La \r_{x}^2-\r_p$ is  convex on $B_\de(x)\cap B_3(p)$.
Hence, $\p B_3(p)$ is an Alexandrov space. Since the tangent cone of an Alexandrov space exists and is unique (see \cite{BBI} for instance),
the tangent cone of $B_3(p)$ at $x$ is a convex closed metric cone $C_{x}$ (with the vertex at 0) contained in a closed half space of $\R^n$.
Let $(T^*_x,0)$ denote a limit of $\f1{r_i}(T^*,x)$ in the current sense for some sequence $r_i\to0$.
Then $T^*_x$ is minimizing in $C_x$ with $0\in\mathrm{spt}T^*_x\cap\p C_x$. So spt$T^*_x$ is flat and $C_x$ is a half space by the maximum principle (see \cite{Wn} for instance). In other words, $\r_{p}$ is differentiable at $x$. Note that the cut locus of $\r_{p}$ is closed. With the uniqueness of geodesic equations, we conclude that $\r_{p}$ is smooth in a small neighborhood of $x$, i.e., $\p B_2(p)$ is smooth in a small neighborhood of $x$.
Since $T^*$ is minimizing in $B_\ep(x)\cap\overline{B_2(p)}$ for a suitably small $\ep>0$, then spt$T^*$ is smooth in a small neighborhood of $x$. In all, spt$T^*$ is smooth outside its singular set in $B_2(p)$.

Let $\De_\Si$ denote the Laplacian of $\Si$ with respect to its Riemannian metric. Let $H_{\p B_2(p)}$ denote the mean curvature of the regular part of $\p B_2(p)$.
Recalling (see \cite{Pl} for instance)
\begin{equation}\label{Hxtge}
H_{\p B_2(p)}=-\De_\Si\r_p\ge -(n-1)\k\coth(2\k)
\end{equation}
on the regular part of $\p B_2(p)$.
Let $H_{T^*}$ denote the mean curvature of spt$T^*$ outside its singular set of Hausdorff dimension $\le n-8$.
Since $T^*$ is locally minimizing in $\overline{B_2(p)}$, then with \eqref{Hxtge} we conclude that 
\begin{equation}\label{Hxtge*00}
|H_{T^*}|\le (n-1)\k\coth(2\k).
\end{equation}
Let $\r_{T^*}$ denote the distance function to spt$T^*$ on $\Si$. 
Let $B_t(T^*)$ denote the $t$-tubular neighborhood of spt$T^*$ for each $t>0$.

At the differentiable point $y$ of $\r_{T^*}$ in $B_3(p)$,
there exist a unique $y'\in \mathrm{spt}T^*$ and a unique non-zero vector $v_y\in\R^{n+1}$ with $|v_y|=\r_{T^*}(y)<1$ such that
$\mathrm{exp}_{y'}(v_y)=y$.
Let $\g_y$ denote the geodesic $\mathrm{exp}_{y'}(tv_y/|v_y|)$ from $t=0$ to $t=|v_y|$. In particular, $\r_{T^*}$ is smooth at $\g_y(t)$ for every $t\in(0,|v_y|]$.
Let $H_y(t), A_y(t)$ denote the mean curvature (pointing out of $\{\r_{T^*}<t\}$), the second fundamental form of the level set $\{\r_{T^*}=t\}$ at $\g_y(t)$, respectively.
From the variational argument (see also Preliminary in \cite{D} for instance),
\begin{equation}\label{HtRic}
\f{\p H_y}{\p t}=|A_y|^2+Ric\left(\dot{\g}_y,\dot{\g}_y\right)\ge\f1{n-1}|H_y|^2-(n-1)\k^2.
\end{equation}
Combining \eqref{Hxtge*00} and \eqref{HtRic}, we get 
\begin{equation}\label{SolveHtRic}
H_y(t)\ge-(n-1)\k\coth(2\k)\qquad \mathrm{for\ each}\ t\in[0,|v_y|].
\end{equation}
In other words,
\begin{equation}\label{Hxtge*}
\De_\Si\r_{T^*}\le (n-1)\k\coth(2\k)
\end{equation}
on $B_{1}(T^*)\setminus  \mathrm{spt}T^*$ outside the cut locus of $\r_{T^*}$.
From the proof of Lemma 7.1 in \cite{D}, the inequality \eqref{Hxtge*} holds on $B_{1}(T^*)\setminus  \mathrm{spt}T^*$ in the distribution sense. 
For each $t\in(0,1]$, integrating by parts using \eqref{Hxtge*} infers
\begin{equation}
\mathcal{H}^{n-1}(\p B_t(T^*))-2\mathcal{H}^{n-1}(\mathrm{spt}T^*)=\int_{B_t(T^*)\setminus  \mathrm{spt}T^*}\De_\Si\r_{T^*}\le n_\k\mathcal{H}^n(B_t(T^*))
\end{equation}
with $n_\k=(n-1)\k\coth(2\k)$, which implies
\begin{equation}\aligned
\f{\p}{\p t}\mathcal{H}^{n}(B_t(T^*))=&\mathcal{H}^{n-1}(\p B_t(T^*))\le2\mathcal{H}^{n-1}(\mathrm{spt}T^*)+n_\k\mathcal{H}^n(B_t(T^*))\\
\le&2\mathbb{M}(T^*)+ n_\k\mathcal{H}^n(B_t(T^*)).
\endaligned
\end{equation}
We solve the above ordinary differential inequality and get
\begin{equation}\aligned
\mathcal{H}^{n}(B_t(T^*))\le\f2{n_\k}\left(e^{n_\k t}-1\right)\mathbb{M}(T^*).
\endaligned
\end{equation}
For any $x_*\in\mathrm{spt}T^*$, we have
\begin{equation}\aligned
\mathcal{H}^{n}(B_{1}(x_*))\le\mathcal{H}^{n}(B_{1}(T^*))\le\f2{n_\k}\left(e^{n_\k}-1\right)\mathbb{M}(T^*).
\endaligned
\end{equation}
With $\mathcal{H}^n(B_1(p))\ge v$ and \eqref{MT*G}, the above inequality is impossible for the suitably small $\mathbb{M}(\G)$. Hence, the claim is true.

Let $\mathcal{M}_{\pi(\G)}$ denote the space including all integral $n$-currents $T$ in $\overline{B_2(p)}$ with $\p T=\pi(\G)$. From the above claim, $\mathcal{M}_{\pi(\G)}$ is non-empty.
From Federer-Fleming compactness theorem, there is an $n$-current $T_*$ in $\mathcal{M}_{\pi(\G)}$ so that
\begin{equation}
\mathbb{M}(T_*)=\inf_{T\in \mathcal{M}_{\pi(\G)}}\mathbb{M}(T).
\end{equation}
Suppose that there is a point $x\in\mathrm{spt}T_*\cap\p B_{3/2}(p)$, then $B_{1/2}(x)\subset \mathrm{spt}T_*$.
Combining  $\mathcal{H}^n(B_1(p))\ge v$ and Bishop-Gromov volume comparison, $\mathcal{H}^n(B_{1/2}(x))$ has a uniform positive lower bound depending only on $n,\k,v$.   
This violates the Sobolev inequality in Lemma \ref{IneqnSi} for the suitably small $\mathbb{M}(\G)$.
So we have $\mathrm{spt}\p T_*\cap\p B_{3/2}(p)=\emptyset$, i.e., $\mathrm{spt}T_*\subset B_{3/2}(p)$. 
For a current $T'$ in $\Si$ and every $t>0$, let $T'\times\{t_1\}$ denote the current in $\Si\times\{t_1\}$ with $\pi(T'\times\{t_1\})=T'$.
There are constants $t_1<t_2$ with $t_2-t_1=h_\De(\G)=\inf\{t-s|\ \mathrm{spt}\G\subset\Si\times[s,t]\}$ such that $\mathrm{spt}\G\subset\Si\times[t_1,t_2]$.
There is a function $u_\G$ such that
$\mathrm{spt}\G=\{(x,u_\G(x))\in\Si\times\R|\ x\in\pi(\G)\}$. 
Let $G_{u_\G}=\{(x,t)\in\Si\times\R|\ x\in\pi(\mathrm{spt}\G),\, t_1\le t\le u_\G(x)\}$, and $[|G_{u_\G}|]$ denote the multiplicity one current associated with $G_{u_\G}$ so that $\p[|G_{u_\G}|]=\G-\p T_*\times\{t_1\}$.
Let
$$T=[|G_{u_\G}|]+T_*\times\{t_1\},$$
then $T$ is a multiplicity one current
with $\p T=\p[|G_{u_\G}|]+\p T_*\times\{t_1\}=\G$ and
\begin{equation}\aligned
\mathbb{M}(T)\le\mathbb{M}([|G_{u_\G}|])+\mathbb{M}(T_*)\le\mathbb{M}(\G)h_\De(\G)+\mathbb{M}(T_*).
\endaligned
\end{equation}
With the isoperimetric inequality in Lemma \ref{IneqnSi} for $T_*$, we complete the proof.
\end{proof}

Let $\Si_i$ be a sequence of $n$-dimensional complete non-compact manifolds with 
\begin{equation}\aligned\label{Rici}
\mathrm{Ric}\ge-(n-1)\k^2\qquad\qquad \mathrm{on}\ \ B_{1+\k'}(p_i)
\endaligned
\end{equation}
for constants $\k\ge0$, $\k'>0$ and with
\begin{equation}\aligned\label{Volvi}
\mathcal{H}^{n}(B_1(p_i))\ge v
\endaligned
\end{equation}
for some constant $v>0$.
Up to choosing the subsequence, we assume that $\overline{B_1(p_i)}$ converges to a metric ball $\overline{B_1(p_\infty)}$ in the Gromov-Hausdorff sense.
For each integer $i\ge1$, let $M_i\subset\Si_i\times\R$ be a minimal graph over $B_1(p_i)$ with $\p M_i\subset\p B_1(p_i)\times\R$.
Let $\pi_i$ be the projection from $\Si_i\times\R$ into $\Si_i$ by $\pi_i(x,t)=x$ for $(x,t)\in\Si_i\times\R$.
Let $\bar{p}_i\in M_i$ with $\pi_i(\bar{p}_i)=p_i$.
Suppose that $M_i$ converges to a closed set $M_\infty\subset \overline{B_1(p_\infty)}\times\R$ in the induced Hausdorff sense, and $\bar{p}_i\to\bar{p}_\infty\in M_\infty$.

For any $y\in M_\infty\cap(B_1(p_\infty)\times\R)$, we define the density function
\begin{equation}\aligned
\Th_{M_\infty,y}(r)=\f{\mathcal{H}^n(M_\infty\cap B_r(y))}{\omega_nr^n}\qquad \mathrm{for\ each}\ r>0.
\endaligned
\end{equation}
Let us prove a monotonicity formula for limits of minimal graphs in metric cones.
\begin{lemma}\label{pM+Minfty}
For any $0<t_1<t_2$ and any sequence $r_i>0$ converging to 0 as $i\to\infty$, we have
\begin{equation}\aligned\label{t1t2HnMpinfty}
\liminf_{i\to\infty}\Th_{M_\infty,y}(t_2r_i)\ge\limsup_{i\to\infty}\Th_{M_\infty,y}(t_1r_i).
\endaligned
\end{equation}
\end{lemma}
\begin{proof}
For any $y\in M_\infty\cap(B_1(p_\infty)\times\R)$, up to choosing the subsequence, we assume $\f1{r_i}(\Si_\infty\times\R,y)$ converges in the Gromov-Hausdorff sense to a metric cone $(C,o)$, and $\f1{r_i}(M_\infty,y)$ converges in the induced Hausdorff sense to a closed set $M_{\infty,y}\subset C$.
There is a sequence of $n$-dimensional complete smooth manifolds $\Si_i'$ satisfying $\mathrm{Ric}\ge-(n-1)\k^2i^{-2}$ on $B_{i}(p_i')$ and $\liminf_{i\to\infty}\mathcal{H}^{n}(B_1(p_i'))>0$ so that $(\Si_i'\times\R,\bar{p}_i')$ converges to $(C,o)$ in the pointed Gromov-Hausdorff sense.
Moreover, there is a sequence of minimal graphs $M_i'\subset B_i(p_i')\times\R$ with $\bar{p}_i'=(p_i',0)\in M_i'$, $\p M_i'\subset\p B_i(p_i')\times\R$ such that $M_i'$ converges to $M_{\infty,y}$ in the induced Hausdorff sense.

Let $\mathcal{S}_C$ denote the singular set of $C$, which has Hausdorff dimension $n-1$ at most.
For any $\ep\in(0,\f12]$, let $\mathcal{S}_{\ep,C}\subset\mathcal{S}_C$ be a closed set in $C$ defined by $\mathcal{S}_{\ep,C}=C\setminus\mathcal{R}_{\ep,C}$ with
\begin{equation}\aligned\label{RSep}
\mathcal{R}_{\ep,C}=\left\{x\in C\Big|\,\sup_{0<s\le r}s^{-1}d_{GH}(B_s(x),B_s(0^{n+1}))<\ep \ \mathrm{for\ some}\ r>0\right\}.
\endaligned
\end{equation}
For every $\tau>0$ and $x\in \p B_\tau(o)$, let $\{tx|\, t\ge0\}$ be the ray through $x$ and starting from $o$.
For any $s>0$, we let $sx$ denote the point in $\p B_{s\tau}(o)\cap\{tx|\, t\ge0\}$.
Then we can rewrite $C$ by
$$C=\{sx|\, x\in \p B_1(o),\, s\ge0\}=\{sx|\, x\in \p B_\tau(o),\, s\ge0\}$$
for any $\tau>0$.
Let $C_{y,t}$ denote the truncated metric cone over $M_{\infty,y}\cap\p B_{t}(o)$ in $\overline{B_{t}(o)}$ for each $t>0$, namely,
$$C_{y,t}=\{sx\in C|\, x\in M_{\infty,y}\cap\p B_{t}(o),\, 0\le s\le 1\}.$$

Let $\mathcal{M}^*(\cdot,\cdot), \mathcal{M}_*(\cdot,\cdot)$ denote the upper and the lower $n$-dimensional Minkowski contents defined by
\begin{equation}\aligned
\mathcal{M}^*(K,\Om)&=\limsup_{\de\to0}\f1{2\de}\mathcal{H}^{n+1}\left(\Om\cap B_\de(K)\setminus K\right)\\\mathcal{M}_*(K,\Om)&=\liminf_{\de\to0}\f1{2\de}\mathcal{H}^{n+1}\left(\Om\cap B_\de(K)\setminus K\right)
\endaligned
\end{equation}
for any $K,\Om\subset C$.
From the co-area formula for non-collapsed metric cones (see Proposition 7.6 in \cite{Honda} by Honda for instance), for any $t>s>0$ and any $\de>0$ we have
\begin{equation}\aligned\label{BdeMinftyy}
\mathcal{H}^{n+1}\left(B_\de(M_{\infty,y})\cap B_{t}(o)\right)-\mathcal{H}^{n+1}\left(B_\de(M_{\infty,y})\cap B_{s}(o)\right)=\int_{s}^{t}\mathcal{H}^n\left(B_{\de}(M_{\infty,y})\cap\p B_\tau(o)\right)d\tau.
\endaligned
\end{equation}
With Fatou lemma, it follows that
\begin{equation}\aligned\label{M*inftyBrsto**}
\mathcal{M}^*\left(M_{\infty,y},B_{t}(o)\right)-\mathcal{M}_*\left(M_{\infty,y},B_{s}(o)\right)\ge\int_s^t\left(\liminf_{\de\to0}\f1{2\de}\mathcal{H}^n\left(B_{\de}(M_{\infty,y})\cap\p B_\tau(o)\right)\right)d\tau.
\endaligned
\end{equation}
Combining Corollary 5.5 in \cite{D} and \eqref{M*inftyBrsto**}, we have
\begin{equation}\aligned\label{HnMinftyM*rs}
\mathcal{H}^n\left(M_{\infty,y}\cap \overline{B_{t}(o)}\right)-\mathcal{H}^n\left(M_{\infty,y}\cap B_{s}(o)\right)\ge\int_s^t\left(\liminf_{\de\to0}\f1{2\de}\mathcal{H}^n\left(B_{\de}(M_{\infty,y})\cap\p B_\tau(o)\right)\right)d\tau,
\endaligned
\end{equation}
which implies 
\begin{equation}\aligned\label{deto0BdeMinfyro}
\liminf_{\de\to0}\f1{2\de}\mathcal{H}^n\left(B_{\de}(M_{\infty,y})\cap\p B_r(o)\right)<\infty
\endaligned
\end{equation}
for almost every $r>0$.

Let $\mathcal{C}$ denote a 2-dimensional metric cone with nonnegative curvarture outside its isolated singular point at the vertex, $\mathcal{C}'$ denote an area-minimizing cone in $\mathcal{C}\times\R^{n-1}$ such that $\mathcal{C}'$ splits off a Euclidean factor $\R^{n-2}$ isometrically. Let $\mathcal{S}_{\mathcal{C}}$ denote the singular set of $\mathcal{C}\times\R^{n-1}$. By dimension reduction argument, clearly $\mathcal{S}_{\mathcal{C}}\cap\mathcal{C}'\cap(\mathcal{C}\times\R^{n-2})$ has Hausdorff dimension $\le n-3$ unless $\mathcal{C}'\subset\mathcal{C}\times\R^{n-2}$.
Therefore, if \eqref{deto0BdeMinfyro} holds for some $r>0$, then the situation $\mathcal{C}'\subset\mathcal{C}\times\R^{n-2}$ will not occur.
As a consequence, from the proof of Lemma 6.5 in \cite{D}, $\mathcal{S}_{\ep,C}\cap M_{\infty,y}\cap\p B_r(o)$ has Hausdorff dimension $\le n-3$ by dimension reduction argument again. 
Hence, for every $\a\in(0,1)$ there is a sequence of geodesic balls $\{B_{\a_k}(x_{\a,k})\}_{k=1}^{N_{\a,\ep}}\subset C$ with $\a_k<\a$ so that $\mathcal{S}_{\ep,C}\cap C_{y,r}\subset\bigcup_{k=1}^{N_{\a,\ep}}B_{\a_k}(x_{\a,k})$ and
\begin{equation}\aligned\label{n-3/2}
\sum_{k=1}^{N_{\a,\ep}}\a_k^{n-3/2}<\a .
\endaligned
\end{equation}
Let $x_{i,\a,k}\in \Si_i'\times\R$ so that $x_{i,\a,k}\to x_{\a,k}$ as $i\to\infty$.
Denote $U_{\a,\ep}=\bigcup_{k=1}^{N_{\a,\ep}}B_{\a_k}(x_{\a,k})\subset C$, and $U_{i,\a,\ep}=\bigcup_{k=1}^{N_{\a,\ep}}B_{\a_k}(x_{i,\a,k})\subset \Si_i'\times\R$.
With co-area formula, Lemma \ref{epdeHM} and \eqref{n-3/2}, there is a sequence $r_i>r$ with $r_i\to r$ so that 
\begin{equation}\aligned\label{MiBriUi}
\lim_{i\to\infty}\mathcal{H}^{n-1}(M_i'\cap\p B_{r_i}(\bar{p}_i')\cap U_{i,i^{-1},2\ep})=0.
\endaligned
\end{equation}

For any $t>0$ and any $\a>0$, let
$$C_{\a,\ep,y,t}=\left\{sx\in C\Big|\, x\in M_{\infty,y}\cap\p B_{t}(o),\, \ep\le s\le 1\right\}\setminus U_{\a,\ep},$$
and
$$C_{\ep,y,t}=\left\{sx\in C\Big|\, x\in M_{\infty,y}\cap\p B_{t}(o),\, \ep\le s\le 1\right\}\setminus \mathcal{S}_{\ep,C}.$$
From homeomorphism in Theorem A.1.8 in \cite{CCo1}, for any $r'>r>0$ and any integer $i\ge1$
there is a smooth hypersurface $S_{i,\a,\ep}$ (outside a closed set of dimension $\le n-7$) in $B_{r'}(\bar{p}_i')\setminus \overline{U_{i,\a,\ep}}$ with $\p S_{i,\a,\ep}\subset\p\left(B_{r'}(\bar{p}_i')\setminus \overline{U_{i,\a,\ep}}\right)$ such that $S_{i,\a,\ep}=M_i'$ in $B_{r'}(\bar{p}_i')\setminus \left(B_r(\bar{p}_i')\cup\overline{U_{i,\a,\ep}}\right)$ and $S_{i,\a,\ep}\cap B_r(\bar{p}_i')$ converges to a closed set in $C_{\a,\ep,y,r}$ as $i\to\infty$.
Up to considering a $s$-tubular neighborhood of $S_{i,\a,\ep}$, with co-area formula and \eqref{n-3/2}, there are a sequence of numbers $t_{i,\a}\in[1,2]$, a sequence of $n$-rectifiable sets $S'_{i,\a,\ep}\subset B_{r'}(\bar{p}_i')\setminus \overline{U^*_{i,\a ,\ep}}$ with $U^*_{i,\a ,\ep}=\bigcup_{k=1}^{N_{\a,\ep}}B_{t_{i,\a}\a_k}(x_{i,\a,k})$ such that $\p S_{i,\a,\ep}'\setminus\overline{U^*_{i,\a ,\ep}}=M_i'\cap\p B_{r_i}(\bar{p}_i')\setminus \overline{U^*_{i,\a ,\ep}}$, 
\begin{equation}\aligned\label{SiUiti}
\lim_{i\to\infty,\a\to0}\mathcal{H}^{n-1}\left(\p S_{i,\a,\ep}'\cap \overline{U^*_{i,\a ,\ep}}\right)=0,
\endaligned
\end{equation}
and $S'_{i,\a,\ep}$ converges as $i\to\infty$ to a closed set in $C_{\a,\ep,y,r}$ (up to choosing subsequences).
Denote $S'_{i,\ep}=S'_{i,i^{-1},\ep}$ in $B_{r'}(\bar{p}_i')\setminus \overline{U^*_{i,i^{-1},\ep}}$ for short.
From \eqref{MiBriUi}\eqref{SiUiti}, we could further assume that $S'_{i,\ep}$ converges to a closed set in $C_{\ep,y,r}$, and
\begin{equation}\aligned\label{MiSiUti}
\lim_{i\to\infty}\left(\mathcal{H}^{n-1}(M_i'\cap\p B_{r_i}(\bar{p}_i')\cap U^*_{i,i^{-1},\ep})+\mathcal{H}^{n-1}\left(\p S_{i,\ep}'\cap\overline{U^*_{i,i^{-1},\ep}}\right)\right)=0.
\endaligned
\end{equation}

Let $\G_{i,\ep}$ be an $(n-1)$-current defined by $\G_{i,\ep}=\p[|M_i'\cap\p B_{r_i}(\bar{p}_i')|]-\p[|S_{i,\ep}'|]$. Then $\p\G_{i,\ep}=0$ and $\lim_{i\to\infty}\mathbb{M}(\G_{i,\ep})=0$ from \eqref{MiSiUti}.
From Lemma \ref{ExistcurrentSi}, for the suitably large $i>0$ there is an integral current $T_{i,\ep}$ of multiplicity one with $\p T_{i,\ep}=\G_{i,\ep}$ such that $\lim_{i\to\infty}\mathbb{M}(T_{i,\ep})=0$. Moreover, spt$T_{i,\ep}$ converges to a closed set in $\mathcal{S}_{\ep,C}$ up to a choice of a subsequence.
Let $T^*_{i,\ep}=[|S_{i,\ep}'|]-T_{i,\ep}+[|M_i'|]\llcorner(B_{r'}(\bar{p}_i')\setminus B_{r_i}(\bar{p}_i'))$. Then  
$$\p(T^*_{i,\ep}\llcorner B_{r_i}(\bar{p}_i'))=\p[|S_{i,\ep}'|]-\p T_{i,\ep}=\p[|M_i'\cap\p B_{r_i}(\bar{p}_i')|]+\G_{i,\ep}-\G_{i,\ep}=\p[|M_i'\cap\p B_{r_i}(\bar{p}_i')|].$$ 
After a suitable perturbation, for large $i$ there is a sequence of hypersurfaces $S_{i,\ep}\subset B_{i}(\bar{p}_i')$ such that
$S_{i,\ep}\setminus M_i'$ is smooth, embedded, $S_{i,\ep}=M_i'$ outside $B_r(\bar{p}_i')$, and $S_{i,\ep}\cap B_r(\bar{p}_i')$ converges to a closed set $S_{\infty,\ep}$ in $C$ with $C_{y,r}\subset S_{\infty,\ep}$ and $S_{\infty,\ep}\setminus C_{y,r}\subset \mathcal{S}_{\ep,C}$.
For each $i$, let $\mathcal{R}_{M_i'}$ denote the regular part of $M_i'$, and $\mathcal{R}_{i,\ep}=(S_{i,\ep}\setminus M_i')\cup(S_{i,\ep}\cap \mathcal{R}_{M_i'})$.
Clearly, $\mathcal{R}_{i,\ep}$ is two-sided and connected. Let $\mathbf{n}_{i,\ep}$ be the unit normal vector field to $\mathcal{R}_{i,\ep}$.
We define two open sets $B^\pm_{r}(S_{i,\ep})$ by
\begin{equation}\aligned\nonumber
B^\pm_{r}(S_{i,\ep})=B_r(S_{i,\ep})\cap\{\mathrm{exp}_{x}(\pm s\mathbf{n}_{i,\ep}(x))|\ x\in \mathcal{R}_{i,\ep},\, 0<s<r\}.
\endaligned
\end{equation}
Up to choosing the subsequences, we can assume that $B^+_{1/j}(S_{i,\ep})$, $B^+_{1/j}(S_{i,\ep})$ converge as $i\to\infty$ in the induced Hausdorff sense to two closed sets $B^+_{j,S}$, $B^-_{j,S}$ in $C$, respectively. 
Combining homeomorphism in Theorem A.1.8 in \cite{CCo1}, $C_{y,r}\subset S_{\infty,\ep}$ and $S_{\infty,\ep}\setminus C_{y,r}\subset\mathcal{S}_{\ep,C}$, for any $\tau>0$ there is an integer $j_0>0$ such that 
$$B^+_{j,S}\cap B^-_{j,S}\subset S_{\infty,\ep}\cup B_\tau(\mathcal{S}_{\ep,C})$$ 
for all integers $j\ge j_0$. 
From Theorem 1.3 in \cite{CN} by Cheeger-Naber, there is a constant $c_{n,\ep,v,r}>0$ such that 
\begin{equation}\aligned\nonumber
\mathcal{H}^{n+1}(B_s(\mathcal{S}_{\ep,C})\cap B_{2r}(o))\le c_{n,\ep,v,r}s^{2-\ep}
\endaligned
\end{equation}
for any small $s\in(0,1]$. Then we can follow the proof of Proposition 4.2 in \cite{D}, and get
\begin{equation}\aligned\label{M*CMINFtoge}
\mathcal{M}_*(M_{\infty,y},B_r(o))\le\mathcal{M}_*(C_{y,r},B_r(o)) \qquad \mathrm{for\ almost\ every}\ r>0.
\endaligned
\end{equation}

From the cone property of $C_{y,t}$, there is a general function $\psi_\de$ with $\lim_{\de\rightarrow0}\psi_\de=0$ such that for any $\tau\in(0,t)$
\begin{equation}\aligned
\mathcal{H}^{n}\left(B_\de(C_{y,t})\cap\p B_\tau(o)\right)\le& (1+\psi_\de)\mathcal{H}^{n}\left(B_\de(C_{y,t}\cap\p B_\tau(o))\cap\p B_\tau(o)\right)\\
=& (1+\psi_\de)\left(\f\tau t\right)^n\mathcal{H}^{n}\left(B_{\de t/\tau}(C_{y,t}\cap\p B_t(o))\cap\p B_t(o)\right).
\endaligned
\end{equation}
Hence, for any $\ep>0$ there is a sequence $\de_i\to0$ such that
\begin{equation}\aligned
\mathcal{H}^{n}\left(B_{\de_i}(C_{y,t})\cap\p B_\tau(o)\right)\le2(1+\ep)\de_i\left(\left(\f\tau t\right)^{n-1}\mathcal{M}_{*}\left(C_{y,t}\cap\p B_t(o),\p B_t(o)\right)+\ep\right).
\endaligned
\end{equation}
Then from the co-area formula, we have
\begin{equation}\aligned
\mathcal{H}^{n+1}&\left(B_{\de_i}(C_{y,t})\cap B_t(o)\right)=\int_0^t\mathcal{H}^{n}\left(B_{\de_i}(C_{y,t})\cap\p B_\tau(o)\right)d\tau\\
\le&2(1+\ep)\de_i\int_0^t\left(\left(\f\tau t\right)^{n-1}\mathcal{M}_{*}\left(C_{y,t}\cap\p B_t(o),\p B_t(o)\right)+\ep\right)d\tau\\
=&2(1+\ep)\de_i\left(\f tn\mathcal{M}_{*}\left(C_{y,t}\cap\p B_t(o),\p B_t(o)\right)+\ep t\right),
\endaligned
\end{equation}
which implies
\begin{equation}\aligned\label{M*CMINFtole}
\mathcal{M}_*\left(C_{y,t}, B_t(o)\right)\le\f tn\mathcal{M}_{*}\left(C_{y,t}\cap\p B_t(o),\p B_t(o)\right).
\endaligned
\end{equation}
Combining \eqref{M*CMINFtoge}\eqref{M*CMINFtole} and Corollary 5.5 in \cite{D}, we have
\begin{equation}\aligned\label{MinftytnMinfty}
\mathcal{M}_{*}\left(C_{y,t}\cap\p B_t(o),\p B_t(o)\right)\ge\f nt\mathcal{M}_*(M_{\infty,y},B_t(o))\ge\f nt\mathcal{H}^n\left(M_{\infty,y}\cap B_{t}(o)\right)
\endaligned
\end{equation}
for all $t>0$. 
From \eqref{HnMinftyM*rs}\eqref{MinftytnMinfty}, we have
\begin{equation}\aligned
&\mathcal{H}^{n}\left(M_{\infty,y}\cap \overline{B_{t}(o)}\right)-\mathcal{H}^{n}\left(M_{\infty,y}\cap B_{s}(o)\right)\\
\ge&\int_s^t\mathcal{M}_*\left(M_{\infty,y}\cap\p B_{\tau}(o),\p B_{\tau}(o)\right)d\tau\\
=&\int_s^{t}\mathcal{M}_{*}\left(C_{y,\tau}\cap\p B_\tau(o),\p B_\tau(o)\right)d\tau\\
\ge&\int_s^{t}\f n\tau\mathcal{H}^{n}\left(M_{\infty,y}\cap B_{\tau}(o)\right)d\tau.
\endaligned
\end{equation}
Noting that $\mathcal{H}^{n}\left(M_{\infty,y}\cap B_{t}(o)\right)$ is monotonic nondecreasing and left continuous. 
Letting $t<r$ and $t\to r$ implies
\begin{equation}\aligned\label{HnMBroapp}
\mathcal{H}^{n}\left(M_{\infty,y}\cap B_{r}(o)\right)-\mathcal{H}^{n}\left(M_{\infty,y}\cap B_{s}(o)\right)\ge\int_s^{t}\f n\tau\mathcal{H}^{n}\left(M_{\infty,y}\cap B_{\tau}(o)\right)d\tau.
\endaligned
\end{equation}
Denote $\phi(t)=t^{-n}\mathcal{H}^{n}\left(M_{\infty,y}\cap B_t(o)\right)$ for each $t>0$. 
From \eqref{HnMBroapp}, 
 \begin{equation}\aligned
\liminf_{s\to t}\f{\phi(t)-\phi(s)}{t-s}=&-\f n{t}\phi(t)+t^{-n}\liminf_{s\to t}\f{\mathcal{H}^{n}\left(M_{\infty,y}\cap B_{t}(o)\right)-\mathcal{H}^{n}\left(M_{\infty,y}\cap B_{s}(o)\right)}{t-s}\\
\ge&-\f n{t}\phi(t)+t^{-n}\liminf_{s\to t}\f{1}{t-s}\int_s^{t}\f n\tau\mathcal{H}^{n}\left(M_{\infty,y}\cap B_{\tau}(o)\right)d\tau\\
\ge&-\f n{t}\phi(t)+t^{-n}\f nt\mathcal{H}^{n}\left(M_{\infty,y}\cap B_t(o)\right)=0.
\endaligned
\end{equation}
Hence, for any $t>0$ and small $\ep>0$ there is a small constant $0<\de_{t,\ep}<t$ so that $\phi(\tau)\le\phi(t)+\ep(t-\tau)$ for all $\tau\in(t-\de_{t,\ep},t)$, and $\phi(\tau)\ge\phi(t)+\ep(t-\tau)$ for all $\tau\in(t,t+\de_{t,\ep})$.
For any $r>s>0$, from the open covering $\bigcup_{t>0}(t-\de_{t,\ep},t+\de_{t,\ep})$ there is a sequence of numbers $s=\tau_0<\tau_1<\cdots<\tau_m<\tau_{m+1}=r$ with $\tau_i\in(\tau_{i-1},\tau_{i-1}+\de_{\tau_{i-1},\ep})$ and $\tau_i\in(\tau_{i+1}-\de_{\tau_{i+1},\ep},\tau_{i+1})$ for all $i=1,\cdots,m$.
Then
\begin{equation}\aligned
\phi(r)-\phi(s)=\sum_{i=0}^m\left(\phi(t_{i+1})-\phi(t_i)\right)\ge\ep\sum_{i=0}^{m-1}(t_i-t_{i+1})-\ep(r-t_m)=-\ep(r-s),
\endaligned
\end{equation}
which implies $\phi(r)\ge\phi(s)$ by letting $\ep\to0$.
From Theorem 5.4 in \cite{D}, by the definitions of $M_{\infty,y}$ and $\phi$, it follows that
\begin{equation}\aligned
\liminf_{i\to\infty}\f{\mathcal{H}^n(M_\infty\cap B_{t_2r_i}(y))}{(rr_i)^n}\ge\limsup_{i\to\infty}\f{\mathcal{H}^n(M_\infty\cap B_{t_1r_i}(y))}{(sr_i)^n}
\endaligned
\end{equation}
fot all $t_2>r>s>t_1>0$, namely, 
\begin{equation}\aligned
\left(\f{t_2}{r}\right)^n\liminf_{i\to\infty}\Th_{M_\infty,y}(t_2r_i)\ge\left(\f{t_1}{s}\right)^n\limsup_{i\to\infty}\Th_{M_\infty,y}(t_1r_i).
\endaligned
\end{equation}
Letting $r\to t_2$ and $s\to t_1$ in the above inequality completes the proof.
\end{proof}

Now let us prove that the tangent cones of $M_\infty$ are metric cones using Lemma \ref{pM+Minfty} (compared with Theorem 6.2 in \cite{D}).
\begin{theorem}\label{CovconeC}
For any $y_\infty\in M_\infty\cap(B_1(p_\infty)\times\R)$, there is a sequence $r_j>0$ with $r_j\rightarrow0$ as $j\to\infty$ such that $\f1{r_{j}}(B_1(p_\infty)\times\R,y_\infty)$ converges to a metric cone $(C,o)$ in the pointed Gromov-Hausdorff sense,
and $\f1{r_{j}}(M_\infty,y_\infty)$ converges in the induced Hausdorff sense to a metric cone $(C',o)$ with $C'\subset C$.
\end{theorem}
\begin{proof}
Combining the proof of Theorem 6.2 in \cite{D} and Lemma \ref{pM+Minfty}, there is a sequence $r_j\rightarrow0^+$ so that $\f1{r_j}(B_1(p_\infty)\times\R,y_\infty)$ converges as $j\to\infty$ to a metric cone $(C,o)$ in the pointed Gromov-Hausdorff sense, and
$\f1{r_j}\left(M_{\infty},y_\infty\right)$ converges in the induced Hausdorff sense to $(M_{y_\infty},o)$ for a closed set $M_{y_\infty}\subset C$ satisfying
\begin{equation}\aligned\label{620epj}
\mathcal{H}^n\left(M_{y_\infty}\cap B_{r}(o)\right)=\de_{n,v} r^n \qquad \mathrm{for\ all}\ r>0.
\endaligned
\end{equation}
Here, $\de_{n,v}$ is a positive constant.

We may assume that $C$ is not a standard Euclidean space $\R^{n+1}$, or else we have finished the proof.
For each $y\in C\setminus\{o\}$, let $l_y$ denote the (radial) geodesic ray through $y$ starting from $o$. For every $0<s<d(o,y)$, let $l_y(s)$ denote a point in $l_y$ with $d(y,l_y(s))=s$.
We define the angle $\th_y$ between $l_y$ and $M_{y_\infty}$ by
\begin{equation}\aligned\label{thyangle}
\th_y=\liminf_{s\to0}\inf_{x\in M_{y_\infty}\cap\p B_s(y)}|\angle xyl_y(s)|,
\endaligned
\end{equation}
where $\angle xyl_y(s)$ is a function satisfying
$\cos\angle xyl_y(s)=1-\f1{2s^2}d^2(x,l_y(s))$ (see the formula (6) in \cite{CN1} for the general definition and its property).

Let $\mathcal{R}_C$ denote the regular part of $C$.
If $y\in M_{y_\infty}\cap \mathcal{R}_C$
satisfies $\th_y=0$, then there is a sequence of numbers $s_i\to0$ and a sequence of points $x_i\in M_{y_\infty}\cap\p B_{s_i}(y)$ so that $\lim_{i\to\infty}\angle x_iyl_y(s_i)=0$, which implies
\begin{equation}\aligned\label{sixilysi}
\lim_{i\to\infty}s_i^{-1}d(x_i,l_y(s_i))=0.
\endaligned
\end{equation}
From Theorem 6.2 in \cite{D},  up to a choice of the subsequence of $s_i$, $\f1{s_{i}}(M_{y_\infty}, y)$ converges in the
induced Hausdorff sense to $(C_y,0)$ in $\R^{n+1}$, wherer $C_y$ is an area-minimizing cone in $\R^{n+1}$.
With \eqref{sixilysi}, there is a connected curve $\g_{y,i}$ in $M_{y_\infty}\cap B_{s_i}(y)$ with two endpoints in $\p B_{s_i}(y)\cap\mathcal{R}_C$ so that
\begin{equation}\aligned\label{isidlytausi}
\lim_{i\to\infty}\sup_{x\in\g_{y,i}}s_i^{-1}d(x,l_y)=0.
\endaligned
\end{equation}

Suppose $\th_y=0$ for all $y\in M_{y_\infty}\cap \mathcal{R}_C$. Noting that $C\setminus \mathcal{R}_C$ has Hausdorff dimension $\le n-1$.
From \eqref{isidlytausi}, for any $\ep>0$, there are a constant $0<s_y<\ep$ and a connected curve $\g_{y}$ in $M_{y_\infty}\cap B_{s_y}(y)$ with two endpoints in $\p B_{s_y}(y)\cap\mathcal{R}_C$ so that
\begin{equation}\aligned
\sup_{x\in\g_y}d(x,l_y)<\ep s_y.
\endaligned
\end{equation}
Let $y_i$ be one endpoint of $\g_{y_{i-1}}$ for $i=1,\cdots,m$ with $y_0=y$. 
Without loss of generality, we assume that $d(y_i,o)\le d(y_{i-1},o)$ for any $i=1,\cdots,m$.
There is a point $z_i\in l_{y_{i}}$ with $d(z_i,o)\le d(y_i,o)$ so that $d(y_{i+1},l_{y_i})=d(y_{i+1},z_i)$ for $i=0,\cdots,m-1$. Since $C$ is a metric cone, it's clear that $d(z_{i},l_y)\le d(y_{i},l_y)$.
Then
\begin{equation}\aligned
d(y_m,l_y)\le& d(y_m,z_{i-1})+d(z_{m-1},l_y)\le d(y_m,z_{m-1})+d(y_{m-1},l_y)\\
\le&\cdots\le\sum_{y=0}^{m-1} d(y_{i+1},l_{y_i})<\ep\sum_{y=0}^{m-1} s_{y_i}=\ep d(y_m,y).
\endaligned
\end{equation}
Since $\ep$ in the above inequality can be arbitrarily small, we conclude that $M_{y_\infty}$ is a metric cone in $C$ with the vertex at $o$.
In other words, if $M_{y_\infty}$ is not a metric cone in $C$ with the vertex at $o$, then there is a point $x\in M_{y_\infty}\cap \mathcal{R}_C$ with $\th_x>0$.

From Theorem 6.2 in \cite{D}, there is a sequence $t_i\to0$ so that $\f1{t_i}(M_{y_\infty}, x)$ converges in the
induced Hausdorff sense to $(C_x,0)$ in $\R^{n+1}$, wherer $C_x$ is an area-minimizing cone in $\R^{n+1}$.
By the definition \eqref{thyangle} and $\th_x>0$, we have
\begin{equation}\aligned
\lim_{i\to\infty}t_i^{-1}d(l_x(t_i),M_{y_\infty})>0.
\endaligned
\end{equation}
Hence, there are a positive constant $\de_x>0$ and a measurable set $E_x\subset M_{y_\infty}\cap B_{\de_x}(x)$ such that
\begin{equation}\aligned
\inf_{y\in E_x}\th_y\ge\de_x\qquad \mathrm{and} \qquad \mathcal{H}^n(E_x)>\de_x^{n+1}.
\endaligned
\end{equation}
With co-area formula, there are a constant $\de_x'\in(0,\de_x)$ and a measurable set $I_x\subset(d(o,x)-\de_x,d(o,x)+\de_x)$ with $\mathcal{H}^1(I_x)\ge\de_x'$ such that
\begin{equation}\aligned\label{BdeMyinfty***}
\mathcal{H}^n\left(B_{\de}(M_{y_\infty})\cap\p B_\tau(o)\right)\ge(1+\de_x')\mathcal{H}^n\left(B_{\de}(M_{y_\infty}\cap\p B_\tau(o))\cap\p B_\tau(o)\right).
\endaligned
\end{equation}
From \eqref{BdeMyinfty***} and the argument \eqref{BdeMinftyy}-\eqref{HnMinftyM*rs}, for any $0<s<r$ we have
\begin{equation}\aligned
&\mathcal{H}^n\left(M_{y_\infty}\cap \overline{B_{r}(o)}\right)-\mathcal{H}^n\left(M_{y_\infty}\cap B_{s}(o)\right)\\
\ge&(1+\de_x')\int_{I_x}\mathcal{M}_*\left(M_{y_\infty}\cap\p B_{\tau}(o),\p B_{\tau}(o)\right)d\tau+\int_{(s,r)\setminus I_x}\mathcal{M}_*\left(M_{y_\infty}\cap\p B_{\tau}(o),\p B_{\tau}(o)\right)d\tau.
\endaligned
\end{equation}
By following the argument of the proof of Lemma \ref{pM+Minfty}, the above inequality contradicts to \eqref{620epj}.
Hence, we conclude that $M_{y_\infty}$ is a metric cone, which completes the proof.
\end{proof}

The complexity of the proof of Lemma \ref{pM+Minfty} comes from no information of the topology of the metric cone $C$. 
However, if  the cone has trivial topology, then Lemma \ref{pM+Minfty} and Theorem \ref{CovconeC} can hold for a more general case in the following remark.
\begin{remark}\label{Monoalmost}
Let $N_i$ be a sequence of $(n+1)$-dimensional complete non-compact manifolds with 
$\mathrm{Ric}\ge-n\ep^2$ on $B_{1}(q_i)\subset N_i$ and with
$\mathcal{H}^{n+1}(B_1(q_i))\ge (1-\ep)\omega_{n+1}$
for some constant $\ep\in(0,1)$.
We suppose that $\overline{B_1(q_i)}$ converges to a metric ball $\overline{B_1(q_\infty)}$ in the Gromov-Hausdorff sense.
For each integer $i\ge1$, let $M_i\subset N_i$ be an area-minimizing hypersurface in $B_1(q_i)$ with $\p M_i\subset\p B_1(q_i)$. Suppose that $M_i$ converges to a closed set $M_\infty\subset \overline{B_1(q_\infty)}$ in the induced Hausdorff sense.

From Theorem A.1.8 in \cite{CCo1}, for small $\ep>0$, any tangent cone of $B_1(q_\infty)$ has trivial topology. 
Then \eqref{M*CMINFtoge} holds in this new situation directly from the proof of Proposition 4.2 in \cite{D}. We follow the argument in the proof of Lemma \ref{pM+Minfty} starting from \eqref{M*CMINFtoge}, and get \eqref{t1t2HnMpinfty} for $M_\infty$ here, as well as Theorem \ref{CovconeC} in our new situation.
\end{remark}

\section{Sobolev and Poincar\'e inequalities on minimal graphs over manifolds}

Let $\Si$ be an $n$-dimensional smooth complete noncompact Riemannian manifold with
\begin{equation}\aligned\label{Ric}
\mathrm{Ric}\ge-(n-1)\k^2\qquad\qquad \mathrm{on}\ \ B_2(p)
\endaligned
\end{equation}
for some constant $\k\ge0$.
Let $M\subset\Si\times\R$ be a minimal graph over $B_2(p)\subset\Si$ with $\bar{p}=(p,0)\in M$ and $\p M\subset\p B_2(p)\times\R$. From De Giorgi \cite{Dg},  $M$ is smooth.
From Bishop-Gromov volume comparison, Lemma 3.1 in \cite{D}, and Lemma 3.3 in \cite{D}, there is a positive constant $\be_{\k}\le\omega_n^{-\f12}$ depending only on $n,\k$ so that
\begin{equation}\aligned\label{ben***}
\be_{\k}\mathcal{H}^{n}\left(B_{1}(p)\right)r^n\le\mathcal{H}^n(M\cap B_r(z))\le\f{r^n}{\be_{\k}}
\endaligned
\end{equation}
for any $B_r(z)\subset B_2(p)\times\R$ with $z\in M\subset\Si\times\R$.
We further assume
\begin{equation}\aligned\label{Volv}
\mathcal{H}^{n}(B_1(p))\ge v
\endaligned
\end{equation}
for some constant $v\in(0,\omega_n]$.

Now let us prove the isoperimetric inequality on minimal graphs over manifolds.
\begin{theorem}\label{ISOM}
Let $S$ be an open set in $M\cap B_{1}(\bar{p})$ with $(n-1)$-rectifiable boundary $\p S$. Then the isoperimetric inequality holds:
\begin{equation}\aligned\label{Iso***}
\left(\mathcal{H}^n(S)\right)^{\f{n-1}n}\le \Th_{\k,v} \mathcal{H}^{n-1}(\p S),
\endaligned
\end{equation}
where $\Th_{\k,v}$ is a positive constant depending only on $n,\k,v$.
\end{theorem}
\begin{proof}
We shall prove \eqref{Iso***} by contradiction. Let $\Si_i$ be a sequence of smooth complete noncompact manifolds with $\mathrm{Ric}\ge-(n-1)\k^2$ on $B_2(p_i)$ and $\mathcal{H}^n(B_1(p_i))\ge v$.
For each $i$, let $M_i\subset\Si_i\times\R$ be a minimal graph over $B_{2}(p_i)$ with $\p M_i\subset\p B_{2}(p_i)\times\R$ and $\bar{p}_i=(p_i,0)\in M_i$,
let $S_i$ be an open set in $M_i\cap B_{1}(\bar{p}_i)$ with $(n-1)$-rectifiable boundary $\p S_i$ such that
\begin{equation}\aligned\label{HnSif1i*}
\left(\mathcal{H}^n(S_i)\right)^{\f{n-1}n}> i \mathcal{H}^{n-1}(\p S_i).
\endaligned
\end{equation}
Without loss of generality,
we can further assume that $S_i$ is an outward minimizing set in $M_i\cap B_{1}(\bar{p}_i)$ (see its definition in \eqref{ODEEE*}). 

From \eqref{ben***} and $\mathcal{H}^n(B_1(p_i))\ge v$, we have
\begin{equation}\aligned\label{HnMiBrzbe*}
\be_\k v r^n\le\mathcal{H}^n(M_i\cap B_{r}(z))\le r^n/\be_\k
\endaligned
\end{equation}
for each $z\in M_i$ and $0<r\le2-d(z,\bar{p}_i)$.
For every $z\in M_i\cap B_{1}(\bar{p}_i)$, from the triangle inequality for distance functions, there exists a point $y_z\in\p B_{\f54}(\bar{p}_i)\cap M_i$ with $B_{\f{1}4}(y_z)\subset B_{\f32-d(z,\bar{p}_i)}(z)\setminus B_{1}(\bar{p}_i)$.
Since $S_i\subset M_i\cap B_1(\bar{p}_i)$, we get
\begin{equation}\aligned
\mathcal{H}^n\left(S_i\cap B_{\f32-d(z,\bar{p}_i)}(z)\right)<\mathcal{H}^n\left(M_i\cap B_{\f32-d(z,\bar{p}_i)}(z)\right)-\mathcal{H}^n\left(M_i\cap B_{\f14}(y_z)\right).
\endaligned
\end{equation}
Combining \eqref{HnMiBrzbe*}, we obtain
\begin{equation}\aligned
\f{\mathcal{H}^n\left(S_i\cap B_{\f32-d(z,\bar{p}_i)}(z)\right)}{\mathcal{H}^n\left(M_i\cap B_{\f32-d(z,\bar{p}_i)}(z)\right)}
<&1-\f{\mathcal{H}^n\left(M_i\cap B_{\f14}(y_z)\right)}{\mathcal{H}^n\left(M_i\cap B_{\f32-d(z,\bar{p}_i)}(z)\right)}\\
\le&1-\f{\be_\k v}{4^n}\f{2^n\be_\k}{3^n}=1-\f{\be_\k^2v}{6^{n}}.
\endaligned
\end{equation}
Put $\ep_{\k,v}=1-6^{-n}\be_\k^2v$. Then $\ep_{\k,v}\ge1-6^{-n}$ from $\be_{\k}\le\omega_n^{-\f12}$ and $v\le\omega_n$. For any $z_i\in S_i$, there is a constant $r_{z_i}\in\left(0,\f32-d(z_i,\bar{p}_i)\right)$ such that
\begin{equation}\aligned\label{3.38}
\f{\mathcal{H}^n(S_i\cap B_{r_{z_i}}(z_i))}{\mathcal{H}^n(M_i\cap B_{r_{z_i}}(z_i))}=\ep_{\k,v},\quad \mathrm{and}\quad \f{\mathcal{H}^n(S_i\cap B_{r}(z_i))}{\mathcal{H}^n(M_i\cap B_{r}(z_i))}>\ep_{\k,v}\quad \mathrm{for\ all}\ 0<r<r_{z_i}.
\endaligned
\end{equation}
With \eqref{HnMiBrzbe*}, for any $0<r\le r_{z_i}$
\begin{equation}\aligned\label{Tjarn}
\ep_{\k,v}\be_\k vr^n\le\ep_{\k,v}\mathcal{H}^n(M_i\cap B_{r}(z_i))\le\mathcal{H}^n(S_i\cap B_{r}(z_i)),
\endaligned
\end{equation}
which implies
\begin{equation}\aligned\label{3.40}
r_{z_i}\le\left(\mathcal{H}^n(S_i\cap B_{r_{z_i}}(z_i))\right)^{\f 1{n}}\left(\ep_{\k,v}\be_\k v\right)^{-\f 1{n}}.
\endaligned
\end{equation}
We can choose a finite collection of geodesic balls $\{B_{r_{z_{i,j}}}(z_{i,j})\}_{j=1}^{m_i}$ in $\Si_i\times\R$ with $z_{i,j}\in S_i$ for each $i$ such that $\overline{S_i}\subset \bigcup_{j=1}^{m_i}B_{r_{z_{i,j}}}(z_{i,j})$, and the number of balls containing any point in $S_i$ satisfies
\begin{equation}\aligned\nonumber
\sharp\{j\in\{1,\cdots,m_i\}|\ y\in B_{r_{z_{i,j}}}(z_{i,j})\}\le n_\k\qquad \mathrm{for\ any}\ y\in S_i\ \ \mathrm{and}\ i\ge0,
\endaligned
\end{equation}
where $n_\k$ is a constant depending only on $n,\k$.
Combining \eqref{HnSif1i*}, we get
\begin{equation}\aligned\nonumber
&\sum_{j=1}^{m_i}\mathcal{H}^{n-1}(\p S_i\cap B_{r_{z_{i,j}}}(z_{i,j}))\le n_\k\mathcal{H}^{n-1}(\p S_i)\\
\le&\f{n_\k}i\left(\mathcal{H}^n(S_i)\right)^{1-\f 1{n}}\le\left(\mathcal{H}^n(S_i)\right)^{-\f 1{n}}\f{n_\k}i\sum_{j=1}^{m_i}\mathcal{H}^{n}(S_i\cap B_{r_{z_{i,j}}}(z_{i,j})).
\endaligned
\end{equation}
For each integer $i\ge1$, we can pick a point $\bar{\xi}_i=(\xi_i,t_{\xi_i})\in\bigcup_{j=1}^{m_i}\{z_{i,j}\}\subset S_i$ for some $\xi_i\in\Si_i$ and $t_{\xi_i}\in\R$ such that
\begin{equation}\aligned\label{3.43}
\mathcal{H}^{n-1}\left(\p S_i\cap B_{r_{\bar{\xi}_i}}(\bar{\xi}_i)\right)\le\f{n_\k}i\left(\mathcal{H}^n(S_i)\right)^{-\f 1{n}}\mathcal{H}^{n}\left(S_i\cap B_{r_{\bar{\xi}_i}}(\bar{\xi}_i)\right).
\endaligned
\end{equation}

Denote $r_i=r_{\bar{\xi}_i}$ for short. Since $r_{i}\le \f32-d(\bar{\xi}_i,\bar{p}_i)$, then $B_{\f43r_i}(\bar{\xi}_i)\subset B_2(\bar{p}_i)$ and then $B_{\f43r_i}(\xi_i)\subset B_2(p_i)$.
Let $(\Si'_i,\xi_i')=\f1{r_i}(\Si_i,\xi_i)$, then $\Si'_i$ has Ricci curvature $\ge-(n-1)\k^2r_i^2$ on $B_{\f43}(\xi_i')$ and $\mathcal{H}^n(B_1(\xi_i'))$ has a uniform positive lower bound (depending on $n,\k,v$).
From $\bar{\xi}_i\in M_i\cap (B_{1}(p_i)\times\R)$ and $\p M_i\subset\p B_{2}(p_i)\times\R$, one has $\p M_i\cap (B_{2-d(\xi_i,p_i)}(\xi_i)\times\R)=\emptyset$. 
Let $(M'_i,\bar{\xi}_i')=\f1{r_i}(M_i,\bar{\xi}_i)$ with $\bar{\xi}_i'\in M'_i$, then $M'_i$ is a minimal graph over $\f1{r_i}B_2(p_i)$ in $\Si'_i\times\R$ with $\p M'_i\cap \left(B_{\f43}(\xi_i')\times\R\right)=\emptyset$.
In particular, from \eqref{HnMiBrzbe*} it follows that
\begin{equation}\aligned\label{HnMiBrzbe***}
\be_\k v r^n\le\mathcal{H}^n(M_i'\cap B_{r}(y))\le r^n/\be_\k
\endaligned
\end{equation}
for each $y\in M_i'$ and $0<r\le\f43-d(y,\bar{\xi}_i')$.
Let $S'_i=\f1{r_i}\left(S_i\cap B_{r_i}(\bar{\xi}_i)\right)$, then $S'_i$ is outward minimizing in $M'_i\cap B_1(\bar{\xi}_i')$. From \eqref{3.38}, for any $t\in(0,1]$ one has
\begin{equation}\aligned
&\mathcal{H}^{n}(S'_i\cap B_t(\bar{\xi}_i'))=\f1{r_i^{n}}\mathcal{H}^{n}(S_i\cap B_{tr_i}(\bar{\xi}_i))\\
\ge&\f{\ep_{\k,v}}{r_i^{n}}\mathcal{H}^{n}(M_i\cap B_{tr_i}(\bar{\xi}_i))=\ep_{\k,v}\mathcal{H}^{n}(M'_i\cap B_{t}(\bar{\xi}_i')),
\endaligned
\end{equation}
and
\begin{equation}\aligned\label{417***}
\mathcal{H}^{n}(S'_i\cap B_1(\bar{\xi}_i'))=\ep_{\k,v}\mathcal{H}^{n}(M'_i\cap B_{1}(\bar{\xi}_i')).
\endaligned
\end{equation}
Moreover, combining \eqref{HnMiBrzbe*}\eqref{3.38}\eqref{3.40}\eqref{3.43} one has
\begin{equation}\aligned\label{pSixiito0}
&\mathcal{H}^{n-1}(\p S'_i\cap B_1(\bar{\xi}_i'))=\f1{r_i^{n-1}}\mathcal{H}^{n-1}(\p S_i\cap B_{r_i}(\bar{\xi}_i))\\
\le&\f{n_\k}{ir_i^{n-1}}\left(\mathcal{H}^n(S_i)\right)^{-\f 1{n}}\mathcal{H}^{n}(S_i\cap B_{r_i}(\bar{\xi}_i))
=\f{n_\k\ep_{\k,v}}{ir_i^{n-1}}\left(\mathcal{H}^n(S_i)\right)^{-\f 1{n}}\mathcal{H}^{n}(M_i\cap B_{r_i}(\bar{\xi}_i))\\\le&\f{n_\k\ep_{\k,v}}{ir_i^{n-1}}\f{r_i^{n}}{\be_\k}\left(\mathcal{H}^n(S_i)\right)^{-\f 1{n}}=\f{n_\k\ep_{\k,v}}{i\be_\k}r_i\left(\mathcal{H}^n(S_i)\right)^{-\f 1{n}}\le\f{n_\k}{i}\ep_{\k,v}^{1-\f1n}\be_\k^{-1-\f1n}v^{-\f 1{n}}.
\endaligned
\end{equation}

Up to the choice of the subsequence, there is a complete metric space $\Si_\infty$ so that $(\Si_i'\times\R,\bar{\xi}_i')$ converges to $(\Si_\infty\times\R,\xi_\infty)$ in the pointed Gromov-Hausdorff sense. 
Let $V_i^+$ and $V_i^-$ be the two disjoint open sets in $B_{1}(\bar{\xi}_i')$ such that $V_i^+\cup V_i^-=B_1(\bar{\xi}_i')\setminus M_i'$, $V_i^+$ is above $M_i'$ , and $V_i^-$ is below $M_i'$.
Moreover, there are closed sets $M_\infty$, $S_\infty$, $T_\infty$, $\G_\infty$, $V_\infty^+$, $V_\infty^-$ in $\Si_\infty\times\R$ so that (up to the choice of the subsequence) $M'_i\cap B_{\f43}(\bar{\xi}_i'),S'_i,M'_i\cap B_1(\bar{\xi}_i')\setminus S'_i$, $\p S'_i\cap B_1(\bar{\xi}_i')$, $V_i^+,V_i^-$ converges to $M_\infty$, $S_\infty$, $T_\infty$, $\G_\infty$, $V_\infty^+$, $V_\infty^-$ all in the induced Hausdorff sense, respectively. It's not hard to see that $\G_\infty\subset S_\infty\cap T_\infty$, $\p S_\infty\cap\p T_\infty\cap B_{1}(\xi_\infty)\subset\G_\infty$, $\p V_\infty^+\cap\p V_\infty^-\cap B_{1}(\xi_\infty)=M_\infty\cap B_{1}(\xi_\infty)$.
For any $x\in S_\infty$, we say $x\in S_\infty^*$ if there is a sequence $x_i\in S_i'$ with $x_i\to x$ so that 
\begin{equation}\aligned
\limsup_{i\to\infty}\mathcal{H}^{n}\left(S_i'\cap B_{\tau}(x_i)\right)>0\qquad \mathrm{for\ any}\  \tau>0.
\endaligned
\end{equation}
On the contrary, if $y\in S_\infty\setminus S_\infty^*$, then for any sequence $y_i\in S_i'$ with $y_i\to y$ there is a constant $\tau_y>0$ such that
\begin{equation}\aligned\label{HnSi'ry0}
\lim_{i\to\infty}\mathcal{H}^{n}\left(S_i'\cap B_{\tau_y}(y_i)\right)=0.
\endaligned
\end{equation}

Now we follow the argument of the proof of Theorem \ref{BRpNf1RM}, and get that 
$S_\infty^*$ is a closed set in $S_\infty$.
Moreover, for any $\ep>0$, there is a finite open covering $\{U_j\}_{j=1}^{N_\ep}$ of $S_\infty^*$ so that 
\begin{equation}\aligned\label{Sinf*Uj}
\mathcal{H}^{n}\left(S_\infty^*\right)\ge2^{-n}\omega_n\sum_{j=1}^{N_\ep}(\mathrm{diam}U_j)^n-\ep.
\endaligned
\end{equation}
Since $S_\infty$ is closed, then by finite covering lemma again, there is a sequence of open balls $\{B_{\tau_{q_j}}(q_j)\}_{j=1}^{m}$ with $q_j\in S_\infty\setminus S_\infty^*$ and $B_{\tau_{q_j}}(q_j)\cap S_\infty^*=\emptyset$ such that $$S_\infty\subset\bigcup_{j=1}^{m}B_{\tau_{q_j}}(q_j)\cup\bigcup_{j=1}^{N_\ep}U_j$$
and
\begin{equation}\aligned\label{HnSi'ryj0}
\lim_{i\to\infty}\mathcal{H}^{n}\left(S_i'\cap B_{\tau_{q_j}}(q_{i,j})\right)=0\qquad \mathrm{for\ each}\ j=1,\cdots,m,
\endaligned
\end{equation}
where $\{q_{i,j}\}_{j=1}^m$ is a finite sequence in $S_i'$ with $q_{i,j}\to q_j$ as $i\to\infty$. Let $U_{i,j}\subset B_1(\xi_i)$ be an open set with $U_{i,j}\to U_j$ as $i\to\infty$ such that
$$S_i'\subset\bigcup_{j=1}^{m}B_{\tau_{q_j}}(q_{i,j})\cup\bigcup_{j=1}^{N_\ep}U_{i,j}.$$
Clearly, for each $i,j$ there is a ball $B_{\mathrm{diam}U_j}(q_{i,j}')\subset \Si_i'\times\R$ with $U_{i,j}\subset B_{\mathrm{diam}U_j}(q_{i,j}')$. Then
\begin{equation}\aligned
\mathcal{H}^{n}\left(S_i'\right)\le&\sum_{j=1}^{m}\mathcal{H}^{n}\left(S_i'\cap B_{\tau_{q_j}}(q_{i,j})\right)+\sum_{j=1}^{N_\ep}\mathcal{H}^{n}\left(S_i'\cap U_{i,j}\right)\\
\le&\sum_{j=1}^{m}\mathcal{H}^{n}\left(S_i'\cap B_{\tau_{q_j}}(q_{i,j})\right)+\sum_{j=1}^{N_\ep}\mathcal{H}^{n}\left(M_i'\cap B_{\mathrm{diam}U_j}(q_{i,j}')\right).
\endaligned
\end{equation}
Combining \eqref{HnMiBrzbe***}\eqref{417***}\eqref{Sinf*Uj}\eqref{HnSi'ryj0}, we get
\begin{equation}\aligned
\ep_{\k,v}\be_\k v\le&\ep_{\k,v}\limsup_{i\to\infty}\mathcal{H}^{n}(M'_i\cap B_{1}(\bar{\xi}_i'))\le
\limsup_{i\to\infty}\mathcal{H}^{n}\left(S_i'\right)\\
\le&\limsup_{i\to\infty}\sum_{j=1}^{N_\ep}\mathcal{H}^{n}\left(M_i'\cap B_{\mathrm{diam}U_j}(q_{i,j}')\right)\\
\le&\sum_{j=1}^{N_\ep}(\mathrm{diam}U_j)^n/\be_\k\le\f{2^n}{\omega_n\be_\k}\left(\mathcal{H}^{n}\left(S_\infty^*\right)+\ep\right).
\endaligned
\end{equation}
Letting $\ep\to0$ in the above inequality implies
\begin{equation}\aligned
\mathcal{H}^{n}\left(S_\infty^*\right)\ge2^{-n}\ep_{\k,v}\be_\k^2\omega_n v.
\endaligned
\end{equation}
By the definition of $S_\infty^*$, it follows that $\p S_\infty^*\subset\G_\infty$, and
\begin{equation}\aligned
 \mathcal{H}^n\left(S_\infty^*\cap\overline{B_\tau(x)}\right)\ge\limsup_{i\to\infty}\mathcal{H}^{n}\left(S_i'\cap \overline{B_{\tau}(x_i)}\right)>0
\endaligned
\end{equation}
from Theorem 5.4 in \cite{D} for any $x\in S_\infty^*\cap B_1(\xi_\infty)$, $0<\tau<1-d(x,\xi_\infty)$ and $S_i'\ni x_i\to x$.
 

Let $\mathcal{S}$ denote the singular set of $B_1(\xi_\infty)$.
For any $\ep\in(0,\f12]$, let $\mathcal{R}_{\ep}$ denote the \emph{$\ep$-regular} set of $B_1(\xi_\infty)$ defined by
\begin{equation}\aligned\label{RSep}
\mathcal{R}_{\ep}=\left\{x\in B_1(\xi_\infty)\Big|\,\sup_{0<s\le r}s^{-1}d_{GH}(B_s(x),B_s(0^{n+1}))<\ep \ \mathrm{for\ some}\ r>0\right\},
\endaligned
\end{equation}
and $\mathcal{S}_{\ep}=B_1(\xi_\infty)\setminus\mathcal{R}_{\ep}$ be the \emph{$\ep$-singular} set of $B_1(\xi_\infty)$. Clearly, $\mathcal{S}_{\ep}\subset\mathcal{S}$ is closed.
From Theorem 1.2 in \cite{D},  $\mathcal{S}\cap M_\infty$ has Hausdorff dimension $\le n-2$. 
Suppose there is a point $x^*\in S_\infty^*\cap\G_\infty\cap\mathcal{R}_\ep$. 
For the suitably small $\ep>0$, using \eqref{ODEEE} 
there is a sequence $x_i^*\in\p S_i'$ with $x_i^*\to x^*$ so that 
\begin{equation}\aligned\label{Si'r01}
\limsup_{i\to\infty}\min\left\{\mathcal{H}^{n}\left(S_i'\cap B_{\tau}(x_i^*)\right),\mathcal{H}^{n}\left(B_{\tau}(x_i^{*})\cap M_i'\setminus S_i'\right)\right\}>0\qquad \mathrm{for\ any}\  \tau>0.
\endaligned
\end{equation}
From \eqref{pSixiito0}, for any small $\tau>0$
\begin{equation}\aligned\label{Si'r03}
\lim_{i\to\infty}\mathcal{H}^{n-1}(\p S'_i\cap B_{\tau}(x_i^*))=0.
\endaligned
\end{equation}
However, \eqref{Si'r01}\eqref{Si'r03} contradict to Theorem \ref{BRpNf1RM} for the suitably small $\ep>0$. Hence we obtain
\begin{equation}\aligned\label{S*inftyRep0}
S_\infty^*\cap\G_\infty\cap\mathcal{R}_\ep=\emptyset, \ i.e. \ , S_\infty^*\cap\G_\infty\cap B_1(\xi_\infty)\subset\mathcal{S}_\ep. 
\endaligned
\end{equation}

Let $V$ be a component of $V_\infty^+\setminus\p V_\infty^+$ or $V_\infty^-\setminus\p V_\infty^-$ so that $S_\infty^*\subset\p V$. 
Given a small $\de>0$, a point $y\in S_\infty^*\cap\G_\infty\cap B_1(\xi_\infty)$ and a constant $r\in(0,1-d(y,\xi_\infty))$. Denote $r_y=1-d(y,\xi_\infty)$.
There is a finite collection of balls $\{B_{s_k}(x_{\de,k})\}_{k=1}^{N_{\de}}$ in $B_{2}(\xi_\infty)$ so that $\mathcal{S}_\ep\cap M_\infty\subset\cup_{k=1}^{N_{\de}}B_{ts_k}(x_{\de,k})\triangleq U_{t,\de}$ for each $t\in[1,2]$ and 
$$\sum_{k=1}^{N_{\de}}s_k^{n-3/2}<\de r^n.$$
Let $x_{\de,k,i}\in B_1(\bar{\xi}_i')$ with $x_{\de,k,i}\to x_{\de,k}$, and $y_{i}\in M_i'$ with $y_{i}\to y$.
Denote $U_{t,\de,i}=\cup_{k=1}^{N_{\de}}B_{ts_k}(x_{\de,k,i})$ for each $i\in\mathbb{N}$, and
$$V_{M_\infty,y,r}=(M_\infty\cap B_{r_y}(y))\cup(\overline{V}\cap\p B_r(y))\setminus(\p V\cap B_r(y)).$$
There is a sequence of countably $n$-rectifiable sets $R_{r,\de,t,i}$ in $B_{4r_y/3}(y_i)$ so that $R_{r,\de,t,i}=M_i'$ outside $B_{r_y}(y_i)$, $\p R_{r,\de,t,i}\cap B_{r_y}(y_i)\subset U_{t,\de,i}$ and $R_{r,\de,t,i}\cap B_{r_y}(y_i)$ converges to 
$V_{M_\infty,y,r}\setminus U_{t,\de}$.
Then we follow the argument of the proof of Lemma \ref{pM+Minfty} with the help of Lemma \ref{ExistcurrentSi} and the co-area formula, and can get a sequence of countably $n$-rectifiable sets $R_{r,i}\subset B_{4r_y/3}(y_i)$ such that
$R_{r,i}=M_i'$ outside $B_{r_y}(y_i)$, $\p R_{r,i}\cap B_{r_y}(y_i)=\emptyset$ and $R_{r,i}\cap B_{r_y}(y_i)$ converges to a closed set $R_{r,\infty}$ in $\overline{B_1(\xi_\infty)}$ with $V_{M_\infty,y,r}\subset R_{r,\infty}$ and $R_{r,\infty}\setminus V_{M_\infty,y,r}\subset \mathcal{S}_{\ep}$. Combining Theorem 1.3 in \cite{CN} by Cheeger-Naber and the proof of Proposition 4.2 in \cite{D}, the lower $n$-dimensional Minkowski contents satisfy(compared \eqref{M*CMINFtoge})
\begin{equation}\aligned
\mathcal{M}_*(\p V\cap B_r(y)),B_r(y))\le\mathcal{M}_*(V\cap\p B_r(y),B_r(y)) \qquad \mathrm{for\ almost\ every}\ r\in(0,r_y).
\endaligned
\end{equation}
With Corollary 5.5 in \cite{D}, we get
\begin{equation}\aligned\label{pVBrpBrV}
\mathcal{H}^n(\p V\cap B_r(y))\le\mathcal{H}^n(V\cap\p B_r(y)).
\endaligned
\end{equation}
Using the proof of Lemma 3.5 in \cite{D}, from \eqref{pVBrpBrV} we have
\begin{equation}\label{nondegenerate*V}
\mathcal{H}^{n+1}(V\cap B_r(y))\ge\de_{n,\k}r^{n+1}
\end{equation}
for any $r\in(0,r_y)$. Here, $\de_{n,\k}$ is a positive constant depending only on $n,\k$.

Now we fix the suitably small $\ep>0$ so that \eqref{S*inftyRep0} holds. For any point $y^*_0\in S_\infty^*\cap\G_\infty\cap B_1(\xi_\infty)$, from Theorem \ref{CovconeC}
there is a sequence $t_{0,i}\to0$ such that $\f1{t_{0,i}}(B_1(\xi_\infty),y^*_0)$ converges to a metric cone $(C_0,o_0)$ in the pointed Gromov-Hausdorff sense, and
$\f1{t_{0,i}}(M_\infty,y^*_0)$ converges in the induced Hausdorff sense to a metric cone $(C'_0,o_0)$ with $C'_0\subset C_0$.
Up to a choice of a subsequence, we further assume that
$\f1{t_{0,i}}(S_\infty^*\cap\G_\infty,y^*_0)$ converges in the induced Hausdorff sense to $(\G_0,o_0)$ for some closed set $\G_0\subset C'_0$. 
For any sequence $z^*_i\in\f1{t_{0,i}}(\mathcal{S}_\ep,y^*_0)$, if $z^*_i$ converges to a point $z^*_0\in C_0$, then by Bishop-Gromov volume comparison and volume convergence by Cheeger-Colding \cite{CCo1}, there is a constant $\ep_0\in(0,\ep]$ so that $z^*_0$ belongs to the $\ep_0$-singular set of $C_0$. With \eqref{S*inftyRep0}, this implies that $\G_0$ belongs to the $\ep_0$-singular set of $C_0$.

If there is a point $y^*_1\in \G_0\setminus\{o_0\}$, then we blow $C_0,C'_0,\G_0$ up at $y^*_1$. More precisely, from Theorem \ref{CovconeC} and \eqref{S*inftyRep0}, there are a sequence $t_{1,i}\to0$ such that $\f1{t_{1,i}}(C_0,y^*_1)$ converges to a metric cone $(C_1\times\R,o_1)$ in the Gromov-Hausdorff sense,
$\f1{t_{1,i}}(C'_0,y^*_1)$ converges in the induced Hausdorff sense to a metric cone $(C'_1\times\R,o_1)$ with $C'_1\subset C_1$ and
$\f1{t_{1,i}}(\G_0,y^*_1)$ converges in the induced Hausdorff sense to $(\G_1,o_1)$ for some closed set $\G_1\subset C'_1\times\R$ with $\G_1$ belonging to the $\ep_1$-singular set of $C_1\times\R$ for some $\ep_1\in(0,\ep_0]$.
If there is a point $y^*_2\in \G_1\setminus\{o_1\}$, then we can blow $C_1,C'_1,\G_1$ up at $y^*_2$ further.
By induction, for some integer $k\in\{1,\cdots,n-1\}$
there are a sequence $t_{k,i}\to0$ such that $\f1{t_{k,i}}(C_{k-1}\times\R^{k-1},y^*_{k-1})$ converges to a metric cone $(C_k\times\R^k,o_k)$ in the Gromov-Hausdorff sense,
$\f1{t_{k,i}}(C'_{k-1}\times\R^{k-1},y^*_{k-1})$ converges in the induced Hausdorff sense to a metric cone $(C'_k\times\R^k,o_k)$ with $C'_k\subset C_k$ and
$\f1{t_{k,i}}(\G_{k-1},y^*_k)$ converges in the induced Hausdorff sense to $(\G_k,o_k)$ for some closed set $\G_k\subset\{o_k\}\times\R^k$ belonging to the $\ep_k$-singular set of $C_k\times\R^k$ for some $\ep_k\in(0,\ep_{k-1}]$.

By the above construction, for each integer $i\ge1$ there are a point $\e_i$, 4 compact metric spaces $Z_i^\G,Z_i,Y_i,X_i$, which are a finite times of scalings of $S_\infty^*\cap\G_\infty,S_\infty^*,M_\infty,\overline{B_1(\xi_\infty)}$  (w.r.t. $(k+1)$-points $y_0^*,\cdots,y_k^*$)  respectively, satisfying $\e_i\in Z_i^\G\subset Z_i\subset Y_i\subset X_i$ such that 
$Z_i^\G$ contains no points of $\f{\ep_k}2$-regular set of $X_i$, $\mathcal{S}_i\cap Y_i$ has the Hausdorff dimension $\le n-2$ (see Theorem 1.2 in \cite{D}) with $\mathcal{S}_i$ denoting the singular set of $X_i$, $(X_i,\e_i)\to(C_k\times\R^k,o_k)$ in the pointed Gromov-Hausdorff sense, and $(Y_i,\e_i)\to(C'_k\times\R^k,o_k)$, $(Z_i,\e_i)\to(Z_\infty,o_k)$, $(Z_i^\G,\e_i)\to(\G_k,o_k)$ all in the induced Hausdorff sense for some closed set $Z_\infty\subset C'_k\times\R^k$. 
On the other hand, there are a sequence $R_i\to\infty$, a sequence of smooth complete noncompact manifolds $\Si_i^*$ with $\mathrm{Ric}\ge-(n-1)\k^2R_i^{-2}$ on $B_{R_i}(p_i^*)\subset\Si_i^*$ and $\liminf_{i\to\infty}\mathcal{H}^n(B_1(p_i^*))>0$, a sequence of minimal graphs $M_i^*$ over $B_{R_i}(p_i^*)$ in $\Si_i^*\times\R$ with $\p M_i^*\subset\p B_{R_i}(p_i^*)\times\R$ and $\bar{p}_i^*=(p_i^*,0)\in M_i^*$,
such that $(\Si_i^*\times\R,\bar{p}_i^*)$ converges to $(C_k\times\R^k,o_k)$ in the Gromov-Hausdorff sense, and $(M_i^*,\bar{p}_i^*)$ converges to $(C'_k\times\R^k,o_k)$ in the induced Hausdorff sense. 
From the area-minimizing $M_i^*$, we conclude that the cone $C'_k$ has dimension $\ge2$ (see also the proof of Lemma 6.5 in \cite{D}).
Since the boundary of $Z_\infty$ in $C'_k\times\R^k$ satisfies $\p Z_\infty\subset\G_k\subset\{o_k\}\times\R^k$, we conclude that $\p Z_\infty=\emptyset$ if $Z_\infty\setminus\{o_k\}\times\R^k\neq\emptyset$. 
With \eqref{nondegenerate*V}, it follows that $Z_\infty\setminus\{o_k\}\times\R^k\neq\emptyset$.
From Theorem 6.6 in \cite{D1}, the cross section of $C'_k$ is connected in $C_k$. Thus we can deduce 
\begin{equation}\aligned\label{Zinftyis***}
Z_\infty=C_k'\times\R^k.
\endaligned
\end{equation}
Analogously, we consider $Y_i\setminus Z_i$ instead of $Z_i$, and can deduce $(Y_i\setminus Z_i,\e_i)\to(C'_k\times\R^k,o_k)$ using Theorem 6.6 in \cite{D1} and \eqref{nondegenerate*V}.
This contradicts to \eqref{Zinftyis***}. We complete the proof.
\end{proof}
\begin{remark}\label{Sobalmost}
Let $N$ be an $(n+1)$-dimensional smooth complete noncompact Riemannian manifold with
$\mathrm{Ric}\ge-n\ep^2$ on $B_2(q)\subset N$ and $\mathcal{H}^{n+1}(B_1(q))\ge(1-\ep)\omega_{n+1}$ for a constant $\ep\in(0,1)$.
Let $M$ be an area-minimizing hypersurface in $B_2(q)$ with $q\in M$ and $\p M\subset\p B_2(q)$.
With Remark \ref{Monoalmost}, we find that the proof of Theorem \ref{ISOM} works for every area-minimizing hypersurface $M$ in $N$ provided $\ep>0$ is suitably small. 
Namely, 
there are constants $\ep>0$ and $\Th_*>0$ depending only on $n$ so that if $S$ is an open set in $M\cap B_{1}(p)$ with $(n-1)$-rectifiable boundary $\p S$, then
\begin{equation}\aligned
\left(\mathcal{H}^n(S)\right)^{\f{n-1}n}\le \Th_* \mathcal{H}^{n-1}(\p S).
\endaligned
\end{equation}
\end{remark}

With the isoperimetric inequality (Theorem \ref{ISOM}), after a standard argument we can get the Sobolev inequality on minimal graphs as follows.
\begin{theorem}\label{SobMinG}
Let $\Si$ be an $n$-dimensional smooth complete noncompact manifold with \eqref{Ric} and \eqref{Volv}. If $M$ is a minimal graph over $B_{2}(p)$ with $\p M\subset\p B_{2}(p)\times\R$ and $\bar{p}=(p,0)\in M$, then for any function $f\in C_0^{1}(M\cap B_1(\bar{p}))$, one has
\begin{equation}\aligned\label{Sob}
\left(\int_{ M}|f|^{\f n{n-1}}d\mu\right)^{\f{n-1}n}\le\Th_{\k,v}\int_{M}|\na f|d\mu,
\endaligned
\end{equation}
where $\Th_{\k,v}$ is the constant in Theorem \ref{ISOM}, $d\mu$ is the volume element of $M$, $\na$ is the Levi-Civita connection on $M$ compatible to the induced metric from $\Si\times\R$.
\end{theorem}

For showing Neumann-Poincar\'e inequality on minimal graphs, we need the following estimate.
\begin{lemma}\label{Cheeger***}
Let $\Si,M$ be as in Theorem \ref{SobMinG}.
There is a positive constant $\Th_{\k,v}'\ge1$ depending only on $n,\k,v$ such that
\begin{equation}\aligned
\mathcal{H}^{n-1}(\p S\cap B_1(\bar{p}))\ge\f1{2\Th_{\k,v}}\left(\min\{\mathcal{H}^n(S\cap B_{1/\Th_{\k,v}'}(\bar{p})),\mathcal{H}^n(M\cap B_{1/\Th_{\k,v}'}(\bar{p})\setminus S)\}\right)^{\f{n-1}n}
\endaligned
\end{equation}
for each $n$-rectifiable set $S\subset M\cap B_1(\bar{p})$ with $\mathcal{H}^n(S)>0$.
\end{lemma}
\begin{proof}
Let us prove it by contradiction. Assume there is a constant $\ep\in(0,\f12]$ such that
\begin{equation}\aligned\label{meapS***}
\mathcal{H}^{n-1}(\p S\cap B_1(\bar{p}))\le\f1{2\Th_{\k,v}}\left(\min\{\mathcal{H}^n(S\cap B_{\ep}(\bar{p})),\mathcal{H}^n(M\cap B_{\ep}(\bar{p})\setminus S)\}\right)^{\f{n-1}n}
\endaligned
\end{equation}
for some $n$-rectifiable set $S\subset M\cap B_1(p)$ with $\mathcal{H}^n(S)>0$. 
We follow the argument of Lemma \ref{Sob-Coarea} and its proof, where Lemma \ref{Rfiso} is replaced by Theorem \ref{ISOM}, and \eqref{1pmepHnM0} is replaced by \eqref{ben***}. We get
an outward minimizing set $V$ in $M\cap B_1(\bar{p})$ such that
\begin{equation}\aligned
\left(4n\Th_{\k,v}\right)^{-n}\le \mathcal{H}^{n}(V)\le\f23\mathcal{H}^{n}(M\cap B_1(\bar{p})),
\endaligned
\end{equation}
and
$$\mathcal{H}^{n-1}\left(\p V\cap B_1(\bar{p})\right)\le\f1{2\Th_{\k,v}}\left(\f12\mathcal{H}^n(M\cap B_{\ep}(\bar{p}))\right)^{\f{n-1}n}$$
for the suitably small $\ep>0$ (depending only on $n,\k,v$).
Then we follow the argument starting from (5.10) in the proof of Theorem \ref{ISOM}, and can deduce a contradiction for the sufficiently small $\ep>0$ (depending only on $n,\k,v$).
This completes the proof.
\end{proof}

Now we can use Lemma \ref{Cheeger***} to show a Neumann-Poincar\'e inequality on minimal graphs by a standard argument as follows.
\begin{theorem}\label{POINCARE}
Let $\Si$ be an $n$-dimensional smooth complete noncompact manifold with \eqref{Ric} and \eqref{Volv}.
There is a constant $\Th^*_{\k,v}$ depending only on $n,\k,v$, such that
if $M$ is a minimal graph over $B_{2}(p)\subset\Si$ with $\p M\subset\p B_{2}(p)\times\R$, $f$ is a $C^1$-function on $M$, then the following Neumann-Poincar\'e inequality
\begin{equation}\aligned\label{B1pPITh*}
\int_{M\cap B_{r}(x)}|f-\bar{f}_{x,r}|d\mu\le \Th^*_{\k,v}r \int_{M\cap B_{\Th^*_{\k,v}r}(x)}|\na f|d\mu
\endaligned
\end{equation}
holds on $B_{\Th^*_{\k,v}r}(x)\subset B_1(p)\times\R$ with $x\in M\subset\Si\times\R$, where $\bar{f}_{x,r}$ is the average of $f$ on $M\cap B_{r}(x)$, i.e., $\bar{f}_{x,r}=\fint_{M\cap B_{r}(x)}fd\mu$, and 
$\na$ is the Levi-Civita connection on $M$ compatible to the induced metric from $\Si\times\R$, $d\mu$ is the volume element of $M$.
\end{theorem}
\begin{proof}
Without loss of generality, let $f$ be not a constant. For simplicity, we always omit the volume element $d\mu$ for the integrations on the subsets of $M$ if no ambiguity. 
Let $\bar{f}_{x,r}$ be the average of $f$ on $M\cap B_r(x)$, i.e.,
$$\bar{f}_{x,r}=\fint_{M\cap B_r(x)}f=\f1{\mathcal{H}^n(M\cap B_r(x))}\int_{M\cap B_r(x)}f.$$
For any fixed $x\in M$,  $r>0$ with  $B_{\Th'_{\k,v}r}(x)\subset B_1(p)\times\R$, let
$U^+_{s,t}=\{y\in M\cap B_s(x)|\ f>\bar{f}_{x,r}+t\}$, $U^-_{s,t}=\{y\in M\cap B_s(x)|\ f<\bar{f}_{x,r}+t\}$ for all $s\in(0,\Th'_{\k,v}r)$ and $t\in\R$.
Without loss of generality, we assume $\mathcal{H}^n(U^+_{r,0})\le\mathcal{H}^n(U^-_{r,0})$. Then $\mathcal{H}^n(U^+_{r,t})\le\mathcal{H}^n(U^-_{r,t})$ for any $t\ge0$.
From Lemma \ref{Cheeger***}, we have
$$\mathcal{H}^{n-1}\left(\p U^+_{\Th'_{\k,v}r,t}\cap B_{\Th'_{\k,v}r}(x)\right)\ge\f1{2\Th_{\k,v}}\left(\mathcal{H}^n(U^+_{r,t})\right)^{\f{n-1}n}\ge\f{\left(\mathcal{H}^n(M\cap B_r(x))\right)^{-\f{1}n}}{2\Th_{\k,v}}\mathcal{H}^n(U^+_{r,t}).
$$
Combining co-area formula and \eqref{ben***}, we have
\begin{equation}\aligned
&\int_{U^+_{r,0}}(f-\bar{f}_{x,r})=\int_0^\infty\mathcal{H}^n(U^+_{r,t})dt\\
\le&2\Th_{\k,v}\left(r^n/\be_\k\right)^{\f1n}\int_0^\infty\mathcal{H}^{n-1}\left(\p U^+_{\Th'_{\k,v}r,t}\cap B_{\Th'_{\k,v}r}(x)\right)dt\\
\le& 2\Th_{\k,v}\be_\k^{-\f1n}r\int_{M\cap B_{\Th'_{\k,v}r}(x)}|\na f|.
\endaligned
\end{equation}
Then
\begin{equation}\aligned
&\int_{M\cap B_{r}(x)}|f-\bar{f}_{x,r}|=\int_{U^+_{r,0}}(f-\bar{f}_{x,r})-\int_{U^-_{r,0}}(f-\bar{f}_{x,r})\\
=&2\int_{U^+_{r,0}}(f-\bar{f}_{x,r})\le 2\Th_{\k,v}\be_\k^{-\f1n}r\int_{M\cap B_{\Th'_{\k,v}r}}|\na f|.
\endaligned
\end{equation}
This completes the proof by denoting $\Th^*_{\k,v}=\max\{2\Th_{\k,v}\be_\k^{-\f1n},\Th'_{\k,v}\}$.
\end{proof}

With Remark \ref{Sobalmost} and the proof of Lemma \ref{Cheeger***}, we can prove Theorem \ref{almostEB} by an argument in the proof of Theorem 3 in \cite{BG} as follows.
\begin{theorem}\label{PTh0varphi}
For each integer $n\ge2$, there are constants $\ep,\Th_0>0$ depending only on $n$ such that
if $N$ is an $(n+1)$-dimensional smooth complete noncompact Riemannian manifold with
$\mathrm{Ric}\ge-n\ep^2$ on $B_2(q)\subset N$ and $\mathcal{H}^{n+1}(B_1(q))\ge(1-\ep)\omega_{n+1}$, and $M$ is an area-minimizing hypersurface in $B_2(q)$ with $q\in M$ and $\p M\subset\p B_2(q)$, then
\begin{equation}\aligned
\min_{k\in\R}\left(\int_{M\cap B_{1/\Th_0}(q)}|\varphi-k|^{\f n{n-1}}\right)^{\f{n-1}n}\le\Th_0\int_{M\cap B_1(q)}|\na \varphi|
\endaligned
\end{equation}
for any function $\varphi\in C^1(B_1(q))$.
\end{theorem}
\begin{proof}
From Remark \ref{Sobalmost} and the proof of Lemma \ref{Cheeger***}, there is a positive constant $\Th_0\ge1$ depending only on $n$ such that
\begin{equation}\aligned\label{SB1qTh0}
\mathcal{H}^{n-1}(\p S\cap B_1(q))\ge\f1{\Th_0}\left(\min\{\mathcal{H}^n(S\cap B_{1/\Th_0}(q)),\mathcal{H}^n(M\cap B_{1/\Th_0}(q)\setminus S)\}\right)^{\f{n-1}n}
\endaligned
\end{equation}
for each $n$-rectifiable set $S\subset M\cap B_1(q)$ with $\mathcal{H}^n(S)>0$. 
Without loss of generality, let $\varphi$ be not a constant. 

Now we follow the argument in the proof of Theorem 3 in \cite{BG}.
For any fixed $k\in\R$, $x\in M$,  $r>0$ with  $B_{\Th_0r}(x)\subset B_1(q)$, let
$U^+_{s,t}=\{y\in M\cap B_s(x)|\ \varphi>k+t\}$, $U^-_{s,t}=\{y\in M\cap B_s(x)|\ \varphi<k+t\}$ for all $s\in(0,\Th_0r)$ and $t\in\R$.
We choose $k\in\R$ so that $\max\{\mathcal{H}^n(U^+_{r,0}),\mathcal{H}^n(U^-_{r,0})\}\le\f12\mathcal{H}^n(M\cap B_r(x))$. Then $\mathcal{H}^n(U^+_{r,t})\le\mathcal{H}^n(M\cap B_r(x)\setminus U^+_{r,t})$ for any $t\ge0$.
From \eqref{SB1qTh0}, we have
$$\mathcal{H}^{n-1}\left(\p U^+_{\Th_0r,t}\cap B_{\Th_0r}(x)\right)\ge\f1{\Th_0}\left(\mathcal{H}^n(U^+_{r,t})\right)^{\f{n-1}n}.$$
Combining co-area formula and an inequality from Hardy-Littlewood-P\'olya, we have
\begin{equation}\aligned
\Th_0&\int_{U^+_{\Th_0r,0}}|\na \varphi|\ge\Th_0\int_0^\infty\mathcal{H}^{n-1}\left(\p U^+_{\Th_0r,t}\cap B_{\Th_0r}(x)\right)dt
\ge \int_0^\infty\left(\mathcal{H}^n(U^+_{r,t})\right)^{\f{n-1}n}dt\\
\ge&\left(\f{n}{n-1}\int_0^\infty t^{1/(n-1)}\left(\mathcal{H}^n(U^+_{r,t})\right)dt\right)^{\f{n-1}n}
=\left(\int_{U^+_{r,0}}(\varphi-k)^{\f n{n-1}}\right)^{\f{n-1}n}.
\endaligned
\end{equation}
By considering $U^-_{r,t}$ for $t\le0$, we similarly get
\begin{equation}\aligned
\Th_0\int_{U^-_{\Th_0r,0}}|\na \varphi|\ge\left(\int_{U^-_{r,0}}(k-\varphi)^{\f n{n-1}}\right)^{\f{n-1}n}.
\endaligned
\end{equation}
Adding the above two inequalities implies
\begin{equation}\aligned
\left(\int_{M\cap B_{r}(x)}|\varphi-k|^{\f n{n-1}}\right)^{\f{n-1}n}\le \Th_0\int_{M\cap B_{\Th_0r}(x)}|\na \varphi|.
\endaligned
\end{equation}
This completes the proof.
\end{proof}

From Bishop-Gromov volume comparison and \eqref{ben***}, exterior balls of the minimal graph $M$ satisfy volume doubling property in the following sense: there is a positive constant $c_{n,\k,v}$ depending only on $n,\k,v$ such that
\begin{equation}\aligned\label{VDPmg}
\mathcal{H}^n\left(M\cap B_{2r}(x)\right)\le c_{n,\k,v}\mathcal{H}^n\left(M\cap B_{r}(x)\right)
\endaligned
\end{equation}
for any $B_{2r}(x)\subset B_{\f74}(p)\times\R$ with $x\in M\subset\Si\times\R$.
Suppose $\bar{p}=(p,0)\in M$.
\begin{lemma}\label{MVI}
Let $\phi$ be a nonnegative subharmonic function on $M\cap B_{2}(\bar{p})$ and $f$ be a nonnegative superharmonic function on $M\cap B_{2}(\bar{p})$, then there is a constant $\vartheta_{\k,v}$ depending only on $n,\k,v$ such that
\begin{equation}\aligned\label{subMVI}
\sup_{M\cap B_\f12(\bar{p})}\phi\le\vartheta_{\k,v}\int_{M\cap B_{1}(\bar{p})}\phi,
\endaligned
\end{equation}
and
\begin{equation}\aligned\label{supMVI}
\int_{M\cap B_{1}(\bar{p})}f\le\vartheta_{\k,v}\inf_{M\cap B_1(\bar{p})}f.
\endaligned
\end{equation}
\end{lemma}
\begin{proof}
We can carry out the De Giorgi-Nash-Moser iteration with the help of volume doubling property \eqref{VDPmg} and Theorem \ref{SobMinG} (see also $\S$ 4 in \cite{D0}), and deduce the mean value inequality \eqref{subMVI} on $M$.
Moreover, combining Theorem \ref{POINCARE} there is  a constant $\de_n>0$ depending only on $n$ such that (see also $\S$ 4 in \cite{D0})
\begin{equation}\aligned\label{supMVIde}
\int_{M\cap B_{\f32}(\bar{p})}f^{\de_n}\le\vartheta_{\k,v}\inf_{M\cap B_{1}(\bar{p})}f^{\de_n}
\endaligned
\end{equation}
up to a choice of the constant $\vartheta_{\k,v}$. 

Now we prove \eqref{supMVI} using \eqref{supMVIde} with Sobolev inequalities several times (following the idea of the proof of Theorem 6 in \cite{BG}).
Without loss of generality, we may assume $f>0$ on $M\cap B_{2}(\bar{p})$, or else $f\equiv0$ by the mean value inequality similarly as \eqref{supMVIde}.
Let $\De,\mathrm{div}_M$ denote the Laplacian and divergence of $M$, respectively.
For any fixed constant $\a\in(0,1/2)$, let $\psi=f^\a$, then (compared with (5.13) in \cite{BG}) one has
\begin{equation}\aligned\label{Depsia}
\De \psi=\mathrm{div}_M(\a f^{\a-1}\na f)=\a f^{\a-1}\De f+\a(\a-1)f^{\a-2}|\na f|^2\le\f{\a-1}{\a}\psi^{-1}|\na\psi|^2.
\endaligned
\end{equation}
Let $\e$ be a Lipschitz function on $M\cap B_{2}(\bar{p})$ with compact support in $M\cap B_{2}(\bar{p})$. Integrating by parts for \eqref{Depsia}, then with Cauchy inequality we get
\begin{equation}\aligned
&\f{\a-1}{\a}\int|\na\psi|^2\e^2\ge\int\psi\e^2\De \psi=-\int|\na\psi|^2\e^2-2\int\psi\e\lan\na\psi,\na\e\ran\\
\ge&-\int|\na\psi|^2\e^2-\f{1-2\a}{2\a}\int|\na\psi|^2\e^2-\f{2\a}{1-2\a}\int|\na\e|^2\psi^2,
\endaligned
\end{equation}
which implies
\begin{equation}\aligned
\int|\na\psi|^2\e^2\le\left(\f{2\a}{1-2\a}\right)^2\int|\na\e|^2\psi^2.
\endaligned
\end{equation}
Hence, with Cauchy inequality again we have
\begin{equation}\aligned
\int|\na(\psi\e)|^2\le2\int|\na\psi|^2\e^2+2\int|\na\e|^2\psi^2\le2\left(\left(\f{2\a}{1-2\a}\right)^2+1\right)\int|\na\e|^2\psi^2.
\endaligned
\end{equation}
Combining Sobolev inequality and H\"older inequality, from the above inequality (up to choosing cut-off functions) we have
\begin{equation}\aligned\label{f2an-2n}
\left(\int_{M\cap B_{1}(\bar{p})} f^{\f{2\a n}{n-2}}\right)^{\f{n-2}n}\le \vartheta_{\a,\k,v}\left(\int_{M\cap B_{3/2}(\bar{p})}f^{\de_n}\right)^{\f{2\a}{\de_n}},
\endaligned
\end{equation}
where $\vartheta_{\a,\k,v}$ is a constant depending only on $n,\a,\k,v$ for all $\a\in(0,1/2)$. Combining \eqref{supMVIde} and \eqref{f2an-2n}, we can get \eqref{supMVI} up to a choice of the constant $\vartheta_{\k,v}$.
\end{proof}

\section{Gradient estimates for minimal graphs}

Let $\Si$ be an $n$-dimensional complete non-compact manifold with Levi-Civita connection $D$, and $M$ be a smooth minimal graph over an open set $\Om\subset\Si$ in $\Si\times\R$ with the graphic function $u$ on $\Om$.
Denote $v_u=\sqrt{1+|Du|^2}$. 
Let $\De,\na$ denote the Laplacian and Levi-Civita connection of $M$ (for simplicity), respectively,
and $A$ denotes the second fundamental form of $M$ in $\Si\times\R$.
For a function $f$ on $\Om$,
we also see $f$ being the function on $M$ by letting $f(x,u(x))=f(x)$ for each $x\in\Om$, which will not cause confusion from the context in general. We still use $d\mu$ denote the volume element of $M$ as before, which can also be seen as an $n$-form on $\Si$. 
Let $d\mu_\Si$ denote the volume element of $\Si$. Then
\begin{equation}\aligned\label{dmudmuSi}
d\mu=v_ud\mu_\Si.
\endaligned
\end{equation}
Let us recall a Bochner type formula as follows (see \cite{DJX2} for instance):
\begin{equation}\aligned\label{De1/v}
\De\left(\f1{v_u}\right)=-\f{|A|^2}{v_u}-\f{\mathrm{Ric}(Du,Du)}{v_u^3},
\endaligned
\end{equation}
where $'\mathrm{Ric}'$ denotes the Ricci curvature of $\Si$.
Equivalently,
\begin{equation}\aligned\label{Logv}
\De\log v_u=|A|^2+\f{\mathrm{Ric}(Du,Du)}{1+|Du|^2}+|\na\log v_u|^2.
\endaligned
\end{equation}
For $\bar{q}=(q,t_q)\in \Si\times\R$, let $B_r(\bar{q})$ be a geodesic ball in $\Si\times\R$ with radius $r$ and centered at $\bar{q}$. Namely,
\begin{equation}\nonumber
B_r(\bar{q})=\left\{(x,t)\in\Si\times\R\big|\ d^2(x,q)+(t-t_q)^2<r^2\right\}.
\end{equation}
Let $\pi$ be the projection from $\Si\times\R$ into $\Si$ by $\pi(x,t)=x$ for all $(x,t)\in\Si\times\R$ as before.
\begin{theorem}\label{GEu}
Let $\Si$ be an $n$-dimensional complete noncompact manifold with Ricci curvature $\ge-(n-1)\k^2R^{-2}$ on the geodesic ball $B_R(p)\subset\Si$ for some constant $\k\ge0$. Assume $\mathcal{H}^n(B_R(p))\ge vR^n$ for some positive constant $v$.
If $M$ is a minimal graph over $B_{R}(p)$ with the graphic function $u$, then we have the gradient estimate
\begin{equation}\aligned\label{GradEstu}
|Du(p)|\le e^{c_{\k,v} R^{-1}\left(u(p)-\inf_{x\in B_{R}(p)}u(x)\right)},
\endaligned
\end{equation}
where $c_{\k,v}$ is a positive constant depending only on $n,\k,v$.
\end{theorem}
\begin{proof}
By scaling, we only need to prove the case $R=3$.
We follow the proof of Lemma 3.1 in \cite{DJX2} with nonnegative Ricci curvature of $\Si$ in \cite{DJX2} instead by Ricci curvature $\ge-\f19(n-1)\k^{2}$ on $B_3(p)$. 
We define a Lipschitz function $\z$ on $[0,\infty)$ with $\mathrm{supp}\z\subset [0,2]$, $\z\big|_{[0,1)}\equiv1$, $|\z'|\le1$, and a Lipschitz function $\tau$ on $\R$ with $0\le\tau\le1$, $\tau\equiv1$ in $(-1,\sup_{B_{2}}u)$, $\tau\equiv0$ outside $(-2,1+\sup_{B_{2}}u)$, $|\tau'|<1$.
Set $\eta(x)=\z(d(x,p))$ and $\xi(x,t)=\eta(x)\tau(t)$ for each $(x,t)\in\Si\times\R$. From \eqref{Logv} and the Ricci curvature condition of $\Si$, we have
\begin{equation}\aligned
\De\log v_u\ge|\na\log v_u|^2-\f{n-1}9\k^2.
\endaligned
\end{equation}
For simplicity, we will omit the volume element $d\mu$ in the following integrations on $M$. 
From integrating by parts and Cauchy inequality,
\begin{equation}\aligned
\int_M|\na\log v_u|^2\xi^2\le&\int_M\xi^2\left(\De\log v_u+ \f{n-1}9\k^2\right)\\
=&-2\int_M\xi\na\xi\cdot\na\log v_u +\f{n-1}{9}\k^2\int_M\xi^2\\
\le&\f12\int_M|\na\log v|^2\xi^2+2\int_M|\na\xi|^2+\f{n-1}{9}\k^2\int_M\xi^2.
\endaligned
\end{equation}
With \eqref{dmudmuSi}, we get
\begin{equation}\aligned\label{nalogv}
\int_M|\na\log v_u|^2\xi^2d\mu\le&4\int_M|\na\xi|^2+\f{2(n-1)}{9}\k^2\int_M\xi^2\\
\le&8\int_M\left(|\na\eta|^2\tau^2+|\na\tau|^2\eta^2\right)+\f{2(n-1)}{9}\k^2\int_M\e^2\tau^2\\
\le&\left(16+\f{2(n-1)}{9}\k^2\right)\int_{B_{2}(p)\cap\{u>-2\}}v_u d\mu_\Si.
\endaligned
\end{equation}
Now we follow the proof of Lemma 3.1 in \cite{DJX2} by choosing $\be=1$ there, and obtain
\begin{equation}\aligned\label{DuK}
\int_{B_1(p)\cap\{|u|<1\}}v_u\log v_u d\mu_\Si\le c_{n,\k}'\left(1+\sup_{B_{3}(p)}u\right)\mathcal{H}^n(B_{3}(p)),
\endaligned
\end{equation}
where $c_{n,\k}'>0$ is a constant depending only on $n,\k$.
Combining Bishop-Gromov volume comparison and \eqref{ben***}, one has
\begin{equation}\aligned\label{ambBallMball}
\mathcal{H}^n(B_{3r}(p))\le c_{n,\k}''\mathcal{H}^n(M\cap B_r(\bar{p}))\qquad \mathrm{for\ any}\ r\in(0,1]
\endaligned
\end{equation}
with some constant $c_{n,\k}''>0$ depending only on $n,\k$.
Combining \eqref{DuK}\eqref{ambBallMball} and the mean value inequality \eqref{subMVI}, we get the gradient estimate \eqref{GradEstu}.
\end{proof}

The above gradient estimate is local. If we assume a global condition on $\Si$, and define the minimal graphic function $u$ on $\Si$, then using the idea of Bombieri-Giusti in \cite{BG} we get a gradient estimate in terms of polynomial growth of $u$ as follows (compared (6.8) in \cite{BG}).
\begin{theorem}\label{GEuglobal}
Let $\Si$ be an $n$-dimensional complete noncompact manifold with nonnegative Ricci curvature and Euclidean volume growth
\begin{equation}\aligned\label{EVol}
\lim_{r\rightarrow\infty}\f{\mathcal{H}^{n}(B_r(x))}{r^{n}}\ge v
\endaligned
\end{equation}
for some constant $v>0$.
If $M$ is a minimal graph over $\Si$ with the graphic function $u$, then
\begin{equation}\aligned\label{GradEstu*}
\sup_{B_r(x)}|Du|\le c_v \left(1+r^{-1}\sup_{y\in B_{r}(x)}|u(y)-u(x)|\right)^n
\endaligned
\end{equation}
for any geodesic ball $B_r(x)\subset\Si$, where $c_v$ is a positive constant depending only on $n$, $v$.
\end{theorem}
\begin{proof}
From \eqref{supMVI} and \eqref{De1/v}, 
\begin{equation}\aligned\label{supmvivu}
\f1{\mathcal{H}^n(M\cap B_{R}(\bar{x}))}\int_{M\cap B_R(\bar{x})}\f1{v_u}\le\vartheta_{0,v}\inf_{M\cap B_{R}(\bar{x})}\f1{v_u}
\endaligned
\end{equation}
for any $\bar{x}=(x,u(x))$ and any $R>0$, 
which implies
\begin{equation}\aligned
\f{\mathcal{H}^n(\pi(M\cap B_{R}(\bar{x})))}{\mathcal{H}^n(M\cap B_{R}(\bar{x}))}\le\f{\vartheta_{0,v}}{\sup_{M\cap B_{R}(\bar{x})}v_u}.
\endaligned
\end{equation}
Namely,
\begin{equation}\aligned
\sup_{\pi(M\cap B_{R}(\bar{x}))}\sqrt{1+|Du|^2}\le\vartheta_{0,v}\f{\mathcal{H}^n(M\cap B_{R}(\bar{x}))}{\mathcal{H}^n(\pi(M\cap B_{R}(\bar{x})))}.
\endaligned
\end{equation}
Since $M$ is area-minimizing in $\Si\times\R$, then
\begin{equation}\aligned
\mathcal{H}^n(M\cap B_{R}(\bar{x}))\le\f12\mathcal{H}^n(\p B_{R}(\bar{x}))\le \f{n+1}2\omega_{n+1}R^n,
\endaligned
\end{equation}
where $\omega_{n+1}$ is the volume of $(n+1)$-dimensional unit Euclidean ball.
Combining the above two inequalities, we can obtain
\begin{equation}\aligned\label{SLOPEBD}
\sup_{\pi(M\cap B_{R}(\bar{x}))}|Du|\le \f{(n+1)\omega_{n+1}\vartheta_{0,v} R^n}{2\mathcal{H}^n(\pi(M\cap B_{R}(\bar{x})))}.
\endaligned
\end{equation}
Denote $\bar{x}=(x,u(x))$. There exists a unique $r>0$ (depending on $R$) such that
$$R^2=r^2+\left(\sup_{B_r(x)}|u-u(x)|\right)^2.$$
Clearly, $B_r(x)\subset \pi(M\cap B_{R}(\bar{x}))$. With \eqref{EVol} and Bishop-Gromov volume comparison, it follows that $\mathcal{H}^n(B_r(x))\ge vr^n$ for every $x\in\Si$ and every $r>0$. From \eqref{SLOPEBD}, we have
\begin{equation}\aligned
&\sup_{B_r(x)}|Du|\le\sup_{\pi(M\cap B_{R}(\bar{x}))}|Du|\le\f{(n+1)\omega_{n+1}\vartheta_{0,v} R^n}{2\mathcal{H}^n(B_r(x))}\\
\le&\f{(n+1)\omega_{n+1}\vartheta_{0,v}}{2v}\left(\f Rr\right)^n=\f{(n+1)\omega_{n+1}\vartheta_{0,v}}{2v}\left(1+\f1{r^2}\left(\sup_{B_r(x)}|u-u(x)|\right)^2\right)^{\f n2},
\endaligned
\end{equation}
which implies \eqref{GradEstu*}.
\end{proof}

Let $\Si$ be a smooth complete manifold of nonnegative Ricci curvature and Euclidean volume growth.
Let $M$ be an entire smooth minimal graph over $\Si$ with the graphic function $u$. If $u$ has linear growth for its negative part, i.e.,
\begin{equation}\aligned\label{LinearGrowthos}
\limsup_{\Si\ni x\rightarrow\infty}\f{\max\{-u(x),0\}}{d(x,p)}<\infty
\endaligned
\end{equation}
for a fixed point $p\in\Si$, then $|Du|$ is uniformly bounded on $\Si$ from Theorem \ref{GEu}.
Put $|Du|_0=\sup_{\Si}|Du|<\infty$, then $e^{2\log v_u}-1=|Du|^2$ is a bounded subharmonic function on $M$ from \eqref{Logv}.
From \eqref{supMVI}, for any fixed point $\bar{x}=(x,u(x))$ we have
\begin{equation}\aligned
\fint_{M\cap B_r(\bar{x})}\left(|Du|^2_0-|Du|^2\right)\le \vartheta_{0,v}\left(|Du|^2_0-|Du|^2(x)\right),
\endaligned
\end{equation}
which implies
\begin{equation}\aligned\label{meanlimDu1}
\sup_{\Si}|Du|^2=\lim_{r\rightarrow \infty}\fint_{M\cap B_r(\bar{x})}|Du|^2.
\endaligned
\end{equation}
By following the proof of Theorem 3.6 in \cite{DJX2} with \eqref{meanlimDu1}, we can prove Theorem \ref{hlcMG}.

\section{Splitting theorems for minimal graphs}

Let $\Si$ be an $n$-dimensional smooth complete manifold of nonnegative Ricci curvature for $n\ge3$. Suppose $\Si$ has Euclidean volume growth
\begin{equation}\aligned
\lim_{r\to\infty}r^{-n}\mathcal{H}^n(B_r(p))=v
\endaligned
\end{equation}
for some constant $v>0$.
Let $\De_\Si$, $\mathrm{Hess}$ denote the Laplacian and Hessian on $\Si$, respectively.
For $p\in\Si$, let $G(p,\cdot)$ be the Green function on $\Si$ with $\lim_{r\rightarrow0}\sup_{\p B_r(p)}\left|Gr^{n-2}-1\right|=0$. We fix the point $p$, and denote $b=G^{\f1{2-n}}$, then
\begin{equation}\aligned\label{Deb}
\De_{\Si}b^2=2n|D b|^2,\qquad \mathrm{and} \qquad b\De_{\Si}b=(n-1)|D b|^2.
\endaligned
\end{equation}
From Mok-Siu-Yau \cite{MSY}, there is a constant $\de_0\in(0,1]$ such that
\begin{equation}\aligned\label{de0dbf}
\de_0d(x,p)\le b(x)\le d(x,p)\qquad\qquad\mathrm{ for\ large}\ d(x,p).
\endaligned
\end{equation}
From Colding \cite{C2} and Colding-Minicozzi \cite{CM1}, we have $|Db|\le1$ on $\Si$, and
\begin{equation}\aligned\label{Estb000}
\lim_{\Si\ni x\rightarrow\infty}\f{b(x)}{d(p,x)}=\lim_{r\rightarrow\infty}\sup_{\Si\setminus B_r(p)}|Db|=\left(\f v{\omega_n}\right)^{\f1{n-2}}.
\endaligned
\end{equation}
Denote $v_*=\left(\f v{\omega_n}\right)^{\f1{n-2}}$. Then we can let $\de_0=\f12v_*$ in \eqref{de0dbf} from \eqref{Estb000}.
For simplicity, we will omit the volume element $d\mu_\Si$ on integrations over subsets of $\Si$ if no ambiguity.
From \cite{CC}\cite{CM1} (see also \cite{C2}), for any constant $\La>1$ we have
\begin{equation}\aligned\label{Estb}
\lim_{r\rightarrow\infty}r^{-n}\int_{B_{\La r}(p)\setminus B_r(p)}\left(\left||Db|^2-v_*^2\right|+\left|\mathrm{Hess}_{b^2}-2v_*^2\si\right|\right)=0,
\endaligned
\end{equation}
where $\si$ is the metric of $\Si$.

Let $M\subset\Si\times\R$ be a minimal graph over $\Si$ with the graphic function $u$.
Suppose that $u$ has linear growth for its negative part, i.e., \eqref{LinearGrowthos} holds. Then $|Du|$ is uniformly bounded on $\Si$ from Theorem \ref{GEu}. Suppose $\sup_\Si|Du|>0$.
\begin{lemma}
For any $x\in\Si$, one has
\begin{equation}\aligned\label{Hessu2}
0=\lim_{r\rightarrow \infty}\f1{r^{n-2}}\int_{B_r(x)}|\mathrm{Hess}_u|^2.
\endaligned
\end{equation}
\end{lemma}
\begin{proof}
At an considered point in $M$, we choose an orthonormal basis $\{e_i\}_{i=1}^n$, and let $\nu$ denote the unit normal vector to $M$. Set $\overline{\na}$ and $\overline{\mathrm{Hess}}$ be the Levi-Civita connection and Hessian matrix of $\Si\times\R$ with respect to the standard product metric $dt^2+\si$, respectively.
Since $M$ is minimal, then $\sum_{i=1}^n\left({\bn_{e_i}e_i}-\na_{e_i}e_i\right)=0$.
For any smooth function $f$ on $\Si$, we define a function $F(x,t)=f(x)$ for each $x\in\Si$ and $t\in\R$. Then
\begin{equation}\aligned\label{DeFDef}
\De F=&\sum_{i=1}^n\left(\na_{e_i}\na_{e_i}F-\left({\na_{e_i}e_i}\right)F\right)\\
=&\sum_{i=1}^n\left(\bn_{e_i}\bn_{e_i}F-\left({\bn_{e_i}e_i}\right)F\right)
+\sum_{i=1}^n\left({\bn_{e_i}e_i}-\na_{e_i}e_i\right)F\\
=&\De_{\Si\times\R}F-\overline{\mathrm{Hess}}_{F}(\nu,\nu)=\De_{\Si}f-\f1{v_u^2}\mathrm{Hess}_{f}(Du,Du).
\endaligned
\end{equation}
Let $\z$ be a nonnegative smooth function on $[0,\infty)$ with $\mathrm{supp}\z\subset [0,2r]$, $\z\big|_{[0,r]}\equiv1$, $|\z'|\le 10r^{-1}$ and $|\z''|\le 10r^{-2}$. Set
$$\eta(x,t)=\z(b(x))\qquad \mathrm{for\ any}\ (x,t)\in\Si\times\R.$$
With \eqref{de0dbf} and $\de_0=\f12v_*$, it follows that supp$\e\subset \overline{B_{4v_*^{-1}r}(p)}\times\R$ for large $r$.

From \eqref{DeFDef},
\begin{equation}\aligned
\De\e=\z'\De_\Si b+\z''|D b|^2-\f{\z'}{v_u^2}\mathrm{Hess}_{b}(Du,Du)-\f{\z''}{v_u^2}\left|\lan D b,Du\ran\right|^2,
\endaligned
\end{equation}
and then
\begin{equation}\aligned\label{Deeta}
|\De\eta|\le\f c{r^2}\left(|Db|^2+\left|\mathrm{Hess}_{b^2}\right|\right),
\endaligned
\end{equation}
where $c$ is a generic positive constant depending only on $n$. 
Denote $\bar{p}=(p,u(p))\in\Si\times\R$ and
$$\La_0=\f12\log\left(1+\sup_\Si|Du|^2\right).$$
Noting $\e=1$ on $M\cap(B_{r}(p)\times\R)$ since $\z\big|_{[0,r]}\equiv1$. Recalling that $d\mu$ is the volume element of $M$.
Combining \eqref{Logv} and \eqref{Deeta}, we have
\begin{equation}\aligned\label{|A|2DbHessb}
\int_{M\cap B_r(\bar{p})}|A|^2d\mu\le&\int_{M}|A|^2\eta d\mu\le\int_{M}\eta\De\log v_u d\mu=\int_{M}(\log v_u-\La_0)\De\eta d\mu\\
\le&\f c{r^2}e^{\La_0}\int_{B_{4v_*^{-1}r}(p)\setminus B_r(p)}\left(\La_0-\log v_u\right)\left(|Db|^2+\left|\mathrm{Hess}_{b^2}\right|\right)\\
\le&\f c{r^2}e^{\La_0}\La_0\int_{B_{4v_*^{-1}r}(p)\setminus B_r(p)}\left(|Db|^2-v_*^2+\left|\mathrm{Hess}_{b^2}-2v_*^2\si\right|\right)\\
&+\f c{r^2}e^{\La_0}\int_{B_{4v_*^{-1}r}(p)\setminus B_r(p)}(\La_0-\log v_u).
\endaligned
\end{equation}
From \eqref{meanlimDu1}, it follows that
\begin{equation}\aligned
\lim_{r\to\infty}\f1{r^n}\int_{B_{4v_*^{-1}r}(p)\setminus B_r(p)}(\La_0-\log v_u)=0.
\endaligned
\end{equation}
Combining \eqref{Estb} and \eqref{|A|2DbHessb}, we obtain
\begin{equation}\aligned
\lim_{r\rightarrow \infty}\f1{r^{n-2}}\int_{M\cap B_r(\bar{p})}|A|^2d\mu=0,
\endaligned
\end{equation}
which implies the equality \eqref{Hessu2}.
\end{proof}

Suppose $u(p)=0$.
Let  $\bar{p}=(p,0)\in M\subset\Si\times\R$.
For each $r>0$, let $(\Si_r,p_r)=\f1r(\Si,p)$, and $(M_r,\bar{p}_r)=\f1r(M,\bar{p})$. Then $M_r\subset\Si_r\times\R$ is a minimal graph over $\Si_r$.
Let $\tilde{u}_r$ denote the graphic function of $M_r$ on $\Si_r$.
For convenience, we still denote the inner product, norm, gradient and Hessian for this re-scaled
metric $r^{-2}\si$ by $\lan\cdot,\cdot\ran$, $|\cdot|$, $D$ and $\mathrm{Hess}$, respectively.
Put $|Du|_0=\sup_\Si|Du|>0$ as before. From \eqref{meanlimDu1}\eqref{Hessu2}, we get
\begin{equation}\aligned\label{supDur}
|Du|_0^2=\lim_{r\rightarrow \infty}\fint_{B_R(p_r)}|D\tilde{u}_r|^2
\endaligned
\end{equation}
and
\begin{equation}\aligned\label{Hessur}
0=\lim_{r\rightarrow \infty}\int_{B_R(p_r)}\left|\mathrm{Hess}_{\tilde{u}_r}\right|^2
\endaligned
\end{equation}
for any $R>0$.
With Cauchy inequality, it follows that
\begin{equation}\aligned\label{Hessur1}
0=\lim_{r\rightarrow \infty}\int_{B_R(p_r)}\left|\mathrm{Hess}_{\tilde{u}_r}\right|.
\endaligned
\end{equation}
If $F_r$ is a subset in $B_R(p_r)$ with $\mathcal{H}^n(F_r)\ge\de$ for some constant $\de>0$, then we have
\begin{equation}\aligned
0\le&\fint_{F_r}\left(|Du|_0-|D\tilde{u}_r|\right)\le\f1{\de}\int_{F_r}\left(|Du|_0-|D\tilde{u}_r|\right)\\
\le&\f1{\de}\int_{B_R(p_r)}\left(|Du|_0-|D\tilde{u}_r|\right)\le\f1{\de|Du|_0}\int_{B_R(p_r)}\left(|Du|_0^2-|D\tilde{u}_r|^2\right).
\endaligned
\end{equation}
With \eqref{supDur}, it follows that
\begin{equation}\aligned\label{supDur1}
0=\lim_{r\rightarrow \infty}\fint_{F_r}\left(|Du|_0-|D\tilde{u}_r|\right).
\endaligned
\end{equation}

For almost all $x,y\in \Si_r$ there exist a unit vector $\mathrm{v}_{xy}\in\R^n$ and a unique minimal geodesic
$\g_{xy}$ connecting $x$ and $y$ such that $\left|\dot{\g}_{xy}\right|=d(x,y)$, $\g_{xy}(0)=x$ and $\g_{xy}(1)=y$.
Denote $\mathrm{v}_{xy}=\f{\dot{\g}_{xy}(0)}{|\dot{\g}_{xy}(0)|}$.
Parallel translating the vector $D\tilde{u}_r(x)$ along $\g_{xy}(t)$ gets a vector at $y$, denoted by $P_{y}(D\tilde{u}_r(x))$. Then we have
\begin{equation}\aligned\label{Dpixyu}
\left|P_{y}(D\tilde{u}_r(x))-D\tilde{u}_r(y)\right|=\left|D\tilde{u}_r(x)-P_{x}(D\tilde{u}_r(y))\right|
\le d(x,y)\int_0^{1}\left|\mathrm{Hess}_{\tilde{u}_r}\right|(\g_{xy}(t))dt.
\endaligned
\end{equation}
For any unit vector $\mathrm{v}\in\R^n$, $r>0$, $\ep\in(0,1)$, $x\in\Si_r$, we define a cone-like region $C^{r,\ep}_{x,\mathrm{v}}$ in $\Si_r$ by
\begin{equation}\aligned
C^{r,\ep}_{x,\mathrm{v}}=\{\mathrm{exp}_x(t\th)\in\Si_r|\ t>0,\ \lan \mathrm{v},\th\ran>1-\ep\}.
\endaligned
\end{equation}
Note that $\Si_r$ is an $n$-dimensional smooth complete manifold of nonnegative Ricci curvature and Euclidean volume growth $\lim_{s\to\infty}s^{-n}\mathcal{H}^n(B_s(p_r))=v>0$.
By an argument of contradiction, there is a function $\psi_{\de,l}$ (depending on $v$) on $\{(\de,l)\in\R^2|\,\de,l>0\}$ satisfying $\lim_{\de\to0}\psi_{\de,l}=0$ so that
\begin{equation}\aligned\label{BdeyCxyr}
B_\de(y)\subset C^{r,\psi_{\de,l}}_{x,\mathrm{v}_{xy}}\qquad \mathrm{for\ each}\ y\in B_{2l}(x)\setminus B_{l/2}(x),\ x\in\Si_r, \ r>0.
\endaligned
\end{equation}
Set $\G_r$ be the level set of $\tilde{u}_r$ at the value 0 in $\Si_r$. Namely,
$$\G_{r}=\{x\in\Si_r|\ \tilde{u}_r(x)=0\}.$$
\begin{lemma}\label{C0est}
For any fixed $R>0$, we have
\begin{equation}\aligned
\lim_{r\rightarrow\infty}\sup_{x_r\in B_R(p_r)}\left|\tilde{u}_r(x_r)-d(x_r,\G_r)|Du|_0\right|=0.
\endaligned
\end{equation}
\end{lemma}
\begin{proof}
We fix a point $x_r\in B_R(p_r)$ and a constant $R>0$, and denote $l_r=d(x_r,p_r)$.
For any $x\in B_\de(x_r)$, there are a point $x'\in\G_{r}$ and a minimal geodesic $\g_{xx'}$ connecting $x,x'$ with $\g_{xx'}(0)=x$, $\g_{xx'}(1)=x'$ and $|\dot{\g}_{xx'}|=l_x=d(x,\G_r).$ Denote $\mathrm{v}_{xx'}=\f{\dot{\g}_{xx'}(0)}{|\dot{\g}_{xx'}(0)|}$. 
For each small $0<\de<l_r/4$ and almost all $y\in B_\de(x')$, there is a unique minimal geodesic $\g_{xy}$ connecting $x,y$ with $\g_{xy}(0)=x$, $\g_{xy}(1)=y$ and $|\dot{\g}_{xy}|=l_{xy}\in(l_x-\de,l_x+\de)\subset(l_r-2\de,l_r+2\de)$.

If $|D\tilde{u}_r(x)|>0$, then $D\tilde{u}_r(x)$ is the normal vector to $\G_{r,x}$ at $x$, where $\G_{r,x}$ is the level set of $\tilde{u}_r$ containing $x$.
Since $B_{l_x}(x)\cap\G_r=\emptyset$ and $x'\in \overline{B_{l_x}(x)}\cap\G_r$, then
\begin{equation}\aligned
\left\lan D\tilde{u}_r(x),\mathrm{v}_{xx'}\right\ran=-|D\tilde{u}_r(x)|.
\endaligned
\end{equation}
The above equality is clearly true for $|D\tilde{u}_r(x)|=0$. 
Since $B_\de(x')\subset C^{r,\psi_{\de,l_r}}_{x,\mathrm{v}_{xx'}}$ for any $x\in B_\de(x_r)$ from \eqref{BdeyCxyr}, one has
\begin{equation}\aligned
\left\lan D\tilde{u}_r(x),\dot{\g}_{xy}(0)\right\ran=\left\lan D\tilde{u}_r(x),\mathrm{v}_{xx'}\right\ran\left\lan \dot{\g}_{xy}(0),\mathrm{v}_{xx'}\right\ran\le-(1-\psi_{\de,l_r})|D\tilde{u}_r(x)|l_{xy}
\endaligned
\end{equation}
for almost all $y\in B_\de(x')$. From $\tilde{u}_r(x')=0$, one gets
\begin{equation}\aligned
\left|\tilde{u}_r(y)\right|=\left|\tilde{u}_r(y)-\tilde{u}_r(x')\right|\le|Du|_0\de.
\endaligned
\end{equation}
Noting that $\g_{xy}$ is a geodesic with $|\dot{\g}_{xy}|=l_{xy}\le l_r+2\de$.
Combining Newton-Leibniz formula, \eqref{Dpixyu} and Fubini's theorem, we have
\begin{equation}\aligned\label{urdeDH}
&\tilde{u}_r(x)+|Du|_0\de\ge\tilde{u}_r(x)-\tilde{u}_r(y)\\
=&-\int_0^{1}\f{d}{dt}\tilde{u}_r(\g_{xy}(t))dt
=-\int_0^{1}\left\lan D\tilde{u}_r(\g_{xy}(t)),\dot{\g}_{xy}(t)\right\ran dt\\
=&-\int_0^{1}\left\lan D\tilde{u}_r(\g_{xy}(t))-P_{\g_{xy}(t)}(D\tilde{u}_r(x)),\dot{\g}_{xy}(t)\right\ran dt-\int_0^{1}\left\lan D\tilde{u}_r(x),\dot{\g}_{xy}(0)\right\ran dt\\
\ge&-l_{xy}\int_0^{1}\left|P_{\g_{xy}(t)}(D\tilde{u}_r(x))-D\tilde{u}_r(\g_{xy}(t))\right|dt+(1-\psi_{\de,l_r})l_{xy}|D\tilde{u}_r(x)|\\
\ge&-l_{xy}^2\int_0^{1}\left(\int_0^t\left|\mathrm{Hess}_{\tilde{u}_r}\right|(\g_{xy}(s))ds\right)dt+(1-\psi_{\de,l_r})(l_r-2\de)|D\tilde{u}_r(x)|\\
=&-l_{xy}^2\int_0^{1}\left(\int_s^1\left|\mathrm{Hess}_{\tilde{u}_r}\right|(\g_{xy}(s))dt\right)ds+(1-\psi_{\de,l_r})(l_r-2\de)|D\tilde{u}_r(x)|\\
\ge&-(l_r+2\de)^2\int_0^{1}(1-s)\left|\mathrm{Hess}_{\tilde{u}_r}\right|(\g_{xy}(s))ds+(1-\psi_{\de,l_r})(l_r-2\de)|D\tilde{u}_r(x)|.
\endaligned
\end{equation}
By Bishop-Gromov volume comparison, for any nonnegative measurable function $f$ on $\Si_r$ one has
\begin{equation}\aligned
\int_{y\in B_R(x)}f(\g_{xy}(t))\le \f1{t^n}\int_{B_{tR}(x)}f.
\endaligned
\end{equation}
Integrating \eqref{urdeDH} for the variable $y$ on $B_\de(x')$ infers
\begin{equation}\aligned\label{urDHg}
&\left(\tilde{u}_r(x)+|Du|_0\de-(1-\psi_{\de,l_r})(l_r-2\de)|D\tilde{u}_r(x)|\right)\mathcal{H}^n\left(B_\de(x')\right)\\
\ge&-(l_r+2\de)^2\int_{y\in B_\de(x')}\int_0^{1}\left|\mathrm{Hess}_{\tilde{u}_r}\right|(\g_{xy}(t))dt\\
\ge&-(l_r+2\de)^2\int_{y\in B_{l_x+\de}(x)}\int_0^{1}\left|\mathrm{Hess}_{\tilde{u}_r}\right|(\g_{xy}(t))dt\\
\ge&-(l_r+2\de)^2\int_0^{1}\left(\f1{t^n}\int_{B_{t(l_x+\de)}(x)}\left|\mathrm{Hess}_{\tilde{u}_r}\right|\right)dt.
\endaligned
\end{equation}
By Bishop-Gromov volume comparison, 
\begin{equation}\aligned
\mathcal{H}^n\left(B_\de(x')\right)\ge\lim_{r\to\infty}\left(\f{\de}{r}\right)^n\mathcal{H}^n\left(B_{r}(x')\right)=\de^nv.
\endaligned
\end{equation}
From Fubini's theorem
\begin{equation}\aligned
\int_{x\in B_{R_2}(z)}\int_{B_{R_1}(x)}f\le\int_{y\in B_{R_1+R_2}(z)}\mathcal{H}^n(B_{R_1}(y))f(y)
\le\omega_n R_1^n\int_{B_{R_1+R_2}(z)}f
\endaligned
\end{equation}
for any nonnegative measurable function $f$ on $\overline{B_{R_1+R_2}(z)}$, one has
\begin{equation}\aligned\label{urHess}
&\int_{x\in B_\de(x_r)}\left(\f{(l_r+2\de)^2}{\mathcal{H}^n\left(B_\de(x')\right)}\int_0^{1}\left(\f1{t^n}\int_{B_{t(l_x+\de)}(x)}\left|\mathrm{Hess}_{\tilde{u}_r}\right|\right)dt\right)\\
\le&\f{(l_r+2\de)^2}{\de^nv}\int_0^{1}\left(\f1{t^n}\int_{x\in B_\de(x_r)}\int_{B_{t(l_r+2\de)}(x)}\left|\mathrm{Hess}_{\tilde{u}_r}\right|\right)dt\\
\le&\f{(l_r+2\de)^2}{\de^nv}\int_0^{1}\left(\omega_n(l_r+2\de)^n\int_{B_{l_r+3\de}(x_r)}\left|\mathrm{Hess}_{\tilde{u}_r}\right|\right)dt\\
=&\omega_n\f{(l_r+2\de)^{n+2}}{\de^nv}\int_{B_{l_r+3\de}(x_r)}\left|\mathrm{Hess}_{\tilde{u}_r}\right|.
\endaligned
\end{equation}
Note that $l_r= d(x_r,\G_r)\le d(x_r,p_r)\le R$.
Combining  \eqref{Hessur1}\eqref{urDHg} and \eqref{urHess}, there exists a constant $r_0>0$ such that
\begin{equation}\aligned\label{BdeFderv}
-\de\le&\fint_{B_\de(x_r)}\left(\tilde{u}_r(x)+|Du|_0\de-(1-\psi_{\de,l_r})(l_r-2\de)|D\tilde{u}_r(x)|\right)
\endaligned
\end{equation}
for all $r\ge r_0$.
Combining \eqref{supDur1}\eqref{BdeFderv}, for large $r_0$ we have
\begin{equation}\aligned\nonumber
-\de\le&\fint_{B_\de(x_r)}\tilde{u}_r(x)+|Du|_0\de-(1-\psi_{\de,l_r})(l_r-2\de)\left(|Du|_0-\de\right).
\endaligned
\end{equation}
Therefore, there is a point $x_*\in B_\de(x_r)$ such that
$$\tilde{u}_r(x_*)\ge(1-\psi_{\de,l_r})(l_r-2\de)\left(|Du|_0-\de\right)-(1+|Du|_0)\de.$$
Then
\begin{equation}\aligned
\tilde{u}_r(x_r)\ge\tilde{u}_r(x_*)-|Du|_0\de\ge(1-\psi_{\de,l_r})(l_r-2\de)\left(|Du|_0-\de\right)-(2+|Du|_0)\de.
\endaligned
\end{equation}
Letting $r\to\infty$ first and $\de\to0$ second in the above inequality implies 
$$\liminf_{r\to\infty}\left(\tilde{u}_r(x_r)- l_r|Du|_0\right)\ge0.$$
On the other hand, it's clear that $\tilde{u}_r(x_r)\le l_r|Du|_0$ by Newton-Leibniz formula. So we complete the proof.
\end{proof}

Combining Lemma \ref{C0est} and \eqref{supDur}\eqref{Hessur1}, integrating by parts implies (see the proof of Lemma 4.6 in \cite{DJX2})
\begin{equation}\aligned\label{C1est}
\lim_{r\rightarrow\infty}\int_{B_R(p_r)}\big|D\left(\tilde{u}_r(x)-d(x,\G_r)|Du|_0\right)\big|^2=0.
\endaligned
\end{equation}
For any sequence $R_i\to\infty$,
there is a subsequence $\{r_i\}$ of $\{R_i\}$ such that $\f1{r_i}(\Si,p)=(\Si_{r_i},p_{r_i})$ converges in the pointed Gromov-Hausdorff sense to a metric cone $(C,o)$ with the vertex at $o$, and $(\G_{r_i},p_{r_i})$ converges in the induced Hausdorff sense to $(X,o)$, where $X$ is a closed subset of $C$.
With Lemma \ref{C0est}, \eqref{C1est} and \eqref{Hessur1}, we can use Theorem 3.6 in \cite{CC} by Cheeger-Colding and deduce $C=X\times\R$, i.e., the tangent cone $C$ of $\Si$ at infinity splits off a line isometrically.
This completes the proof of Theorem \ref{splitting}.

Let $M_r$ denote the graph of $\tilde{u}_r$ in $\Si_r\times\R$. Without loss of generality, we assume $(M_{r_i},\bar{p}_{r_i})=\f1{r_i}(M,\bar{p})$ converges in the induced Hausdorff sense to $(C',o)$, where $C'$ is a closed subset in $C\times\R$. We call $C'$ \emph{a tangent cone of $M$ at infinity}.
From \eqref{meanlimDu1}\eqref{Hessu2}, Lemma \ref{C0est} and $(\G_{r_i},p_{r_i})\to(X,o)$, we conclude that
\begin{equation}\aligned
t_2^{-n}\lim_{i\to\infty}\mathcal{H}^n(M_{r_i}\cap B_{t_2}(\bar{p}_{r_i}))=t_1^{-n}\lim_{i\to\infty}\mathcal{H}^n(M_{r_i}\cap B_{t_1}(\bar{p}_{r_i}))
\endaligned
\end{equation}
for any $t_1,t_2>0$.
From Theorem 5.4 in \cite{D} and the argument in Theorem \ref{CovconeC}, we conclude that $C'$ is a metric cone in $C\times\R$ (the case $r_j\to0$ is similar to the case $r_j\to\infty$). From the above argument, $C'$ also splits off a line isometrically, i.e., the tangent cone $C'$ of the minimal graph $M$ at infinity splits off a line isometrically.
Furthermore, the splitting holds even independent of the linear growth of $u$ as follows.
\begin{theorem}
Let $\Si$ be a complete noncompact Riemannian manifold of nonnegative Ricci curvature and Euclidean volume growth. If $M$ is a minimal graph over $\Si$ with the non-constant smooth graphic function $u$, then there is a tangent cone of $M$ at infinity, which is a metric cone and splits off a line isometrically.
\end{theorem}
\begin{proof}
From Theorem \ref{hlcMG} and the above argument, 
we only need to consider the case $\sup_\Si|Du|=\infty$.
From Lemma \ref{MVI} and \eqref{De1/v}, we have 
\begin{equation}\aligned\label{En+1nuM}
\lim_{r\rightarrow\infty}\f1{r^n}\int_{M\cap B_r(p)}\f1{v_u}=0.
\endaligned
\end{equation}
From Theorem \ref{CovconeC}, there is a sequence $r_i\to\infty$ so that 
$\f1{r_i}(\Si,p)=(\Si_{r_i},p_{r_i})$ converges in the pointed Gromov-Hausdorff sense to a metric cone $(C,o)$ with the vertex at $o$, and
$\f1{r_i}(M_{r_i},\bar{p}_{r_i})$ converges in the induced Hausdorff sense to $(C',o)$, where $C'$ is a tangent cone of $M$ at infinity in $C\times\R$. Here, $C'$ is a metric cone.
If there are no closed sets $Y\subset C$ with $C'=Y\times\R$, then there is a point $z=(z',t_z)\in C'$ such that $\{(z',t)\in C\times\R|\ t>t_z\}\cap C'=\emptyset$. Therefore, there is a constant $0<\ep<<1$ such that
\begin{equation}\aligned\label{BepzC'tzk}
C'\cap B_{\ep}(z_*)=\emptyset\qquad \mathrm{with}\ z_*=(z',t_z+1).
\endaligned
\end{equation}
Let $z_i=(z_i',t_z)\in M_{r_i}$ be a sequence of points in $\Si_{r_i}\times\R$ with $z_i'\to z'$. Since $M_{r_i}\cap B_2(z_i)\to C'\cap B_2(z)$ in the induced Hausdorff sense, from \eqref{BepzC'tzk} it follows that
\begin{equation}\aligned
M_{r_i}\cap B_{\ep/2}(z_{i,*})=\emptyset\qquad \mathrm{with}\ z_{i,*}=(z_i',t_z+1)
\endaligned
\end{equation}
for suitably large $i$.
From Theorem \ref{GEu}, there is a constant $\de\in(0,\ep/2)$ (depending on $n,\ep$) so that the graphic function $\tilde{u}_{r_i}$ satisfies 
$$\sup_{B_\de(z_i')}|D\tilde{u}_{r_i}|<\f1\de\qquad \mathrm{for\ suitably\ large}\ i.$$ 
However, this contradicts to \eqref{En+1nuM} as $i\to\infty$. 
We complete the proof.
\end{proof}

\section{Appendix}

\begin{lemma}\label{Mollifybij}
Let $\Si$ be an $n$-dimensional complete Riemannian manifold, and $\r_p$ be the distance function from a point $p\in\Si$. Let $\mathcal{C}_p$ denote the cut locus w.r.t. p. There is a constant $\ep_*$ (depending only on $n$ and the geometry of $B_2(p)$) such that for any $\ep\in(0,\ep_*]$, there is a smooth function $f_\ep$ on $\overline{B_1(p)}$ satisfying $f_\ep=\r_p$ on $\overline{B_1(p)}\setminus B_\ep(\mathcal{C}_p)$ and $|Df_\ep |\le1+\ep$. 
\end{lemma}
\begin{proof}
The idea is using cut-off functions and smoothers of distance functions in a small tubular neighborhood of $\mathcal{C}_p$. We give the detailed proof here.
Let $\r_{x}$ denote the distance function from any considered point $x$ in $\Si$. There is a constant $\ep_*\in(0,1]$ depending only on $n$ and the geometry of $B_2(p)$ such that $\r_{x}$ is smooth on $B_{\ep_*}(x)\setminus\{x\}$ for any $x\in\overline{B_1(p)}$. For any $0<\ep<\ep_*$,
let $\{x_k\}_{k=1}^{n_\ep}\subset\overline{B_1(p)}$ be a finite collection of points such that $\overline{B_1(p)}\subset\cup_{k=1}^{n_\ep}B_{\ep/4}(x_{k})$. With Besicovitch covering lemma, we can assume $\sharp\{k\in\{1,\cdots,n_{\ep}\}|\, x\in B_{\ep/4}(x_{k})\}\le c_n$ for any $x\in\cup_{k=1}^{n_\ep}B_{\ep/4}(x_{k})$, where $c_n\ge1$ is a constant depending only on $n$. 

Now we adopt the idea from the proof of partition of unit.
Let $\sigma$ be a symmetric smooth function on $\R$ with support in $(-1,1)$, $0\le\si\le1$ on $\R$ and $\si\equiv1$ on $[-1/2,1/2]$. Up to choosing the constant $c_n$, we assume $|\si'|\le c_n$ on $\R$.
Let $\si_{k,\ep}(y)=\si(2\r_{x_{k}}(y)/\ep)$ and
$$\phi_{k,\ep}(y)=\left(\sum_{l=1}^{n_{\ep}}\si_{l,\ep}(y)\right)^{-1}\si_{k,\ep}(y)\qquad \mathrm{for\ any}\ y\in \Si.$$
Then spt$\si_{k,\ep}\subset B_{\ep/2}(x_{k})$, $1\le\sum_{l=1}^{n_{\ep}}\si_{l,\ep}\le c_n$, and $\phi_{k,\ep}$ is smooth on $\overline{B_1(p)}$.
Let $\mathcal{I}_{\ep}$ be a discrete set defined by
 $$\mathcal{I}_{\ep}=\left\{k\in\{1,\cdots,n_{\ep}\}|\, B_{\ep/2}(x_{k})\cap \mathcal{C}_{p}\neq\emptyset\right\},$$
and $\phi_{\ep}$ be a smooth (cut-off) function with support in $B_\ep(\mathcal{C}_{p})$ defined by
\begin{equation}\aligned
\phi_{\ep}=\sum_{k\in\mathcal{I}_{\ep}}\phi_{k,\ep}.
\endaligned
\end{equation}
By the definition of $\mathcal{I}_{\ep}$, $\phi_\ep=1$ on $\mathcal{C}_{p}\cap\overline{B_1(p)}$.
In particular, there is a small constant $\ep'\in(0,\ep)$ so that $\phi_{\ep}=1$ on $B_{\ep'}\left(\mathcal{C}_{p}\cap\overline{B_1(p)}\right)$.

Put $\si_{t}(s)=\left(\int_{\R^n}\si(|z|^2)dz\right)^{-1}t^{-n}\si(t^{-2}s^2)$ and $\xi_{t}(x,y)=\si_{t}(d(x,y))$ for all $x,y\in \Si$ and $t>0$. Then by the definition of $\si$
\begin{equation}\aligned\label{sits=1}
n\omega_n\int_0^t s^{n-1}\si_{t}(s)ds=n\omega_n\int_0^\infty s^{n-1}\si_{t}(s)ds=\int_{\R^n}\si_t(|z|)dz=1.
\endaligned
\end{equation}
For small $t>0$, let $\e_{\ep,t}$ be a convolution of $\r_{p}\phi_{\ep}$ and $\si_{t}$ defined by
\begin{equation}\aligned\label{wep}
\e_{\ep,t}(x)=((\r_{p}\phi_{\ep})*\xi_{t})(x)=\int_{y\in\Si} \r_{p}(y)\phi_{\ep}(y)\xi_{t}(x,y).
\endaligned
\end{equation}
Then $\e_{\ep,t}$ is a smooth function on $\overline{B_1(p)}$ for any $t\in(0,\ep_*)$ by the definition of $\ep_*$.
Let $ds^2+g_{\a\be,x}(s,\vartheta)d\vartheta_\a d\vartheta_\be$ denote the metric of $B_t(x)$ in the polar coordinate w.r.t. $x$, where $\vartheta=(\vartheta_1,\cdots,\vartheta_n)\in\mathbb{S}^{n-1}$ satisfies $|\vartheta|=1$.
Let $\mathrm{exp}_x(s\vartheta)$ denote the exponential map in $\Si$ starting from $x$, and $J_x(s,\vartheta)=\sqrt{\det(g_{\a\be,x}(s,\vartheta))}$. Let $D$ denote Levi-Civita connection of $\Si$. Then
\begin{equation}\aligned\label{Dexpx}
\lim_{t\to0}\sup_{\vartheta\in\mathbb{S}^{n-1},0<s<t}|D\mathrm{exp}_x(s\vartheta)|=1,
\endaligned
\end{equation}
and
\begin{equation}\aligned\label{DJx}
\lim_{t\to0}\sup_{\vartheta\in\mathbb{S}^{n-1},0<s<t}|DJ_x(s,\vartheta)|=0.
\endaligned
\end{equation}
From \eqref{wep},
\begin{equation}\aligned
\e_{\ep,t}(x)=\int_0^t\int_{\mathbb{S}^{n-1}} \r_{p}(\mathrm{exp}_x(s\vartheta))\phi_{\ep}(\mathrm{exp}_x(s\vartheta))\si_t(s)J_x(s,\vartheta)d\vartheta ds.
\endaligned
\end{equation}
For any $x\in B_{\ep'/2}\left(\mathcal{C}_{p}\cap\overline{B_1(p)}\right)$, $\phi_{\ep}(\mathrm{exp}_x(s\vartheta))=1$ for any $0<s\le t$ with $t<\ep'/2$.
Then with \eqref{sits=1}\eqref{Dexpx}\eqref{DJx} there is a positive constant $t_*\le\ep'/2$ such that
\begin{equation}\aligned\label{Deept*}
|D\e_{\ep,t_*}|(x)\le\int_0^{t_*}\int_{\mathbb{S}^{n-1}} (1+\ep)\si_{t_*}(s)s^{n-1}d\vartheta ds=(1+\ep)n\omega_n\int_0^{t_*}\si_{t_*}(s)s^{n-1}ds=1+\ep.
\endaligned
\end{equation}
For any $x\in B_\ep(\mathcal{C}_{p})\cap\overline{B_1(p)}\setminus B_{\ep'/2}\left(\mathcal{C}_{p}\cap\overline{B_1(p)}\right)$, and any tangent field $\xi$ on $\Si$ around $x$ with $|\xi|(x)=1$, with \eqref{sits=1}\eqref{Dexpx}\eqref{DJx} we have
\begin{equation}\aligned\label{Dxieet*}
\left|D_\xi\e_{\ep,t_*}(x)-D_\xi(\r_{p}\phi_{\ep})(x)\right|\le\int_0^{t_*}\int_{\mathbb{S}^{n-1}} \ep\si_{t_*}(s)s^{n-1}d\vartheta ds=\ep
\endaligned
\end{equation}
up to a choice of the constant $\ep_*$.

Let $f_\ep$ be a smooth function on $\overline{B_1(p)}$ defined by
\begin{equation}\aligned
f_\ep=\r_p(1-\phi_\ep)+\e_{\ep,t_*}.
\endaligned
\end{equation}
Since supp$\phi_{\ep}\subset B_\ep(\mathcal{C}_{p})$, then $f_\ep=\r_p$ on $\overline{B_1(p)}\setminus B_\ep(\mathcal{C}_p)$. On $B_{\ep'/2}\left(\mathcal{C}_{p}\cap\overline{B_1(p)}\right)$, we have $|Df_\ep|=|D\e_{\ep,t_*}|\le1+\ep$ from $\phi_\ep=1$ and \eqref{Deept*}. On $B_\ep(\mathcal{C}_{p})\cap\overline{B_1(p)}\setminus B_{\ep'/2}\left(\mathcal{C}_{p}\cap\overline{B_1(p)}\right)$, from \eqref{Dxieet*} we have
\begin{equation}\aligned
|Df_\ep|\le|D\r_p|+|D\left(\r_p\phi_\ep-\e_{\ep,t_*}\right)|\le1+\ep.
\endaligned
\end{equation}
In all, we have $|Df_\ep|\le1+\ep$ on $B_\ep(\mathcal{C}_{p})\cap\overline{B_1(p)}$. This completes the proof.
\end{proof}

\bibliographystyle{amsplain}

\begin{thebibliography}{10}

\bibitem{Al} F. Almgren, Optimal isoperimetric inequalities, Bull. Amer. Math. Soc. (N.S.) {\bf13} (1985), no. 2, 123-126.

\bibitem{AK} L. Ambrosio, B. Kirchheim, Currents in metric spaces, Acta Math. {\bf 185(1)} (2000), 1-80.

\bibitem{An1} M. T. Anderson, The $L^2$ structure of moduli spaces of Einstein metrics on 4-manifolds, GAFA., {\bf 2} (1992), 29-89.

\bibitem{Bo} E. Bombieri, , Theory of minimal surfaces and a counter-example to the Bernstein conjecture in high dimensions, Notes of Lectures held at the Courant Institute, New York University, 1970.

\bibitem{BDM} E. Bombieri, E. De Giorgi and M. Miranda, Una maggiorazione a priori relativa alle ipersuperfici minimali non parametriche, Arch. Ration. Mech. Anal. {\bf 32} (1969), 255-267.

\bibitem{BG} E. Bombieri, E. Giusti, Harnack's inequality for elliptic differential equations on minimal surfaces, Invent. Math. {\bf 15} (1972), 24-46.

\bibitem{Br} S. Brendle, The isoperimetric inequality for a minimal submanifold in Euclidean space, J. Amer. Math. Soc. {\bf 34} (2021), 595-603.

\bibitem{Br1} S. Brendle, Sobolev inequalities in manifolds with nonnegative curvature, 	arXiv:2009.13717.

\bibitem{BBI} D. Burago, Y. Burago, S. Ivanov, A Course in Metric Geometry, Graduate Studies in Mathematics, 33. American Mathematical Society, Providence, RI, 2001. xiv+415 pp. ISBN: 0-8218-2129-6.

\bibitem{Bu} P. Buser, A note on the isoperimetric constant, Ann. Scient. Ec. Norm. Sup. {\bf 15} (1982), 213-230.

\bibitem{CHH} Jean-Baptiste Casteras, Esko Heinonen, Ilkka Holopainen, Existence and non-existence of minimal graphic and $p$-harmonic functions, Proc. Roy. Soc. Edinburgh Sect. A {\bf 150} (2020), no. 1, 341-366.

\bibitem{CC} Jeff Cheeger and Tobias H. Colding, Lower Bounds on Ricci Curvature and the Almost Rigidity of Warped Products, Ann. Math. {\bf 144} (1996), 189-237.

\bibitem{CCo1} Jeff Cheeger, Tobias H. Colding, On the structure of spaces with Ricci curvature bounded below. I, J. Differential Geom. {\bf 46} (1997),  no. 3, 406-480.

\bibitem{CCo2} Jeff Cheeger, Tobias H. Colding, On the structure of spaces with Ricci curvature bounded below. II, J. Differential Geom. {\bf 54} (2000),  no. 1, 15-35.

\bibitem{CCo3} Jeff Cheeger, Tobias H. Colding, On the structure of spaces with Ricci curvature bounded below. III, J. Differential Geom. {\bf 54} (2000),  no. 1, 37-74.

\bibitem{CCM} Jeff Cheeger, Tobias H. Colding, William P. Minicozzi II, Linear growth harmonic functions on complete manifolds with nonnegative Ricci curvature, Geom. Funct. Anal. {\bf 5} (1995), no. 6, 948-954.

\bibitem{CN} Jeff Cheeger, A. Naber, Lower bounds on Ricci curvature and quantitative behavior of singular sets, Invent. Math. 191 (2013), 321-339.

\bibitem{C1} Tobias H. Colding, Shape of manifolds with positive Ricci curvature, Invent. math. {\bf 124} (1996), 175-191.

\bibitem{C} Tobias H. Colding, Ricci curvature and volume convergence, Ann. Math. {\bf 145} (1997), 477-501.

\bibitem{C2} Tobias H. Colding, New monotonicity formulas for Ricci curvature and applications; I, Acta Math. {\bf 209} (2012), no. 2, 229-263.

\bibitem{CM1} Tobias H. Colding and William P. Minicozzi II, Large scale Behavior of Kernels of Schr$\mathrm{\ddot{o}}$dinger Operators, Amer. J. Math. {\bf 119}, no. 6 (1997), 1355-1398.

\bibitem{CN1} Tobias H. Colding, A. Naber, Lower Ricci curvature, branching and the bilipschitz structure of uniform Reifenberg spaces, Adv. Math., {\bf 249} (2013), 348-358.

\bibitem{CMMR} Giulio Colombo, Marco Magliaro, Luciano Mari, Marco Rigoli, Bernstein and half-space properties for minimal graphs under Ricci lower bounds, arXiv:1911.12054.

\bibitem{Cr} C. Croke, Some isoperimetric inequalities and eigenvalue estimates, Ann. Scient. \'ec. Norm. Sup. 4, T {\bf 13} (1980), 419-435.

\bibitem{Dg} E. De Giorgi, E., Sulla differentiabilit\'a e l'analiticit\'a delle estremali degli integrali multipli regolari, Mem. Accad. Sci. Torino, s. III, parte I, (1957), 25-43.

\bibitem{D0} Qi Ding, Liouville-type theorems for minimal graphs over manifolds, Analysis $\&$ PDE {\bf 14}(6) (2021), 1925-1949.

\bibitem{D} Qi Ding, Area-minimizing hypersurfaces in manifolds of Ricci curvature bounded below, arXiv:2107.11074.

\bibitem{D1} Qi Ding, Minimal hypersurfaces in manifolds of Ricci curvature bounded below, J. Reine. Angew. Math. {\bf 791} (2022), 247-282..

\bibitem{DJX1} Qi Ding, J.Jost and Y.L.Xin, Existence and non-existence of area-minimizing hypersurfaces in manifolds of non-negative Ricci curvature, Amer. J. Math., {\bf 138} (2016), no.2., 287-327.

\bibitem{DJX2} Qi Ding, J.Jost and Y.L.Xin, Minimal graphic functions on manifolds of non-negative Ricci curvature, Comm. Pure Appl. Math. {\bf 69} (2016), no. 2, 323-371.

\bibitem{FF} H. Federer, W. Fleming, Normal and integral currents, Annals of Math. {\bf 72} (1960), 458-520.

\bibitem{Fi} R. Finn, On equations of minimal surface type, Ann. of Math., {\bf 60(2)} (1954), 397-416.

\bibitem{Fi1} Robert Finn, New estimates for equations of minimal surface type, Arch. Rational Mech. Anal. {\bf14}
(1963), 337-375.

\bibitem{GT} D. Gilbarg and N. Trudinger, Elliptic Partial Differential Equations of Second Order, Springer-Verlag, Berlin-New York, (1983).

\bibitem{Gi} E. Giusti, Minimal surfaces and functions of bounded variation, Birkh$\mathrm{\ddot{a}}$user Boston, Inc., 1984.

\bibitem{Gri} Alexander Grigor'yan, Estimates of heat kernels on Riemannian manifolds, Spectral theory and geometry (Edinburgh, 1998), 140-225, London Math. Soc. Lecture Note Ser., 273, Cambridge Univ. Press, Cambridge, 1999.

\bibitem{G} M. Gromov, Metric Structures for Riemannian and Non-Riemannian Spaces. Birkh$\mathrm{\ddot{a}}$ser Boston, Boston (2007). With appendices by M. Katz, P. Pansu and S. Semmes.

\bibitem{GLP} M. Gromov, J. Lafontaine and P. Pansu, Structures metriques pour les variete riemannienes, Cedic-Fernand Nathan, Paris (1981).

\bibitem{HK} P. Hajlasz, and P. Koskela, Sobolev meets Poincar\'e, C. R. Aead. Sci. Paris S\'er. I Math., {\bf 320} (1995), 1211-1215.

\bibitem{HL} Qing Han and Fanghua Lin, Elliptic partial differential equations, Courant Institute of Mathematical Sciences, New York University, 1997.

\bibitem{HS} David Hoffman and Joel Spruck, Sobolev and isoperimetric inequalities for Riemannian submanifolds, Comm. Pure Appl. Math. {\bf 27} (1974), 715-727.

\bibitem{Honda} Shouhei Honda, Harmonic functions on asymptotic cones with Euclidean volume growth, J. Math. Soc. Japan {\bf 67(1)} (2015), 69-126.

\bibitem{Ho} Julio Cesar Correa Hoyos, Poincar\'e and Sobolev type inequalities for intrinsic rectifiable varifolds, arXiv:2001.09256.

\bibitem{Pl} Peter Li, Lecture notes on geometric analysis, University of California, Irvine, 1996, preprint.

\bibitem{LW} Peter Li, Jiaping Wang, Mean value inequalities, Indiana Univ. Math. J. {\bf 48} (1999), no. 4, 1257-1283.

\bibitem{LY} P. Li, S.T. Yau, On the parabolic kernel of the Schr$\mathrm{\ddot{o}}$dinger operator, Acta Math. {\bf 156} (1986), 153-201.

\bibitem{LYa} F.H. Lin, X.P. Yang, Geometric measure theory: an introduction, Science Press, Beijing/ New York; International Press, Boston, 2002.

\bibitem{LWW} L. Liu, G. Wang, L. Weng, The relative isoperimetric inequality for minimal submanifolds in the Euclidean space, arXiv:2002.00914, 2020.

\bibitem{MS} J. Michael and L.M. Simon, Sobolev and mean-vaule inequalities on generalized submanifolds of $\R^n$, Comm. Pure Appl. Math., {\bf 26} (1973), 361-379.

\bibitem{Mm} M. Miranda, Disuguaglianze di Sobolev sulle ipersuperfici minimali, Rend. Sem. Mat. Univ. Padova, {\bf 38}, 1967.

\bibitem{MSY} Ngaiming Mok, Yum-Tong Siu and Shing-Tung Yau, The Poincar\'e-Lelong equation on complete K$\mathrm{\ddot{a}}$hler manifolds, Compositio Mathematica {\bf 44} (1981), 183-218.

\bibitem{M} J$\mathrm{\ddot{u}}$rgen Moser, On Harnack's Theorem for Elliptic Differential Equations, Comm. Pure Appl. Math. {\bf 14} (1961), 577-591.

\bibitem{P} P. Peterson, Riemannian Geometry, Graduate Texts in Mathematics {\bf 171}, Springer Science+Business Media, LLC, 2006.

\bibitem{RSS} Harold Rosenberg, Felix Schulze and Joel Spruck, The half-space property and entire positive minimal graphs in $M\times\R$, J. Diff. Geom. {\bf 95} (2013), 321-336.

\bibitem{SY} R. Schoen and S.T. Yau, Differential Geometry, Science Press, Beijing, 1988.

\bibitem{Sc} Christian Scharrer, Some geometric inequalities for varifolds on Riemannian manifolds based on monotonicity identities, arXiv:2105.13211.

\bibitem{Si0} Leon Simon, Thesis, University of Adelaide, 1971.

\bibitem{S} Leon Simon, Lectures on Geometric Measure Theory, Proceedings of the center for mathematical analysis Australian national university, Vol. 3, 1983.

\bibitem{Sp} Joel Spruck, Interior gradient estimates and existence theorems for constant mean curvature graphs in $M^n\times\R$, Pure Appl. Math. Q. {\bf 3} (2007), no. 3, Special Issue: In honor of Leon Simon. Part 2, 785-800.

\bibitem{W} Xu-Jia Wang, Interior gradient estimates for mean curvature equations, Mathematische Zeitschrift, {\bf 228} (1998), 73-81.

\bibitem{Wn} N. Wickramasekera, A sharp strong maximum principle and a sharp unique continuation theorem for singular minimal hypersurfaces, Calc. Var. Partial Differential Equations {\bf 51} (2014), no. 3-4, 799-812.


\end{thebibliography}

\end{document}